\documentclass[11pt]{article}

\usepackage{graphicx}
\usepackage{latexsym}
\usepackage{amssymb}
\usepackage{amsmath}
\usepackage{enumerate}

\usepackage{amsmath}
\usepackage{stmaryrd}
\usepackage{hyperref}
\usepackage{amssymb}
\usepackage{CJK}
\usepackage{color}

\usepackage{amsmath,latexsym,amssymb,amsthm,array,amsfonts}

\usepackage{mathrsfs,dsfont,bbm}
\usepackage{csquotes}
\usepackage{layout}
\usepackage{layout}
\newtheorem{prop}{Proposition}
\newtheorem{lemma}{Lemma}
\newtheorem{definition}{Definition}

\newtheorem{theorem}{Theorem}
\newtheorem{proposition}{Proposition}
\newtheorem{remark}{Remark}

\newtheorem{conjecture}{Conjecture}

\def\real{{\mathord{{\rm I\kern-2.8pt R}}}}        
\def\inte{{\mathord{{\rm I\kern-2.8pt N}}}}

\def\sZZ{{\rm Z\kern-2.8ptem{}Z}}

\def\z{{\mathchoice
  {\sZZ}
  {\sZZ}
  {\rm Z\kern-0.30em{}Z}
  {\rm Z\kern-0.25em{}Z} }}
\def\sQQ{{\kern 0.27em \vrule height1.45ex width0.03em depth0em
          \kern-0.30em \rm Q}}
\def\qu{{\mathchoice
    {\sQQ}
    {\sQQ}
  {\kern 0.225em \vrule height1.05ex width0.025em depth0em \kern-0.25em \rm Q}
  {\kern 0.180em \vrule height0.78ex width0.020em depth0em \kern-0.20em \rm Q}
        }}
\def\sCC{{\kern 0.27em \vrule height1.45ex width0.03em depth0em
          \kern-0.30em \rm C}}
\def\complex{{\mathchoice
    {\sCC}
    {\sCC}
  {\kern 0.225em \vrule height1.05ex width0.025em depth0em \kern-0.25em \rm C}
  {\kern 0.180em \vrule height0.78ex width0.020em depth0em \kern-0.20em \rm C}
        }}

\MakeOuterQuote{"}

\newcommand{\ba}{\begin{array}}
\newcommand{\ea}{\end{array}}
\newcommand{\be}{\begin{equation}}
\newcommand{\ee}{\end{equation}}
\newcommand{\bea}{\begin{eqnarray}}
\newcommand{\eea}{\end{eqnarray}}
\newcommand{\beaa}{\begin{eqnarray*}}
\newcommand{\eeaa}{\end{eqnarray*}}

\def\z{\zeta}

\font\tenmath=msbm10 \font\sevenmath=msbm7 \font\fivemath=msbm5
\newfam\mathfam \textfont\mathfam=\tenmath
\scriptfont\mathfam=\sevenmath \scriptscriptfont\mathfam=\fivemath

\def \={{\buildrel {\rm (law)} \over =}}

\def\qed{ \hfill \vrule width.25cm height.25cm depth0cm\smallskip}

\newcommand{\basa}{\begin{assumption}}
\newcommand{\easa}{\end{assumption}}

\newcommand{\bas}{\begin{assum}}
\newcommand{\eas}{\end{assum}}

\newcommand{\N}{\mathbb{N}}
\newcommand{\E}{\mathbb{E}}

\newcommand{\ignore}[1]{}
\usepackage{authblk}
\begin{document}

\renewcommand{\thefootnote}{\fnsymbol{footnote}}

\renewcommand{\thefootnote}{\fnsymbol{footnote}}

\title{Free Malliavin-Stein-Dirichlet method: multidimensional semicircular approximations and chaos of a quantum Markov operator}

\author[1,2]{Charles-Philippe Diez \thanks{charles-philippe.diez@univ-lille.fr}}
\affil[1]{CNRS, Universit\'e de Lille,
Laboratoire Paul Painlev\'{e}.}
\affil[2]{Department of Statistics, The Chinese University of Hong-Kong.}

\renewcommand\Authands{ and }

\maketitle

\begin{abstract}
We combine the notion of free Stein kernel and the free Malliavin calculus to provide quantitative bounds under the free (quadratic) Wasserstein distance in the multivariate semicircular approximations for self-adjoint vector-valued multiple Wigner integrals. On the way, we deduce an HSI inequality for a modified non-microstates free entropy with respect to the potential associated with these semicircular families in the case of non-degeneracy of the covariance matrix. The strategy of the proofs is based on functional inequalities involving the free Stein discrepancy. We obtain a bound which depends on the second and fourth free cumulant of each component. We then apply these results to some examples such as the convergence of marginals in the free functional Breuer-Major CLT for the non commutative fractional Brownian motion, and we provide a bound for the free Stein discrepancy with respect to semicircular potentials for {\it q-semicirculars operators}. Lastly, we develop an abstract setting on where it is possible to construct a free Stein Kernel with respect to the semicircular potential: the quantum chaos associated to a quantum Markov semigroup whose $L^2$ generator $\Delta$ can be written as the square of a real closable derivation $\delta$ valued into the square integrable bi-processes or into a direct sum of the coarse correspondence.
\end{abstract}

\vskip0.3cm

{\bf 2010 AMS Classification Numbers:}   46L54, 60H07, 60H30.

\vskip0.3cm


\section{Introduction}
Stein's method, invented by Charles Stein in 1972 \cite{Stein}, is a powerful tool to prove the central limit theorem and to obtain bounds for distances between probability measures. In the classical case, a lot has been discovered on this topic using various techniques such as exchangeable pairs (Nathan Ross in \cite{NR}) or Malliavin calculus (see the monograph of Nourdin and Peccati \cite{NP-book} for a complete exposition). The rate of convergence under numerous distances and conditions of convergence in the univariate and multidimensional cases to reference measures are well known (see e.g.  Kusuoka and Tudor \cite {ST} and a constantly updated webpage maintained by Nourdin \url{https://sites.google.com/site/malliavinstein/home}, for a complete list of papers related to {\it fourth moment theorems}). Transportation cost inequalities and functional inequalities between quantities such as Wasserstein distance, entropy or Fisher information were proved using various branches of Mathematics: PDE techniques (Otto and Villani \cite{VI}), semigroup approach or Gamma-calculus (see for example the monograph of Bakry, Gentil Ledoux \cite{BGL} to have a complete exposure of the theory). These results have been very useful in many applications, for example, in the topic of concentration of measure (e.g Gozlan \cite{G}). It is still a topic of interest for many researchers in the commutative and free case, which try to improve these inequalities. One can mention the remarkable and recent paper \cite{LNP} of Ledoux, Nourdin and Peccati (2015) which improves respectively the log-Sobolev and Talagrand transportation cost inequalities by two new inequalities called {\it "HSI" and "WSH"}, which links a new quantity called the Stein discrepancy.
\bigbreak
Motivated by deep problems related to von Neumann algebras (with a particular interest for the von Neumann algebras generated by free groups) such as the existence of prime property or Cartan subalgebras for $\Pi_1$ factors (the precises definitions could be find in general textbooks of von Neumann algebras, see for example the book of Sinclair \cite{sain}), Voiculescu, in a series of breakthrough papers \cite{Voic1}, \cite{Voic2}, has developed powerful techniques to have a deeper understanding of their structure via the idea of free probability. In fact, he discovered an analog of classical independence called free independence, he introduced the notion of free convolution which led to a free central limit theorem and proved that the large $N$ limit of $N\times N$ Gaussian random matrices behave as semicircular systems. Voiculescu in \cite{V} was also able to define a free analogue of information theory: free Fisher information, free entropy (microstates and non-microstates), non-commutative Hilbert transform... Voiculescu has shown they behave as well (for example, change of variable for microstates entropy) as in the commutative case and are very useful to prove results for the von Neumann algebras $W^*(X_1,\ldots,X_n)$ generated by $X_1,\ldots,X_n$ elements of a finite von Neumann algebra equipped with a faithful normal tracial state. For example, under finite microstates entropy, these von Neumann algebras does not have Cartan subalgebras and are primes (Voiculescu \cite{Voic2}, Ge \cite{Ge}).
\bigbreak
Unfortunately, two notions of entropy have appeared and, until recently it has remain an open problem to prove or disprove that the two quantities are equal (the Connes embedding conjecture which have been claimed to hold false, by a recent result of Zhengfeng Ji, Anand Natarajan, Thomas Vidick, John Wright, Henry Yuen
\cite{CEC}, and which states that any separable $II_1$ factor embeds into an ultrapower of the hyperfinite factor $R$, which would imply that the two entropies cannot agree in full generality). However, one has an important inequality due to Biane, Capitaine and Guionnet which shows by that the non-microstate entropy is greater than or equal to the microstate entropy (see \cite{BGC}). Recently, Dabrowski in \cite{Dab16} was able to show the equality between the two entropies for tuples satisfying a Schwinger-Dyson equation with a subquadratic bounded from below strictly convex potentials with Lipschitz derivative sufficiently approximable by non-commutative polynomials by means of Malliavin calculus and in particular the Boué-Dupuis formula. In light of these results, one could also point out that the associated free energy/relative entropy appears as the large deviations rate function for the empirical spectrum of large random matrices. It also detects freeness: a tuple of non-commutative random variables are free if and only if the entropy of the tuple is the sum of entropy of each one.
Numerous inequalities which hold true in the commutative case have been proved to hold true in the context of free probability : {\it Free Stam} inequality, {\it Cramer-Rao bound}, {\it log-Sobolev} inequality with respect to the semicircular potential, due to Voiculescu \cite{V}, Biane and Speicher in \cite{BS} in the one dimensional case and more recently the free "{\it HSI}" inequality proved by Fathi and Nelson for semicircular potential with homothetic covariance in \cite {FN} (where "H" stands for the free entropy, "S" for the free Stein discrepancy and "I" for the free Fisher information). We will focus on the last one, whose proof relies on the study of free Stein kernels. In fact, surprisingly and contrary to the commutative case where their existence is not always ensured,  Cébron, Fathi and Mai showed that the existence of free Stein kernels relative to a potential is always ensured provided a moment condition relative to the cyclic derivative potential is fulfilled (see \cite{FCM} for details and construction by two different methods).

\bigbreak
Considering analogies between Wigner and Wiener chaos, Kemp, Nourdin, Peccati and Speicher in their remarkable work \cite{KNPS}, have studied the convergence of self-adjoint multiple Wigner integral (living in a fixed homogeneous Wigner chaos) toward a free $(0,1)$ semicircular variable and were able to provide a free probabilistic analog of the fourth moment theorem. They also obtained a quantitative bound under a distance $d_{C_2}$ (defined over smooth functions such that the non-commutative derivative of the derivative is bounded by 1) for second order 
multiple Wigner integrals with "mirror-symmetric" kernel by means of free stochastic analysis and especially free Malliavin Calculus introduced by Biane, Speicher in \cite{BS}. These objects are respectively the analogue in the context of free probability of the well-known Wiener chaos and Gaussian random variable. In 2017, Bourguin and Campese \cite{BC} strongly improved the result, by using a new product formula for bi-integrals, and obtained a more general bound for the distance $d_{C_2}$ and fully-symmetric multiple Wigner integral of any order, with a constant which grows linearly with the order of the chaos. In 2018, Cébron \cite{C} has extended the results by proving that the free Stein discrepancy is bounded by the fourth free cumulant by constructing a new free Stein kernel (with respect to the potential associated with a free $(0,1)$ semicircular variable) for all (centered) self-adjoint elements in Wigner chaos. In particular, in his proofs, the fully-symmetry assumption is no longer required. He was also able to get a quantitative bound for the Wasserstein distance (introduced by Biane and Voiculescu in \cite{BV}) between a self-adjoint multiple Wigner integral with and a free $(0,1)$ semicircular variable which involves again the fourth free cumulant to the pow $\frac{1}{4}$ instead of the usual square root with a constant dependent on the order of chaos to the pow $\frac{3}{4}$. He has also shown that the distance $d_{C_2}$ is weaker than the quadratic Wasserstein distance (the non commutative and the commutative notions coincide in dimension one by the result of Biane and Voiculescu \cite{BV}).
\bigbreak
Following the ideas of Cébron in \cite{C}. We extend here the results to the multivariate semicircular approximation for semicircular family with covariance $C$ (a symmetric positive definite matrix) by constructing a free Stein kernel with respect to their associated potential for all self-adjoint tuple belonging to some (non necessarily homogeneous) finite Wigner chaos, and we also obtain a quantitative estimation for the free quadratic Wasserstein distance between a tuple of multiple Wigner integrals and a semicircular family with a non degenerate covariance matrix which involves the second and fourth free cumulants of each components, which will allow us to prove in an easier way the theorem 1.3 of Nourdin, Peccati and Speicher \cite{NPS}. This last result is in fact the free analog of the famous multivariate fourth moment theorem on Wiener chaos which was proved by Peccati and Tudor in \cite{PT}, and where several years later the powerful tools of Malliavin-Stein method developed in the multidimensional case by Nourdin, Peccati and Réveillac in \cite{NPR} has allowed to obtain quantitative bounds for usual $1$-Wasserstein distance between functionals of a centered Gaussian isonormal process and a centered Gaussian vector with strictly positive covariance matrix. In the construction, we especially see the importance of the Ornstein-Uhlenbeck operator (the infinite dimensional Laplacian on the Wiener space) and its pseudo-inverse, the key lemma to obtain this theorem being a multivariate counterpart of the Stein's identity for multivariate Gaussian vectors.
Contrary to the classical case, our strategy is based on functional inequalities, especially on the links between the free quadratic Wasserstein distance and free Stein discrepancy with respect to strictly convex semicircular potentials. A starting point of our investigations is related to the notion of conjugate variables with respect to a potential, which will be specified later in the paper. An inequality between the free Fisher information associated with these potential along the flow of a free stochastic differential equation and the associated free Stein Discrepancy with our target will be provided. The methods to find these bounds are generalized ideas of breakthrough papers, such as: Dabrowski in \cite{Dab10} for the notion of non-commutative path space to construct stationary solutions of free stochastic differential equations which lead to functional inequalities, and on the other side, a construction of a new free Stein kernel for all self-adjoint elements in some homogeneous Wigner chaos by Cébron in \cite{C} which allows relaxing the fully-symmetry assumption (which was a necessary condition for the free Stein kernel constructed in \cite{KNPS} or \cite{BC}, remark 3.9). In contrast to the Gaussian case, we have to avoid the use of the free Ornstein-Uhlenbeck operator (and its inverse), and use better the properties of operators of free Malliavin and Ito calculus.
\bigbreak
Lastly, we will develop an abstract setting build on quantum Markov semigroup (QMS) on where it might be plausible to derive quantitative fourth moment theorems. In particular, we will give a possible notion of chaotic random variables in this context. That is, we begin with a real closable derivation valued into the square integrable biprocesses or valued in a direct sum of the coarse correspondence. We construct via the associated carré du champ, a free Stein kernel with respect to the standard semicircular potential (we can also extend the results to more general free Gibbs state). Note that, this particular assumption over the derivation is necessary as the Schwinger-Dyson equation which characterize a semicircular family involved free difference quotient which are valued in this coarse correspondence. Moreover, one cannot expect to have convergence toward the semicircular law for a general setting since there are uncountably many non-isomorphic $II_1$ separable factors (Mc Duff \cite{MC}). In particular, it will be of interest to show that for the ones we have constructed, under suitable assumptions, possibly on the magnitude of the eigenvalues (discrete spectrum) of the corresponding generator of the completely Dirichlet form, or existence of non-amenability set, (see Dabrowski \cite{DAB} v1, Dabrowski and Ioana \cite{DI}) the underlying von Neumman algebra behaves as the free groups factors $L(\mathbb{F}_n),1\leq n\leq \infty$: non-Gamma property, strong solidity, absence of Cartan subalgebras, primeness or even non $L^2$-rigidity results (see Peterson \cite{Pete}).
\newpage
\section{Definitions and notations}
\begin{flushleft}
Let us denote $\mathcal{M}$ a von Neumann algebra equipped with $\tau$ a faithful normal state.
\newline
Let $\mathds{P} =\mathbb{C}\langle t_1,...,t_n \rangle$ be the algebra of non-commutative polynomials in $n$ variables $t_1,...,t_n$.
\end{flushleft}
\begin{definition}
A free Stein kernel for a $n$-tuple $X$ with respect to a potential $V\in \mathds{P}$ is an element of $L^2(M_n(\mathcal{M}\bar{\otimes} \mathcal{M}^{op}),(\tau \otimes \tau^{op}) \circ Tr)$ such that for any $P \in \mathds{P} ^n$:
\begin{equation}
    \langle [DV](X),P(X)\rangle_{\tau}=\langle A,[\mathcal{J}P](X)\rangle_{\tau \otimes \tau^{op}}
\end{equation}
\end{definition}
\begin{flushleft}
The Stein discrepancy of $X$ relative to $V$ is then defined as :
\end{flushleft}
\begin{equation}
\Sigma ^*(X|V)=\inf_{A}\lVert A-(1\otimes 1^{op})\otimes I_n\rVert_{L^2(M_n(\mathcal{M}\bar{\otimes} \mathcal{M}^{op}),(\tau \otimes \tau^{op}) \circ Tr)}
\end{equation}
where the infinimum is taken over all admissible Stein kernel $A$ of $X$ relative to $V$.

\begin{flushleft}
Here $DV$ is a cyclic gradient and $\mathcal{J}P$ is the Jacobian matrix of $P$ which will be defined in the forthcoming pages.
\end{flushleft}
\begin{flushleft}
Recall by the GNS construction, $\tau$ defines an inner product on $\mathcal{M}$ by setting for all $x,y \in \mathcal{M}$.
\end{flushleft}
\begin{equation}
    \langle x,y\rangle_{\tau}=\tau (x^*y)\nonumber
\end{equation}
\begin{flushleft}
The completion of $M$ with respect to the induced norm
$\lVert.\rVert_{\tau}$ is denoted $L^2(\mathcal{M},\tau)$. We will omit to denote the state when its clearly defined and denote $\lVert.\rVert_{\tau}$  as $\lVert.\rVert_{2}$ and $L^2(\mathcal{M},\tau)$ as $L^2(\mathcal{M})$. We can also define in the same way the spaces $L^p(\mathcal{M},\tau)$ for $1\leq p\leq\infty $ by taking the completion with respect to the norm :
\end{flushleft}
\begin{equation}
    \lVert x\rVert_p=\tau(\lvert x\rvert^p)^{\frac{1}{p}}
\end{equation}
where $\lvert x\rvert=(x^*x)^{\frac{1}{2}}$
and $L^{\infty}(\mathcal{M},\tau):=\mathcal{M}$ equipped with the operator norm $\rVert .\lVert$.
\begin{flushleft}
One also have (more generally for every unital $C^*$ algebra) for a faithful state $\tau$ and $x\in M$:
\begin{equation}
    \lVert x\rVert=\lim_{n\rightarrow \infty}\tau( (x^*x)^{n})^{\frac{1}{2n}}
\end{equation}
\end{flushleft}
\begin{flushleft}
From the von Neumann tensor product $\mathcal{M}\bar{\otimes} \mathcal{M}^{op}$ equipped with operator norm : $\lVert .\rVert_{\mathcal{M}\bar{\otimes} \mathcal{M}^{op}}$, and the faithful normal state $\tau\otimes \tau^{op}$, we can consider the Hilbert space $L^2(\mathcal{M}\bar{\otimes} \mathcal{M}^{op},\tau\otimes \tau^{op})$. This space can be identified with $HS(L^2(\mathcal{M}))$ which is the space of Hilbert–Schmidt operators on $L^2(M)$ via the following map.
\end{flushleft}
\begin{equation}\label{HS}
    x\otimes y \mapsto \langle y,.\rangle_2 x\nonumber
\end{equation}
For a $n$-tuple $X$ we define $\lVert X\rVert=\max_j \lVert x_j\rVert$.
\begin{flushleft}
We will also write $C^{*}(X)$ and $W^{*}(X)$ for the $C^*$-algebra and von Neumann algebra generated by $X=(x_1, . . . , x_n)$.
\end{flushleft}
\begin{flushleft}
Let $\mathcal{A}_1,..,\mathcal{A}_n$ be von Neumann subalgebras of $\mathcal{M}$. These subalgebras are called $free$ if for all $n\in \mathbb{N}$ and all indices $i_1\neq i_2\neq\ldots\neq i_n$ such as $\tau(A_{j})=0$ and $A_{j}\in\mathcal{A}_{i_j}$, then $\tau(A_{1}\ldots A_{n})=0$.
\end{flushleft}
\begin{flushleft}
We will say that a family $a_1,\ldots a_n\in \mathcal{A}$ are free from another family $b_1,...,b_n$ if $W^*(a_1,\ldots, a_n)$ and $W^*(b_1,\ldots,b_n)$ are free.
\end{flushleft}
\begin{flushleft}
Given another $n$-tuple $Y = (y_1,...,y_n) \in \mathcal{M}^n$, we write
$\langle X,Y\rangle_{\tau} =\sum_{j=1}^n\langle x_j,y_j\rangle_{\tau}$ and when $\tau$ is cleared we will denote it just by $\langle X,Y\rangle_{2}$.
\end{flushleft}

\begin{remark}\label{adjo}
In the sequel, we will use linear transformation of an element in $\mathcal{M}^n$: for a vector $X=(x_1,\ldots,x_n)\in \mathcal{M}^n$ by $C$ a matrix in $M_n(\mathbb{C})$, we will denote $CX$ an element of $\mathcal{M}^n$ as:
\begin{equation}
    CX=\left(\sum_{j=1}^nC_{i,j}x_j\right)_{i=1}^n
\end{equation}
We will be mainly concerned with $C$ a real symmetric positive definite matrix which allows us to deduce the following but useful result (also true in the case where $C$ is hermitian).
\begin{equation}
\langle CX,Y\rangle_2=\langle X,CY\rangle_2
\end{equation}
\end{remark}
We easily notice that $M_n(\mathbb{C})$ embeds in $M_n(\mathbb{C})\otimes \mathcal{M}$ as $M_n(\mathbb{C})\otimes 1$ and thus can be seen as von Neumann subalgebra of $M_n(\mathcal{M})$ equipped with the usual matrix operations (addition, multiplication and involution) and endowed with the operator norm $\lVert .\rVert_{M_n(\mathcal{M})}$.
\bigbreak
If one allows usual matrix multiplications notations, one can also write for $X,Y\in M_{1,n}(\mathcal{M})$ (identified as $\mathcal{M}^n$):
\begin{equation}
\langle CX,Y\rangle_2=(CX)^*Y,
\end{equation}
We recall the following inequality which can be easily deduced from the positivity of $\tau$.
\bigbreak 
For $x,y \in \mathcal{M}$, we have:
\begin{equation}
    \lVert xy\rVert_2\leq \lVert x\rVert\lVert y\rVert_2,
\end{equation}
\bigbreak
We can also deduce for $C\in M_n(\mathbb{C})$ and $X,Y \in {L^2(\mathcal{M}^n)}$:
\begin{equation}
    \lvert\langle CX,Y\rangle_2\rvert \leq \lVert C\rVert_{op}\lVert X\rVert_2\lVert Y\rVert_2,
\end{equation}
We now remind some properties of non-commutative differential calculus. A complete monograph of all the theory and even further results is contained in \cite{MS} of Mai and Speicher.
\begin{definition}[Voiculescu, \cite{Voic1}]
Formally, we define the cyclic derivative on monomials $m\in \mathds{P}$ as :
\begin{equation}
    Dp=(D_1p,...,D_np),\nonumber
\end{equation}
where
\begin{equation}
D_jm=\sum_{m=at_{j}b}ba,\nonumber
\end{equation}
and then we extend linearly to $\mathds{P}$.
\end{definition}
\begin{definition}(Voiculescu section 3 in \cite{V})
The $j$-free difference quotient is defined as:
\begin{equation}
    \partial_jp=\sum_{m=at_jb}a\otimes b^{op},\nonumber
\end{equation}
and then extended linearly to $\mathds{P}$.
\end{definition}
It is now obvious that the free difference quotients and cyclic derivatives are related by the following equations:
\begin{equation}
    D_j=m\circ flip (\partial_j),
\end{equation}
where for $A=a\otimes b,B=c\otimes d \in M\otimes M^{op}$, 
$flip(A)=b\otimes a$ and $m(B)=cd$.
\begin{flushleft}
We will also denote for $A,B \in M_n(\mathcal{M}\bar{\otimes} \mathcal{M}^{op}),(\tau\otimes\tau^{op})\circ Tr)$ (identified with $(\mathcal{M}\otimes \mathcal{M}^{op})\otimes M_n(\mathbb{C}))$, the inner product: \begin{eqnarray}
    \langle A,B\rangle_{(\tau\otimes\tau^{op})\circ Tr}=\sum_{i,j=1}^n \langle [A_{j,k}],[B_{j,k}]\rangle_{\tau\otimes\tau^{op}}\nonumber
    =(\tau\otimes\tau^{op})\circ Tr(A^*.B)
\end{eqnarray}
\end{flushleft}
where "$.$" stands the multiplication in $M_n(\mathcal{M}\bar{\otimes} \mathcal{M}^{op})$. Note that the trace $(\tau\otimes\tau^{op})\circ Tr$ in non-normalized).

\begin{flushleft}
One also can take the completion of $M_n(\mathcal{M}\bar{\otimes} \mathcal{M}^{op})$ with respect to this inner product and denote the completion as $L^2(M_n(\mathcal{M}\bar{\otimes} \mathcal{M}^{op}),(\tau\otimes\tau^{op})\circ Tr)$.
And when the state is clearly fixed, we will denote it simply as $\langle A,B\rangle_{HS}$.
\end{flushleft}
\bigbreak
The space $M_n(\mathcal{M}\bar{\otimes} \mathcal{M}^{op})$ is a von Neumann algebra equipped again with the usual operations (addition, multiplication and involution) and endowed with the operator norm which will be denoted as $\lVert .\rVert_{M_n(\mathcal{M}\bar{\otimes }\mathcal{M}^{op})}$ to avoid confusion. One can obtain the unique norm making $M_n(\mathcal{M}\bar{\otimes} \mathcal{M}^{op})$ a $C^*$-algebra by using the GNS construction $\pi : \mathcal{M}\rightarrow \mathcal{B}(\mathcal{H})$. It implies it particular that the norm is independent of the choice of such representation, since one-to-one $*$-isomorphism between $C^*$-algebras is isometric. Note that in full generality, we don't have a closed formula to compute this operator norm, however we have the following bounds (see e.g \cite{PS}):
\bigbreak
For $T=(T_{i,j})\in M_n(\mathcal{M}\bar{\otimes} \mathcal{M}^{op})$
\begin{equation}
    \lVert T_{i,j}\rVert_{\mathcal{M}\bar{\otimes}\mathcal{M}^{op}}\leq \lVert (T_{i,j})\rVert_{M_n(\mathcal{M}\bar{\otimes} \mathcal{M}^{op})}\leq \sum_{i,j=1}^n\lVert T_{i,j}\rVert_{\mathcal{M}\bar{\otimes} \mathcal{M}^{op}}
\end{equation}
\begin{flushleft}
For an element $C \in M_n(\mathbb{C})$, we will write the element $(1 \otimes 1^{op})\otimes C$ the element of 
$M_n(\mathcal{M}\bar{\otimes}\mathcal{M}^{op})$ (this element can also be identified with $\sum_{i,j=1}^n (C_{i,j}.1\otimes1^{op}) \otimes E_{i,j}$ where $\left(E_{i,j}\right)_{i,j=1}^n$ is the usual basis of $M_n(\mathbb{C})$:
\begin{equation}
    \left((1 \otimes 1^{op}) \otimes C\right)=\left(C_{i,j}.1\otimes 1^{op}\right)_{i,j=1}^n,\nonumber
\end{equation}
\end{flushleft}
It is known that it is very difficult to compute the norm $\lVert .\rVert_{M_n(\mathcal{M}\bar{\otimes}\mathcal{M}^{op}}$, fortunately we can actually prove that for the following type of matrices, we have:
\begin{equation}
    \lVert (1\otimes 1)\otimes C\rVert_{M_n(\mathcal{M}\bar{\otimes} \mathcal{M}^{op})}=\lVert C\rVert_{op},\nonumber
\end{equation}
Where, as usual, $\lVert.\rVert_{op}$ stands for the usual matrix operator norm which is the largest singular value of $C$ : $\sqrt{\rho(CC^*)}$ where $\rho$ is the spectral radius in $M_n(\mathbb{C})$.
It it simply achieved using the cross norm property of the tensor norm in $M_n(\mathbb{C})\otimes (\mathcal{M}\bar{\otimes}\mathcal{M}^{op}).$
\begin{remark}
It is then easily checked for $C,D\in M_n(\mathbb{C})$:
\begin{equation}\label{rem2}
    \left((1\otimes 1^{op})\otimes C\right).\left((1\otimes 1^{op})\otimes D\right)=(1\otimes 1^{op})\otimes CD,\nonumber
\end{equation}
\end{remark}
This also preserve the commutation relations between two matrices and, indeed for $C,D \in M_n(\mathbb{C})$ which commutes, we have:
\begin{equation}
    \left((1\otimes 1^{op})\otimes C\right).\left((1\otimes 1^{op})\otimes D\right)=(1\otimes 1^{op})\otimes CD\nonumber=(1\otimes 1^{op})\otimes DC,\nonumber
\end{equation}
and for $ K \in GL_n(\mathbb{C})$, $(1\otimes1^{op})\otimes K$ is invertible in $M_n(\mathcal{M}\bar{\otimes} \mathcal{M}^{op})$ with inverse: $(1\otimes 1^{op}) \otimes K^{-1}$.
\bigbreak
One also has the inequality for $A,B \in M_n(\mathcal{M}\bar{\otimes}\mathcal{M}^{op})$.
\begin{equation}
    \lVert AB\rVert_{HS}\leq \lVert B\rVert_{M_n(\mathcal{M}\otimes \mathcal{M}^{op}})\lVert A\rVert_{HS},\nonumber
\end{equation}
If the state $\tau$ is also a trace, then we have :
\begin{equation}
    \lVert AB\rVert_{HS}\leq \lVert A\rVert_{op}\lVert B\rVert_{HS},\nonumber
\end{equation}

\begin{definition}
For $P=(p_1,...,p_n)$, we define the non-commutative Jacobian as :
\begin{equation*}
\mathcal{J}P = 
\begin{pmatrix}
\partial_1 p_1 & \partial_2 p_1 & \cdots & \partial_n p_n\\
\vdots  & \vdots  & \ddots & \vdots  \\
\partial_1 p_n & \partial_2 p_n & \cdots & \partial_n p_n
\end{pmatrix} \in M_n(\mathds{P}\otimes \mathds{P}^{op}),\nonumber
\end{equation*}
\end{definition}
Fortunately, the Jacobian enjoys a fundamental chain rule properties (up to natural operations).
In the following, we will be concerned with the  particular case of linear transformation of polynomials.

\begin{flushleft}

Let $t_1,...,t_n$ be non-commutating self-adjoint indeterminates, collected as the $n$-tuple $T = (t_1, . . . , t_n)$. We let $\mathds{P}$ denote the set of polynomials in these indeterminates. For a polynomial $p \in \mathds{P}$ and a monomial $m$, we let $c_{m(p)}\in \mathbb{C}$ which denotes the coefficient of $m$ in $p$. After [18], for each $R > 0$ we define the following norm
\begin{equation}
    \lVert P\rVert_R=\sum_m c_m(p)R^{deg (m)}\nonumber
\end{equation}
where the (finite) sum is running over all monomials appearing in p. We can then take the completion of $\mathds{P}$ with respect to this norm and we will denote the space as $\mathds{P}^{(R)}$, this can be interpreted as the formal power series of radius of convergence at least $R$.
\end{flushleft}
\begin{remark}
We could extend all the previous definitions to the formal power series setting by a simple argument of polynomial approximation.
\end{remark}
\begin{flushleft}
We recall here the definition of free Fisher information, which is defined via conjugate variables.
\end{flushleft}

\begin{definition}(Voiculescu def 3.1 in \cite{V})
 We say that $X=(x_1,\ldots,x_n)$ admits conjugate variables $\xi_{x_1},...,\xi_{x_n}\in L^2(W^*(X),\tau)$, if the following relation holds true for every $p\in \mathds{P}$ :
\begin{equation}
    \langle \xi_{x_j},p(x)\rangle_2=\langle 1\otimes 1^{op},[\partial_jp](X)\rangle_{\tau\otimes\tau^{op}},
\end{equation}
\end{definition}
\begin{definition}[Voiculescu def 6.1 in \cite{V}]
 The free Fisher information of X relative to a potential $V \in \mathds{P}^{(R)} $ (with $R > \lVert X\rVert $) is the quantity :
 \begin{equation}
     \Phi^*(X|V)=\sum_{j=1}^n \lVert \xi_{x_j}-[D_jV](X)\rVert_2 ^2=\lVert \xi_{X}-[DV](X)\rVert_2^2,
 \end{equation}
where 
\begin{equation}
    \xi_X=(\xi_{x_1},...,\xi_{x_n}),
\end{equation}
and 
\begin{equation}
    [DV](X)=\left([D_iV](X),...,[D_nV](X)\right),
    \end{equation}
 if such conjugate variables for $X$ exists and $+ \infty$ otherwise.
\end{definition}
For $V \in \mathds{P}^{(R)}, R > \lVert X\rVert$, we say that the joint law of X with respect to $\varphi$ is a free Gibbs state with potential $V$ if for each $j = 1,...,n$ and each $p \in \mathds{P}$ :
\begin{equation}
    \langle [D_jV](X),p(X)\rangle_2 =\langle 1\otimes 1^{op},[\partial_jp](X)\rangle_{HS}
\end{equation}
That is, if the conjugate variables to $X$ are given by $[D_1V ](X), . . . , [D_nV](X)$. Equivalently, the following equation holds for all $P \in \mathds{P}^n$:
\begin{equation}
    \langle[DV](X),P(X)\rangle_2 =\langle (1\otimes 1^{op})\otimes I_n,[\mathcal{J}P](X)\rangle_{HS}
\end{equation}
\bigbreak
Where $p(X):=ev_X(p)$ is the image of $p$ through the canonical evaluation homomorpism (always surjective but may fail to be injective).
\begin{eqnarray}
    ev_X : \mathds{P} \rightarrow \mathcal{M}\nonumber
\end{eqnarray}
which sends $ev_X(1)=1$ and $ev_X(t_i)=x_i$.
\bigbreak
In the same vein, the Jacobian defined in the section 2 is formally defined on indeterminates $t_1,\ldots,t_n$, and then specified through a matricial canonical evaluation homomorphism :
\begin{eqnarray}
    M_n(ev_Y\otimes ev_Y) : M_n(\mathds{P}\otimes \mathds{P}^{op}) \rightarrow M_n(\mathbb{C}\langle X\rangle \otimes \mathbb{C}\langle X\rangle^{op} )\nonumber
\end{eqnarray}
It is important to recall that we can actually replace the formal indeterminates $T=(t_1,\ldots,t_n)$ by $X=(x_1,\ldots,x_n)$ and write $\partial_i=\partial_{X_i}$ and $J_X$ defined on polynomials acting onto the variable $X$, provided that the tuple $X$ does satisfy no algebraic relations.
\bigbreak
This can rephrased in an algebraic formulation, by saying that the two-sided ideal :
\begin{equation}
    I_X=\left\{P\in \mathds{P}^n, P(x_1,\ldots,x_n)=0\right\}
\end{equation}
is the zero ideal of $\mathds{P}$. Equivalently, it means that the evaluation homomorphism $ev_X$ is in fact an isomorphism.
\bigbreak
There exists a powerful characterization of the absence of algebraic relations due to Mai, Speicher and Weber. Indeed, this condition is always ensured under finite free Fisher information (details could be found in \cite{MSW}).
\bigbreak
This implies in particular that we can see the Jacobian $J_X$ as a densely-defined closable operator on $L^2(W^*(X))^n$ with codomain $M_n(L^2(W^*(X)\bar{\otimes}W^*(X)))$.
\bigbreak
We will denote respectively the adjoint the free difference quotient and Jacobian respectively as: $\partial_i^*$ for each $i=1,\ldots,n$ and $J^*_X$.

\begin{remark}
The previous proposition could also be reformulated in the following way : when $\Phi^*(X) <\infty $, then $(1\otimes 1^{op})\otimes I_n$ belongs to $dom(\mathcal{J}_X^*)$, $\mathcal{J}_X^*$ acts on $B\in M_n(\mathbb{C}\langle X\rangle\otimes\mathbb{C}\langle X\rangle^{op})$ as:
\begin{equation}
    \mathcal{J}_X^*(B)=\left(\sum_{i=1}^{n}\partial_{i}^*(B_{j,i})\right)_{j=1}^{n}:=\left(J_{X,1}^*(B),\ldots,J_{X,k}^*(B)\right),
\end{equation}
Note also that we denote for $Q \in M_n(\mathcal{M}\otimes \mathcal{M}^{op})$ and $X=(X_i)_{i=1}^n$, the action of $Q$ on $X$, $Q\sharp X$ by:
\begin{equation}
    Q\sharp X=\bigg(\sum_{i=1}^n Q_{i,j}\sharp X_i\bigg)_{j=1}^n,
\end{equation}
where for $q=a\otimes b\in \mathcal{M}\otimes \mathcal{M}^{op}$ and $x\in \mathcal{M}$,
\begin{equation}
    q\sharp x=axb
\end{equation}
It is also easy to deduce for $A\in M_n(\mathbb{C})$, and $B\in dom(J_X^*)$, that $\left((1\otimes 1^{op})\otimes A\right).B$ belongs to $dom(J_X^*)$ and:
\begin{equation}
    J_X^*\left(((1\otimes 1^{op})\otimes A\right).B)=AJ_X^*(B),
\end{equation}
\end{remark}
\begin{proof}
Firstly, one can see that $dom(J_X^*)$ admits a left $M_n(\mathbb{C}\langle X\rangle \otimes \mathbb{C}\langle X\rangle^{op})$-action (section 2.2 in \cite{CN}).
\bigbreak
Then a simple computation shows that :
\begin{eqnarray}\label{25}
    J_X^*\bigg((\left(1\otimes 1^{op})\otimes A\right).B\bigg) &=&\left(\sum_{i=1}^n\partial_i^*(\left(\left(1\otimes 1^{op})\otimes A\right).B\right)_{j,i}\right)_{j=1}^n\nonumber\\
    &=&\left(\sum_{i=1}^n\partial_i^*\left(\sum_{k=1}^n(A_{j,k}.1\otimes1^{op})\sharp B_{k,i}\right)\right)_{j=1}^n\nonumber\\
    &=&\left(\sum_{k=1}^nA_{j,k}J_{X,k}^*(B)\right)_{j=1}^n\nonumber\\
    &=&AJ_X^*(B)
\end{eqnarray}\label{adj}

Where in the fourth line we used the linearity of the adjoint of free difference quotient (which is well-defined as a unbounded and densely defined closable operator) since the free Fisher information is finite. This formula is a particular case of the multidimensional extension of Voiculescu formula (section 4.1 in \cite{V}).
\end{proof}
\begin{definition}
We denote $V_C$ the following (self-adjoint) potential defined for all $C\in S_n^{++}(\mathbb{R})$ symmetric positive definite real matrix, and we denote this class of potential \begin{equation}
    V_{SC}^{++}=\left\{V_C,C\in S_n^{++}(\mathbb{R})\right\},
\end{equation}
\begin{equation}
V_C=\frac{1}{2}\sum_{i,j=1}^n C_{i,j}^{-1}t_i t_j,
\end{equation}
\end{definition}
\begin{remark}
We can see that $V_{SC}^{++}$ is a convex set of self-adjoint potential $V_{sa}=\left\{V=V^*, V\in \mathds{P}^n\right\}$ since $S_n^{++}(\mathbb{R})$ is a convex set and we have for every positive scalars $\lambda,\mu\geq0$ and $C,C'\in S_n(\mathbb{R}^{++})$, $\lambda V_{C}+\mu V_{C'}=V_{\lambda C+\mu C'}$.
\end{remark}
Note that, as will we see it further in the paper, this potential is the potential associated with a $n$-semicircular family with covariance $C$. In fact, by taking $C$ diagonal with $\rho > 0$ we find the potential $V_{\rho}$ which characterize a free  $(0,\rho^{-1})$ semicircular n-tuple.
\begin{definition}
 [The non-microstates free entropy. Voiculescu definition 7.1 in \cite{V}]
 Let $S = (s_1,\ldots, s_n)$ be a free (0, 1)-semicircular n-tuple,
free from $x_1,\ldots , x_n$. Then the non-microstates free entropy of $x_1, . . . , x_n$ is defined in \cite{V}  to be the quantity 
\begin{equation}
    \chi^*(x_1,\ldots,x_n)=\frac{1}{2}\int_{0}^{\infty}\left(\frac{n}{1+t}-\Phi^*(X+ \sqrt{t}S)\right)dt+\frac{n}{2}log(2\pi e),
\end{equation}
which we also denote by $\chi^*(X)$.
If $X$ is a free $(0, \rho^{-1})$-semicircular family, then
\newline
$X +\sqrt{t}S$ is also semicircular, 
therefore $\Phi^{*}(X +\sqrt{t}S)=n({\rho}^{-1}+t)^{-1}$. From this it is easy to compute $\chi^*(X) = \frac{n}{2}log(2\pi e\rho^{-1})$.
\end{definition} 
\bigbreak
\begin{definition}
For $X = (x_1,\ldots,x_n) \in M^n$, let $V \in P^{(R)}$ for $R \geq \lVert X \rVert$, we define the non-microstates
free entropy of $x_1,\ldots, x_n$ relative to $V$ to be the quantity
\begin{equation}
\chi^*(x_1,\ldots,x_n |V):=\varphi(V(X))-\chi^*(x_1,...,x_n),
\end{equation}
which we will also denote by $\chi^*(X | V )$. In this context, we refer to $V$ as a potential.
Since $\chi^*(.)$ is maximized (for fixed variance) by an n-tuple of free semicircular operators, it is easy to see
that $\chi^*(.| V_\rho)$ is minimized by a free $(0, \rho^{-1})$-semicircular n-tuple.
\end{definition}

\textbf{In the sequel, we will only moreover assume that the state is also a trace and repressed the notation $op$ for the trace. We also will work in $(\mathcal{A},\tau)$ a tracial $W^*$-probability space, which is a von Neumann algebra equipped with a faithful normal tracial state.}

\section{Stein method and semicircular families}
From now, we will focus on multidimensional semicircular approximations for semicircular family with a given covariance matrix $C$ (suppose to be symmetric definite positive). We will see in the sequel that our approach seems different for the original one on Gaussian case. Indeed, this one focus on minimizing the discrepancy between a Stein kernel and the matrix $C$ defined by the usual Gaussian integration by parts, which involves a second order differential operator : the Hessian. In free probability, the context is a bit different. Indeed it is better to understand the tuple of conjugate variables for a potential $V$ (suppose to be a convex potential to ensure uniqueness of the Free Gibbs State : \cite{YGS} for technical requirements). An approach which would follow the original proof will lead to an estimation for a weaker distance (not yet defined and studied in the multivariate case) by linking the tuple and the target via a bridge (this idea taken from proofs of smart path methods) and use usual Malliavin operators don't seem optimal as pointed out by Cébron because it leads to an estimation for a weaker distance and the fully-symmetry assumption.
\bigbreak
That's why, we adapt the original proofs of Cébron (\cite{C}) which is based on transport inequality (links between quantities such as Wasserstein distance, and non-microstate free entropy or free Stein discrepancy). As a concluding remark, we point out that our strategy of is closed to the ideas from Ledoux, Nourdin Peccati which states such inequalities in the classical case for larger classes of random variables, with a particular focus on multivariate normal, gamma, beta and even more general measures. They use in fact the semigroup interpolation and powerful tools of iterated gradient of Markov operators and Gamma calculus to get various inequalities ({\it LSI, HSI, WSH}) for invariant measure of diffusion process (second order differential operators, see \cite{LNP} for details).
\bigbreak
Let us introduce the fundamental distribution in free probability, which is the semicircular distribution:
\begin{definition}
The centered semicircular distribution of variance $\sigma^2$ is the probability distribution :
\begin{equation}
    S(0,\sigma^2)(dx)=\frac{1}{2\pi\sigma^2}\sqrt{4\sigma^2-x^2}dx, \; \lvert x\rvert\leq 2\sigma,
\end{equation}
This distribution has all his odd moments which vanish by the symmetry of the distribution around $0$. Its even moments can be expressed by the help of the {\it Catalan numbers} through the following relation for all non-negative integers $m$ :
\begin{equation}
    \int_{-2\sigma}^{2\sigma}x^{2m}S(0,\sigma^2)(dx)=C_m\sigma^{2m},
\end{equation}
Where $C_m=\frac{1}{m+1}\binom{2m}{m}$ is the {mth Catalan number}.

\end{definition}
\begin{definition}
Let $n\geq 2$ be an integer, and let $C = (C_{i, j})_{i,j=1}^n$
be a positive definite symmetric matrix. A n-dimensional vector $(S_1, ..., S_n)$ of random variables in $(\mathcal{A},\tau)$ is said to be a semicircular family with covariance $C$, if $\forall n \in \mathbb{N}$, $\forall(i_1, ..., i_n) \in [n]=\left\{1,\ldots,n\right\}$ :

\begin{equation}
    \varphi(S_{i_1}S_{i_2}...S_{i_n})=\sum_{\pi \in NC_2[n]}\prod_{\left\{a,b\right\}\in \pi}C_{i_a,i_b},
\end{equation}
where $NC_2(n)$ is the set of pairings of $\left\{1,\ldots,n\right\}$ which have no crossing.
\end{definition}
\begin{lemma}
Let $C^{\frac{1}{2}}$, the unique non-singular symmetric matrix such as $\left(C^{\frac{1}{2}}\right)^2=C$, then given a free (0, 1)-semicircular family $(X_1,...,X_n)$,
\begin{equation}
    C^{\frac{1}{2}}X=\left(\sum_{j=1}^nC^{\frac{1}{2}}_{i,j}X_j\right)_{i=1}^n
\end{equation} is a semicircular family with covariance C.
\end{lemma}
\begin{proof}
Standard, as semicirculars families in tracial $W^*$-probability spaces are stable by linear transformation. Indeed:
\newline
For $k,n\geq 1$, let $S=(S_1,...,S_n)$ a centered semircircular family in some tracial $W^*$-probability space and let $M\in M_{k,n}(\mathbb{R})$, then $Y:=MX$ is a semicircular family. In particular with $M\in M_{1,n}(\mathbb{R})=(\lambda_1,\ldots,\lambda_n)$ a column vector we have that $Y=\lambda_1S_1+\ldots+\lambda_nS_n$ is semicircular. This last result ensure that any semicircular family in a tracial $W^*$-probability space are only determined by their covariance matrix $\left\{\tau(X_iX_j)/ i,j \in [n]\right\}$
(lemma 1.4 in \cite{DS} and section 2.9 in \cite{NT}) :
\end{proof}

\bigbreak
Now, we would like to be able to compute the conjugate variables (solution of Schwinger-Dyson equations), to have a natural guess for our construction of free Stein Kernel. It turns out, as in the classical Gaussian case, there is a simple expression for the conjugate system of semicircular family with covariance C (see \cite{NPSw} for details in the commutative case).
\begin{lemma}
Let $S$ a n-semicircular family with covariance C, then the conjugate system of $S$, is given by :
\begin{equation}
    C^{-1}S=\left(\sum_{j=1}^n C^{-1}_{i,j}S_{j}\right)_{i=1}^n,
\end{equation}
\end{lemma}
\begin{remark}
In fact it's now obvious that the potential $V_C$ defined before, which satisfies $[D_iV_C]=\sum_{j=1}^nC^{-1}_{i,j}t_j$ is the potential associates with a $n$-semicircular family with covariance C.
Moreover, it characterizes completely a semicircircular with symmetric positive definite covariance $C$. Indeed, a self-adjoint vector $X=(x_1,\ldots,x_n)$ satisfies the following Schwinger-Dyson equation for all $P\in \mathds{P}^n$:
\begin{equation}
    \langle C^{-1}X,P(X)\rangle_2=\langle (1\otimes 1)\otimes I_n,[\mathcal{J}P](X)\rangle_{HS},
\end{equation}
if and if only $X$ is a semicircular family with covariance $C$. 
\end{remark}
\bigbreak
Let us consider the following free stochastic differential equation, where $C^{-1}$ is the inverse of the covariance matrix which does exist by the positive definite assumption on $C$, with $(X_0^1,\ldots,X_0^n)$ free from $(S_t^i)_{i=1}^n$:
\begin{equation*}
X_t=X_0-\frac{1}{2}\int_0^t C^{-1}X_tdt+S_t
 \end{equation*}
in vector notations or equivalently for all $i=1,\ldots,n$ :
\begin{equation*}
X_t^i=X_0^i-\frac{1}{2}\int_0^t (C^{-1}X_t)_idt+S_t^i,
 \end{equation*}
where 
\begin{equation}
    (C^{-1}X_t)_i=\sum_{j=1}^nC^{-1}_{i,j}X_t^j,
\end{equation} for all $i=1,\ldots, n$.
\begin{remark}
The existence of such a free SDE is ensured, by a classical Picard fixed point argument, since the drift and the volatility are linear functions, the existence and uniqueness turns out to be well understood (see \cite{Dab10}, theorem 25 or \cite{BS2}).
\end{remark}
\section{An infinitesimal estimate for free SDE}
In this section, we will be interested in infinitesimal estimations of Wasserstein distance between two close points of the free stochastic differential equations. Before stating our results, we introduced the non-commutative Wasserstein distance, which is the free counterpart of the usual quadratic Wasserstein distance.
\begin{definition}(Biane Voiculescu, \cite{BV})
The free quadratic Wasserstein distance is defined as the infinimum over couplings :
\begin{eqnarray*}
    d_W((X_1,...,X_n),(Y_1,...,Y_n))&=&inf \Bigl\{\lVert (X_i^{'} - Y_i^{'})_{1\leq i\leq n}\rVert_2 /
    (X_1^{'},...,X_n^{'},Y_1^{'},...,Y_n^{'}) \subset (M_3,\tau)\nonumber\\
    &&(X_1^{'},...,X_n^{'})\simeq(X_1,...,X_n),(Y_1^{'},...,Y_n^{'})\simeq(Y_1,...,Y_n)\Bigl\}
, \end{eqnarray*}
with $(M_3,\tau)$ is a $W^*$ tracial probability space, and with each $(X_i^{'},Y_i^{'})\in M_3$, where $\simeq$ means the equality in distribution.
\end{definition}
\begin{remark}
The following theorem is a powerful result of Dabrowski (\cite{Dab10}), We don't expose all the theory invented by Dabrowski which is rather complex, and we send the reader to the paper to have a complete exposure of all the results. We will only remind a part of the proof to show what type of computations we will use in the sequel. We also remcall to the interested reader that a free analog of the Wasserstein space as been introduced and studied by Jekel, Shlyakhtenko and Gangbo in \cite{JSG} and a Monge-Kantorovitch duality was obtained in this context in \cite{JSGN}.
\end{remark}
\begin{theorem}(Dabrowski, \cite{Dab10})\label{th1}
Assume that we are given a process of the following form where
 $S_t=(S_t^1,...,S_t^n)$ is a free Brownian motion on some filtered $W^*$-probability space $(\mathcal{A},\tau)$ with self-adjoint initial conditions:
\begin{equation*}
X_t=X_0-\frac{1}{2}\int_0^t \Theta_t dt+S_t
 \end{equation*} 
 in vectors notations or equivalently :
 \begin{equation*}
X_t^i=X_0-\frac{1}{2}\int_0^t \Theta_t^i dt+S_t^i
 \end{equation*} 
 We suppose for each $s\geq 0$, that $\Theta_s=(\Theta_1,...,\Theta_n)$ is continuous to the right for $\lVert .\rVert_2$ and belongs to the von Neumann algebra $W^*(X^1_s,...,X^1_s)$, we also assume that  conjugate variables for this tuple exists :
 \begin{equation}\partial^*_{i,s}(1\otimes 1)=\xi^i_s
 \end{equation}
  Then we have the following estimation of the Wasserstein distance between two close points (see \cite{Dab10} for the precise definition), i.e for $t\geq s$ :
 \begin{equation}
     d_W(X_t,X_s)\leq \frac{t-s}{2}\lVert \Theta_{s}-\xi_s\rVert_2 +o(t-s),
 \end{equation}
\end{theorem}

We would like to apply the theorem before to obtain quantitative bounds for the Wasserstein distance between a $n$-tuple of self-adjoint random variables and a n-semicircular family with covariance $C$.
Intuitively, we have to connect them through an appropriate Ornstein-Uhlenbeck process, the idea to make appear the conjugate variables is to see that if the tuple is close to be a $n$-semicircular family with covariance $C$, and thus the conjugate variables of $X$ should be close to $C^{-1}X$, its then natural to consider a free SDE with that potential, in fact for a $n$-$(0,1)$ free semicircular system, we get that the conjugate variables should be close to the tuple itself (this characterized for centered normalized semicircular system, see \cite{Dab10} proof of free Talagrand inequality). We will see in the following that Fisher information is closely related to the Stein discrepancy.
\bigbreak 
We consider the following approach. Following Dabrowski's ideas: the drift,i.e the convex potential $V$ of the free SDE carries all the information.
\bigbreak
The other one is based on a linear transformation of a free SDE: the information is contained in the volatility process and not anymore in the drift.
\bigbreak
\begin{theorem}\label{Th2}
Let's consider as before the following free differential equation, with the initial conditions $X_0$ free from $(S_t)_{t\geq 0}$ :
\begin{equation*}
X_t=X_0-\frac{1}{2}\int_0^t C^{-1}X_tdt+S_t,
 \end{equation*}
This equation admits a solution which turns out to be equal in distribution to 
\begin{eqnarray}\label{24}
X_t&\simeq& e^{-\frac{t}{2}C^{-1}}X_0+\left(I_n-e^{-tC^{-1}}\right)^{\frac{1}{2}}C^{\frac{1}{2}}S^{'}\\
&\simeq& e^{-\frac{t}{2}C^{-1}}X_0+\left(I_n-e^{-tC^{-1}}\right)^{\frac{1}{2}}S_C
\end{eqnarray}
with $S'$ a free $(0,1)$ $n-$semicircular system and $S_C$ a $n$-semicircular family with covariance $C>0$.
\bigbreak
Since the transformation is linear, the potential is again valued in the Von-Neumann algebra $W^*(X_s^1,...,X_s^n)$ and thus
all hypothesis of our previous theorem are fulfilled, especially the right continuity of the process for the 2-norm.
\bigbreak
From all of these conditions, we get that: \begin{equation}
\frac{d^{+}}{dt} d_W(X_t,S_C)\leq \frac{1}{2}\Phi^*(X_t|V_C)^{\frac{1}{2}}
\end{equation}
\end{theorem}
\begin{proof}
By the previous theorem, we get that $d_W(X_t,X_s)\leq\frac{1}{2}\lVert C^{-1}X_s-\xi_s\rVert_2+o(t-s)$.
Now it is essentially the same idea of \cite{Dab10}, which is a triangle inequality and the estimation given in the Theorem \ref{th1} for the quadratic Wasserstein distance, to get:
\begin{equation}
\limsup_{\epsilon \to 0}\frac{1}{\epsilon}|d_W(X_{t+\epsilon},S)-d_W(X_t,S)|\leq \frac{1}{2} {\Phi^*(X_t|V_C)}^{\frac{1}{2}}.
\end{equation}
\end{proof}
\qed
\begin{remark}
Note that we have the equality in distribution by a simple Ito formula with the usual variation of constant method with integrating factor: $e^{tC^{-1}}$, which is the semigroup generated by $C^{-1}$ (understand as the usual matrix exponential).
\bigbreak
The following process $(e^{tC^{-1}}X_t)_{t\geq 0}$ is then well-defined because it is a simple linear transformation of the process $(X_t)_{t\geq 0}$, and moreover, we get by integration by parts for all $t\geq 0$:
\begin{equation}
    e^{tC^{-1}}X_t=X_0+\int_0^{t}e^{vC^{-1}}dS_v
\end{equation}
where for an adapted process $\left(U_t\right)_{t\geq 0}$ valued in $M_n(\mathcal{A})$ (to the filtration generated by the free Brownian motion : $W^*(S_t^1\ldots,S_t^n)$), we denote, as in the preliminaries (where the integral is constructed for step process and then extended to the class of matricial adapted processes by density and Ito isometry) :
\begin{equation}
    \int_{0}^tUdS_v=\left(\sum_{j=1}^n\int_0^tU_{i,j}dS_v^j\right)_{i=1}^n
\end{equation}
From that we deduce:
\begin{equation}
    X_t=e^{-\frac{t}{2}C^{-1}}X_0+e^{-\frac{t}{2}C^{-1}}\int_0^{t}e^{\frac{v}{2}C^{-1}}dS_v
\end{equation}
From this, it is easy to remark that the last part of the member: the stochastic integral, is a (multivariate) semicircular process, with centered mean, and moreover we can deduce that:
\bigbreak
For all $t\geq 0$,
\begin{equation}
    \tau(X_t)=e^{-\frac{t}{2}C^{-1}}
    \left(
    \begin{array}
  c \tau (X_0^1)\\
  \tau (X_0^2) \\
  \vdots \\
  \tau(X_0^n)\\
  \end{array} 
    \right)\\
\end{equation}
Note moreover, that if we assume that $X_0$ is a semicircular family with covariance $C$, by freeness between $X_0$ and $(S_t)_{t\geq 0}$, then $X_t$ is stationary (in distribution).
\bigbreak
The cross covariance matrix of the process $(V_t)_{t\geq 0}=\left(e^{-\frac{t}{2}C^{-1}}\int_0^{t}e^{\frac{v}{2}C^{-1}}dS_v\right)_{t\geq 0}$ is simply obtained by Ito isometry (with the commutativity between the two exponential matrix since they are functions of the matrix $C$) : \begin{equation}
    cov(V_t,V_s)=e^{-\frac{(t+s)}{2}C^{-1}}\int_0^{t\wedge s}e^{vC^{-1}}dv
\end{equation}
Now, it is well known that the last integral could be computed explicitly, since for $A\in GL_n(\mathbb{R})$ :
\begin{equation}
    \int_0^{t}e^{sA}ds=\left(e^{tA}-I_n\right)A^{-1}
\end{equation}
which gives for the covariance matrix of $X_t$ (here $t$ is fixed) :
\begin{equation}
    cov(V_t)=\left(I_n-e^{-tC^{-1}}\right)C
\end{equation}
\bigbreak
And finally we get in distribution, for all $t\geq 0$:
\begin{equation}
X_t\simeq e^{-\frac{t}{2}C^{-1}}X_0+\left(I_n-e^{-tC^{-1}}\right)^{\frac{1}{2}}C^{\frac{1}{2}}S^{'}
\end{equation}
With $S^{'}$ a free $(0,1)$ semicircular system.
\bigbreak
This provides a useful interpolation between the tuple $X$ and the semicircular family of covariance $C$ (analog of the Ornstein-Ulhenbeck semigroup with drift in the classical case). It is easily seen that if the matrix $C=\rho I_n$, we find the usual interpolation given by the Ornstein-Uhlenbeck semigroup and already used in \cite{FN} (with a simple time change in this case $t\mapsto \rho t$) or in \cite{C}. It seems rather convenient to work with the interpolation (\ref{24}) compared to the following theorem \ref{3} which computes explicitly the conjugate variable of a free SDE with polynomial drift, in order to deduce a bound of free Fisher information along the path of the Ornstein-Uhlenbeck process by means of free Stein discrepancy. Note that we will work with $\left(X_{2t}\right)_{t\geq 0}$ to avoid the factor $\frac{1}{2}$ in the exponentials.
\end{remark}
\bigbreak
As in the commutative case, our previous results will provide some inequalities relating free Fisher information and a quantity which is slightly modification (gaussian variant) of non-microstates free entropy. This is the price to pay to deduce such inequality since there is still an unknown change of variables for the non-microstate free entropy $\chi^{*}$, this type of inequality is known as Logarithmic Sobolev inequality ({\it LSI} in abbreviation). Before that, we have to introduce some notations and theorems of Dabrowski.
\begin{theorem}(theorem 25 in \cite{Dab10})\label{3}
For a free SDE with self-adjoint initial conditions
\begin{equation}\label{dabc}
    X_t^i=X_0^i-\frac{1}{2}\int_0^t D_iV(X_t)dt+S_t^i
\end{equation}
with $V$ a self-adjoint polynomial convex potential as in \cite{Dab10} (one can solve this equation in a strong sense. In the sequel, we will only suppose that we are given a solution), then
$X_t^i,\ldots X_t^n$ have conjugate variable $\xi_t^i$ for $t>0$, and if they also does for $t=0$, we have  $\tilde\xi_{V,t}=(\tilde\xi_{V,t}^i,\ldots,\tilde\xi_{V,t}^n)$ with for each $i=1,\ldots,n$:
\begin{equation}
    \tilde\xi_{V,t}^i=\tilde\xi_t^i-D_iV(X_t)\nonumber
\end{equation}
solutions of the following free SDE :
\begin{equation}
   \tilde\xi_{V,t}^i= \tilde\xi_{V,0}^i-\frac{1}{2}\sum_{j=1}^n\int_0^{t}\partial_j([D_iV](X_s))\sharp\tilde\xi_{V,s}^ids\nonumber
\end{equation}\label{dablem}
Moreover $\xi_s^i=\tau(\tilde\xi_s^i|X_s)$ where $\tau(.|X_s)$ is the (only trace preserving) conditional expectation on $W^*(X_s)$.
\bigbreak
Two notions of convexity have been studied, the first one due do Guionnet and Shlyakhtenko in \cite{GS}, precisely one says that $V$ is $(c,M)$-convex if for all $X,Y$ with $\lVert X\rVert,\lVert Y\rVert\leq M$, we have the existence of a constant $c>0$, such that:
\begin{equation}
    [DV(X)-DV(Y)].(X-Y)\geq c(X-Y).(X-Y)
\end{equation}
where $X.Y=\frac{1}{2}\sum_{j=1}^nX_i.Y_i+Y_iX_i$ stands for the anti-commutator of $X$ and $Y$ (if $X,Y$ are self-adjoint and $\lVert X\rVert$,$\lVert Y\rVert \leq M$).
\newline
We can also obtain the following exponential decay of the free Fisher information if the potential $V$ satisfies the following (different) convexity assumption, i.e there exist $c>0$ (uniform in $i,j$):
\begin{equation}\label{conv}
    (\partial_jD_iV)_{i,j}\geq c(1\otimes 1\delta_{i,j})_{i,j}
\end{equation}
for all $i,j=1,\ldots,n$, where the inequality holds in $M_n(\mathds{P}\otimes\mathds{P}^{op})$.
\bigbreak
\begin{equation}
    \Phi^*(X_t^1,\ldots,X_t^n|V)\leq e^{-c(t-s)}    \Phi^*(X_s^1,\ldots,X_s^n|V)
\end{equation}
\bigbreak
We will give a short proof (as we were unable to find a proof in the literature) of the last statement, we can already notice that if $X_1,\ldots,X_n$ and $V=V^*$, then $\theta(X)=JDV(X)$ is self-adjoint:
\begin{proof}
From (\ref{dabc}), we are left to prove the exponential decay of the $L^2$-norms of each $\tilde\xi^i_{V,u}$ for $i=1,\ldots,n$.
\bigbreak
Now, for $s\leq t$, a simple integration by parts shows that :
\begin{equation}
    e^{c\frac{t}{2}}\tilde\xi^i_{V,u}=e^{c\frac{s}{2}}\tilde\xi^i_{V,s}+\frac{1}{2}\sum_{j=1}^n\int_s^{t}\bigg(c.1\otimes1^{op}\delta_{i,j}-\partial_j[D_iV](X_u)\bigg)\sharp\left(e^{c\frac{u}{2}}\tilde\xi^i_{V,u}\right)du
\end{equation}
Taking the $L^2$ norms and applying Grownall lemma, since the semigroup  $e^{\frac{ct}{2}.(1\otimes 1^{op})\otimes I_n-\frac{t}{2}\mathcal{J}DV}$ (where the exponential is defined as the usual exponential in Banach algebra) generated by \begin{equation}
    Q=\left(\frac{c}{2}.(1\otimes 1^{op}\delta_{i,j})-\frac{1}{2}\partial_jD_iV\right)_{i,j=1}^n\in M_n(\mathds{P}\otimes \mathds{P}^{op})
\end{equation} is contractive by the previous assumption (\ref{conv}) as well as (in another context) the conditional expectation for the $L^2$-norm, we finally get the result.
\end{proof}
\begin{remark}
As mentioned by Guionnet and Shlyakhtenko, when the underlying non commutative probability space is the space of complex matrices equipped with natural operations (with in particular the transconjugate as involution and operator norm), a potential $V$ is self-adjoint $(c,M)$ convex if the map $\left((X_{i,j})_{i,j=1}^n\right)\rightarrow Tr(V(X))$ is strictly convex on the set of Hermitian matrices with spectral radius less or equal than $M$.
\bigbreak
The other one, investigated recently by Dabrowski, Guionnet and Shlyakhtenko in \cite{YGS} imposes that the trace of the Hessian of $V$ generates a semigroup of contraction in a new space of non commutative function, on which the topology is given by the extended Haagerup tensor product. For example quartic convex potential are allowed in this one. It's is really powerful, because only subquadratic potential were studied before.
\end{remark}
We remind to the reader that such estimates have been obtained by Fathi and Nelson in \cite{FN} for semicircular systems with (strictly positive) homothety covariance matrices, by computing conjugate variables with another method which is based on the equality in distribution of the free Ornstein-Uhlenbeck process. Note that this method will be used in the sequel to compute an inequality between Free Fisher information and Free Stein discrepancy along the path of an Ornstein-Uhlenbeck with positive linear drift.
\begin{remark}
In our case, a simple computation shows that this condition holds for $V_C$ defined before. In fact, the first condition which implies that one can solve the free SDE in a strong sense, and the last one to get exponential decay of the Free Fisher information holds true for $c=\lVert C\rVert_{op}^{-1}$.
\end{remark}
\end{theorem}
\begin{definition}
For a free SDE with $V\in \mathbb{C}\langle t_1,\ldots,t_n\rangle$ (we assume that we are given a strong solution, it is always ensured if the first convexity assumption is fulfilled) : 
\begin{equation}
    X_t=X_0-\frac{1}{2}\int_0^t DV(X_t)dt+S_t\nonumber
\end{equation} where use vector notations and $X_0=\left(X_0^i\right)_{i=1}^n$.
\bigbreak
We define the modified non-microstates free entropy relative to a potential $V$ by :
\begin{equation}
    \chi^*_V(X_0^1,...,X_0^n)=-\int_0^{\infty}\frac{1}{2}\sum_{j=1}^n\lVert \xi^i_t-D_iV(X_t)\rVert_2^2dt=-\int_0^{\infty}\frac{1}{2}\lVert \xi_t-DV(X_t)\rVert_2^2dt
\end{equation} where $\xi_t$ is the conjugates variables of $X_t$ (provided that their exists).
\end{definition}
This allows us to obtain the Log-Sobolev inequality for self-adjoint convex potential $V$, since we have the following De Bruijn's formula which holds true (in fact this is the main motivation for considering such a different definition of the non-microstates free entropy relative to a potential)
\begin{theorem}
Let $V$ a self-adjoint convex potential as in \cite{Dab10}, and $\chi^*_V$, the modified non-microstates free entropy defined previously, then by definition, we have :
\begin{equation}
    \chi^*_V(X_0^1,...,X_0^n)=-\frac{1}{2}\int_0^{\infty}\Phi^*(X_t^1,\ldots,X_t^n|V)dt,
\end{equation}
And the following Log-Sobolev inequality holds :
\begin{equation}
    -\chi^*_V(X_0^1,...,X_0^n)\leq \frac{1}{2c}\Phi^*(X_0^1,\ldots,X_0^n|V),
    \end{equation}
\end{theorem}
\begin{proof}
The first part follows easily from the definition and for the other one we use the exponential decay of the free Fisher information, to get :
\begin{eqnarray}
    -\chi^{*}_V(X_0^1,...,X_0^n)&\leq& \frac{1}{2}\Phi^*(X_0^1,\ldots,X_0^n|V)\int_0^{\infty} e^{-cs}ds\nonumber\\
    &=&\frac{1}{2c}\Phi^*(X_0^1,\ldots,X_0^n|V)
    \end{eqnarray}
which leads to the desired conclusion.
\end{proof}
\qed
\bigbreak
The following lemma allows us to obtain an inequality between the free Fisher information relative to the semicircular potentials and the associated free Stein discrepancy and also a bound for the quadratic Wassertein distance (WS) between a tuple and a semicircular family of covariance $C>0$. Moreover, it will also lead to an {\it HSI} inequality for the class of potential $V_C$. It was already obtained by Fathi and Nelson in \cite{FN} for semicircular family of covariance $C=\rho^{-1}I_n$ and for $\chi^*(.|V_C)$, we will obtain it for the modified non-microstates free entropy $\chi^*_V$.
\bigbreak
It seems really important to notice that in the case of \cite{FN}, we have in fact the equality between the two non-microstates free entropy : $\chi^*(.|V_{\rho^{-1}I_n})-\chi^*(S_{\rho}|{V_\rho^{-1}I_n})=\chi^*_{V_{\rho^{-1}I_n}}(.)$ and it remains an open conjecture to prove the following formula which is a change of variable for $\chi^*_V$ :
\begin{equation}
    \chi^*_V(.)=\chi^*(.|V)-K
\end{equation}
for every convex potential where $K$ is an unknown constant.
\paragraph{First approach}
\begin{flushleft}
For the first approach, consider $X$ a $n$-tuple of self-adjoint non commutative random variables. We note for $S^{'}$ a standard semicircular family and a covariance matrix $C>0$, and $S$ a semicircular family with covariance $C>0$. $X_t=e^{-2t}+\sqrt{1-e^{-2t}}C^{\frac{1}{2}}S$.
\end{flushleft}
\begin{flushleft}
Then it is easy to remark that for the first equality is evident since $S\stackrel{(d)}{=}C^{\frac{1}{2}}S^{'}$ (where $(d)$ denotes the equality in distribution) and that the free Wasserstein distance compares only the distributions :
\begin{equation}
     \frac{d^+}{dt}d_W(X_t,X)\leq \lVert C\rVert_{op}^{\frac{1}{2}}\Phi^*(C^{-\frac{1}{2}}X|V_{I_n})^{\frac{1}{2}}
\end{equation}
Now it is essentially equivalent to prove an estimate for $\Phi^*(C^{-\frac{1}{2}}X_t|V_{I_n})^{\frac{1}{2}}$ by a quantity involving a different Stein discrepancy. Indeed, in this approach, one have to in a different free Stein discrepancy.
\end{flushleft}
\begin{flushleft}
By Cébron result (lemma 2.6 in \cite{C}) , one has : 
\begin{equation}
    \Phi^*(C^{-\frac{1}{2}}X_t|V_{I_n})^{\frac{1}{2}}\leq \frac{e^{-2t}}{\sqrt{1-e^{-2t}}} \Sigma^*(C^{-\frac{1}{2}}X|V_{I_n})
\end{equation}
Now the trick to pursue the computations, is to remark that when $A$ is a free Stein kernel with respect to the potential $V_{I_n}$, then for any invertible matrix $B$, $(1\otimes 1)\otimes B.A.(1\otimes 1)\otimes B^{T}$ is a free Stein kernel for $BX$ with respect to the potential $V_{I_n}$.
\newline
Indeed, denote $A$ a free Stein kernel for $X$ with respect to $V_{I_n}$. By the chain rule for the non-commutative Jacobian, one has :
\begin{eqnarray}
\langle BX,P(BX)\rangle_2&=&\langle X,B^TP(BX)\rangle_2\nonumber\\
&=&\langle A,[J(B^TP)](BX)\rangle_2\nonumber\\
&=&\langle A,(1\otimes 1)\otimes B^T. J_X(BX)\rangle_{HS}.\nonumber\\
&=&\langle A,(1\otimes 1)\otimes B^T. J_{BX}(BX).(1\otimes 1)\otimes B\rangle_{HS}.\nonumber\\
&=&\langle (1\otimes 1)\otimes B.A.(1\otimes 1)\otimes B^T,[JP](BX)\rangle_{HS}.
\end{eqnarray}
where for $P\in \mathds{P}^n$, $B^TP$ is the linear transformation of $P$ by $B^T$.
\end{flushleft}

Now we define the following modified Stein discrepancy, called the 
\begin{definition}
The free Gaussian Stein discrepancy with respect to the potential $V_C$ is then defined as :
\begin{eqnarray}
\widetilde{\Sigma} ^*(X|V_C)=\inf_{A}\bigg\lVert (1\otimes 1)\otimes C^{-\frac{1}{2}}.A.(1\otimes 1)\otimes C^{-\frac{1}{2}}-(1\otimes 1^{op})\otimes I_n\bigg\rVert_{L^2(M_n(\mathcal{M}\bar{\otimes} \mathcal{M}^{op}),(\tau \otimes \tau^{op}) \circ Tr))}\nonumber
\end{eqnarray}
where $A$ is a free Stein kernel with respect to $V_{I_n}$.
\end{definition}
\begin{flushleft}
And we can conclude that :
\end{flushleft}

\begin{proposition}
For any self adjoint vector $X$ of non-commutative random variables, one has :
\begin{equation}
    d_W(X,S)\leq \lVert C\rVert^{\frac{1}{2}}_{op}\widetilde{\Sigma}(X|V_C)
\end{equation}
\end{proposition}
\begin{flushleft}
Note that this approach contributes to nothing, in other functionnal inequalities. Indeed, this one can only be useful for probabilistic approximations.
\end{flushleft}
\paragraph{Second approach}
\begin{flushleft}
The following lemma is dedicated to the second approach which is more adapted to the free setting, and this one will connect better quantites with respect to the potential $V_c$. In fact, to prove more general inequalities involving the Dabrowski's variant of non-microstates free entropy, the free Fisher information the free Wasserstein distance and the free Stein discrepancy with respect to the potential $V$, one have to choose this approach.
First, to lighten the notations, we will work directly with $S$ a semicircular family with covariance $C$ which will be supposed to be free from $X$, rather than suppose that $X$ and $S'$ (a centered normalized semicircular system) and then consider $C^{\frac{1}{2}}S'$.
\end{flushleft}
\begin{lemma}\label{lma3}
Let X be a n-tuple of self-adjoint element of $(\mathcal{A},\tau)$ a $W^*-$probability space and let $S$ be a n semicircular family of covariance C living also in the same space $(\mathcal{A},\tau)$ (one always can construct such space with the help of the free product) suppose definite positive (it implies in particular that $C^{-1}$ is bounded from below by $cI_n$ in the usual operator sense). 
\bigbreak
Denote for all $t\geq0$, $X_t=e^{-tC^{-1}}X+\left({I_n-e^{-2tC^{-1}}}\right)^{\frac{1}{2}}S$, then we have :

\begin{equation}
    \Phi^*(X_t|V_C)^{\frac{1}{2}}\leq \frac{e^{-2t\lVert C\rVert_{op}^{-1}}}{\sqrt{1-e^{-2t\lVert C\rVert_{op}^{-1}}}}\lVert C^{-1}\rVert_{op}^{\frac{1}{2}} \Sigma(X|V_C)
\end{equation}
\end{lemma}
\begin{proof}
{\it We follow carefully the original proof of Cébron with some adaptations. 
In fact, by Voiculescu formulas, since we have freeness between $X$ and $S$ (corollary 3.9 and corollary 6.8 in \cite{V}), we get :
\begin{eqnarray}
    \xi_t=\left({I_n-e^{-2tC^{-1}}}\right)^{-\frac{1}{2}}\tau\left(C^{-1}S|X_t\right)\nonumber
\end{eqnarray}
With
\begin{equation}
    \tau(C^{-1}S|X_t)=\bigg(\tau((C^{-1}S)_i)|X_t)\bigg)_{i=1}^n
\end{equation}
where we use vector notations in order to avoid heavy ones. In particular the conditional expectation has to be taken componentwise, i.e \begin{equation}
    \xi_t=(\xi^1_t,...,\xi^1_t)
\end{equation}
And since we have a linear transformation of $S$, we can also see that $\tau(C^{-1}S|X_t)=C^{-1}\tau(S|X_t)$. Note that using one or the other, it will lead to the same conclusion.
\begin{equation}
    \tau(C^{-1}S|X_t)=\bigg(\tau((C^{-1}S)_i)|X_t)\bigg)_{i=1}^n
\end{equation}
We can also assume that the free Stein discrepancy is finite: $\Sigma(X|V_C)< \infty$ (if not we have nothing to prove,  moreover as seen in \cite{FCM}, it always exists, provided that the evaluation of the tuple by the cyclic derivative potential vanishes) and we denote by $A$, a free Stein kernel with respect to $V_C$.
\bigbreak
Then because $X$ is supposed free from $S$, and since the orthogonal projection of Stein kernel onto the subspace $L^2(M_n(W^*(F)\otimes W^*(F)),(\tau \otimes \tau)\circ Tr) $ is also a free Stein kernel, we can assume that $A$ belongs to this subspace.
\bigbreak
Now, we can consider the map for each $i=1\ldots,n$:
\begin{equation}\label{65}
    \partial_{S_i} : L^2(W^*(X,S))\rightarrow L^2(W^*(X,S))\otimes L^2(W^*(X,S))
\end{equation}
\begin{eqnarray}
\partial_{S_i}=\sum_{M=P_1S_iP_2} P_1 \otimes P_2\nonumber
\end{eqnarray}
already defined before and the non commutative Jacobian associated to this $n$-tuple of semicircular family with covariance $C$ (which is an unbounded closable operator).
\bigbreak
On $\mathbb{C}\langle X\rangle^{\otimes^2}$,
we easily deduce by the freeness between $F$ and $S$, that :
\begin{eqnarray}
\partial_{S_{i}}^*(P_1\otimes P_2)= P_1(C^{-1}S)_i P_2\nonumber
\end{eqnarray}
which implies for all $D \in \mathbb{C}\langle X\rangle^{\otimes^2}$
\begin{eqnarray}
    \lVert \partial_{S_i}^*(D) \rVert_2^2=\langle \partial_{S_i}\partial_{S_i}^*(D),D\rangle= C^{-1}_{i,i}\lVert D \rVert^2\nonumber
\end{eqnarray}
The absolute value is not needed for $C^{-1}_{i,i}$, since all diagonal elements of a positive definite matrix are positive.
\bigbreak
We can then deduce that the norm of the adjoint of the non-commutative Jacobian is given for all $B\in M_n(W^*(X)\otimes W^*(X))$ by using the freeness between $B$ and $S$
\begin{eqnarray}
    \lVert \mathcal{J}_{S}^*(B) \rVert_{2}^2=\lVert B\sharp(C^{-1}S) \rVert_{2}^2\leq \lVert C^{-1}\rVert_{op}\lVert B\rVert_{HS}^2\nonumber
\end{eqnarray}
\bigbreak
Where $(C^{-1}S)_i=\sum_{j=1}^n C^{-1}_{i,j}S_j$ for all $i=1,\ldots n$.
\bigbreak
Note also that the free Fisher information of the $n$-semicircular family $S$ with covariance $C$ is equal to :
\begin{equation}
    \Phi^*(S)=tr(C^{-1})
\end{equation}
\bigbreak
As pointed out previously in some papers of Cébron or Fathi, it could be very difficult to compute $J^*_S(A)$, however, under the conditional expectation $\tau(.|X_t)$, this term could be handled more easily.
\bigbreak
To have a nicer exposure, because the following computations involve a lot of notations, we set for all $t>0$:
\begin{eqnarray}
B_t&=&(1_{\mathcal{A}}\otimes 1_{\mathcal{A})}\otimes \left({I_n-e^{-tC^{-{1}}}}\right)^{\frac{1}{2}}e^{-tC^{-1}}\nonumber\\
D_t&=&(1_{\mathcal{A}}\otimes 1_{\mathcal{A}})\otimes \left({I_n-e^{-tC^{-{1}}}}\right)^{-\frac{1}{2}}e^{-tC^{-1}}\nonumber\\\nonumber\\
E_t&=&(1_{\mathcal{A}}\otimes 1_{\mathcal{A}})\otimes e^{-tC^{-1}}\nonumber\\
\end{eqnarray}
\bigbreak
Indeed, one can show that :
\begin{eqnarray}\label{expec}
\tau(C^{-1}X|X_t)=\tau\bigg(\mathcal{J}^*_S(A.D_t)|X_t\bigg)
\end{eqnarray}
To see that, we use the existence of free Stein kernel $A$, the freeness between $X$ and $S$, the fact that the matrices are self-adjoint and finally the chain-rule for the non-commutative Jacobian (details can be found in \cite{V}).
\bigbreak
Finally, we introduce the two following Jacobian $J_X$ and $J_{X_t}$ defined with respect to $X$ and $X_t$ in the same way of \ref{65} : 
\bigbreak
We now remind that on $\mathbb{C}\langle X_t\rangle$, we can invert the conditional expectation and the free difference quotient (by the freeness between $X$ and $S$). Then, we have the following proposition which is nothing but a chain rule for the Jacobian of a linear transformation $X\mapsto AX$ with \label{chain} $A\in GL_n(\mathbb{C})$:
$J_{AX}=J_X.(1\otimes1)\otimes A^{-1}$
\begin{equation}\label{pro67}
 \mathcal{J}_X.\bigg((1_{\mathcal{A}}\otimes 1_{\mathcal{A}})\otimes e^{tC^{-{1}}}\bigg) =\mathcal{J}_{X_t}=J_S.\left((1_{\mathcal{A}}\otimes 1_{\mathcal{A}})\otimes\left({I_n-e^{-2tC^{-{1}}}}\right)^{-\frac{1}{2}}\right)
 \end{equation}
\bigbreak 
We can then compute for all $P\in \mathds{P}^n$ :

\begin{eqnarray}
   \left\langle P(X_t), C^{-1}X\right\rangle_{2}\nonumber
   &=&\left\langle  \mathcal{J}_{X_t}(P(X_t)).E_t,A\right\rangle_{HS}\nonumber\\
   &=&\left\langle \mathcal{J}_S(P(X_t)),A.D_t\right\rangle_{HS}\nonumber\\
    &=&\left\langle P(X_t),\mathcal{J}_S^*[A.D_t]\right\rangle_{2}\nonumber\\
\end{eqnarray}
Where we used the remark \ref{adj} in the first line. In the second one, we used the equation (\ref{pro67}). In the fifth line we used the chain rule satisfied by the Jacobian. Finally, using the property (\ref{25}), we may obtain the desired conclusion. 
\bigbreak
From all these computations, we evaluate the free Fisher information with respect to the potential $V_C$ by and we note for $t>0$, $\psi_t=\xi_t-C^{-1}X_t$:
\begin{eqnarray}
    \Phi^*(X_t|V_C)&
    =&
    \left\langle \xi_t-C^{-1}X_t,\xi_t-C^{-1}X_t\right\rangle_2\nonumber\\
    &=&\left\langle \xi_t-C^{-1}\left(e^{-tC^{-1}}X+\left({I_n-e^{-2tC^{-1}}}\right)^{\frac{1}{2}}S\right),\xi_t-C^{-1}X_t\right\rangle_{2} \nonumber\\
    &=&\left\langle \left({I_n-e^{-2tC^{-1}}}\right)^{-\frac{1}{2}}C^{-1}S-e^{-tC^{-1}}C^{-1}X-C^{-1}\left({I_n-e^{-2tC^{-1}}}\right)^{\frac{1}{2}}S,\xi_t-C^{-1}X_t\right\rangle_{2}\nonumber\\
    &=&\left\langle\left({I_n-e^{-2tC^{-1}}}\right)^{-\frac{1}{2}}\left(C^{-1}S-\left({I_n-e^{-2tC^{-1}}}\right)^{\frac{1}{2}}e^{-tC^{-1}}C^{-1}X-\left(I_n-e^{-2tC^{-1}}\right)C^{-1}S\right),\psi_t\right\rangle_{2}\nonumber\\
    &=&\left\langle\left({I_n-e^{-2tC^{-1}}}\right)^{-\frac{1}{2}}\left(e^{-2tC^{-1}}C^{-1}S-\bigg({I_n-e^{-2tC^{-1}}}\bigg)^{\frac{1}{2}}e^{-tC^{-1}}C^{-1}X\right),\psi_t\right\rangle_{2}\nonumber\\
    &=&\left\langle\left({I_n-e^{-2tC^{-1}}}\right)^{-\frac{1}{2}}\left(e^{-2tC^{-1}}J_S^*\bigg[(1_{\mathcal{A}}\otimes1_{\mathcal{A}})\otimes I_n\bigg]-\bigg({I_n-e^{-2tC^{-1}}}\bigg)^{\frac{1}{2}}e^{-tC^{-1}}\mathcal{J}_S^*(A.D_t)\right),\psi_t\right\rangle_{2}\nonumber\\
    &=&\left\langle\left({I_n-e^{-2tC^{-1}}}\right)^{-\frac{1}{2}} \bigg(\mathcal{J}_S^*\bigg[E_{2t}-B_{2t}.A.D_{2t}\bigg]\bigg),\psi_t\right\rangle_{2} \nonumber\\
    &\leq& \left\lVert\left({I_n-e^{-2tC^{-1}}}\right)^{-\frac{1}{2}}\right\rVert_{op}\left\lVert \mathcal{J}_S^*\bigg[E_{2t}-B_{2t}.A.D_{2t}\bigg] \right\rVert_{2} \Phi^*(X_t|V_C)^{\frac{1}{2}}\nonumber\\
    &\leq& \left\lVert\left({I_n-e^{-2tC^{-1}}}\right)^{-\frac{1}{2}}\right\rVert_{op}\lVert C^{-1}\rVert_{op}^{\frac{1}{2}}\bigg\lVert E_{2t}-B_{2t}.A.D_{2t} \bigg\rVert_{HS}\Phi^*(X_t|V_C)^{\frac{1}{2}}\nonumber
\end{eqnarray}

Now we use that $\bigg(I_n-e^{-2tC^{-1}}\bigg)^{-\frac{1}{2}}$ is symmetric (because $C$ is symmetric). Moreover, it is straightforward to see that the operator norm of this matrix is equal to :
\begin{equation}
    \bigg\lVert\left(I_n-e^{-2tC^{-1}}\right)^{-\frac{1}{2}}\bigg\rVert_{op}=\frac{1}{\sqrt{1-e^{-2t\lVert C\rVert_{op}^{-1}}}}
\end{equation}
And that 
\begin{equation}
\bigg\lVert E_{2t}-B_{2t}.A.D_{2t} \bigg\rVert_{HS}\leq \bigg\lVert e^{-2tC^{-1}}\bigg\rVert_{op}\bigg\lVert (1_{\mathcal{A}}\otimes 1_{\mathcal{A}})\otimes I_n-A\bigg\rVert_{HS}
\end{equation}
\bigbreak
Since for all $t>0$, $B_{t},D_{t},E_{t}$ commutes, we see that:
\begin{equation}
    B_{2t}^{-1}E_{2t}D_{2t}^{-1}=(1_{\mathcal{A}}\otimes1_{\mathcal{A}})\otimes I_n
\end{equation}
By using that for $t>0$: $\lVert e^{-2tC^{-1}}\rVert_{op}=e^{-2t\lVert C\rVert_{op}^{-1}}$, the conclusion follows directly by minimizing over $A$.}
\end{proof}
\begin{remark}
When the covariance matrix is a homothety, it is easier to achieve the proof because in the proposition \ref{pro67}, we have that for $D$, a homothety in $M_n(\mathbb{C})$, $(1_{\mathcal{A}}\otimes 1_{\mathcal{A}})\otimes D$ belongs to the center of $M_n(\mathcal{A}\otimes \mathcal{A}^{op})$, and thus one can simplify formula (\ref{expec}).
\end{remark}
\begin{theorem}
For a potential $V_C$ where the matrix $C$ is symmetric definite positive, we have the following HSI inequality. 
\bigbreak
Let $X=(X_1,\ldots,X_n)$ a n-tuple of non-commutative random variables (self-adjoint operator in $(\mathcal{A},\tau)$, then 
\begin{equation}
    -\chi^{*}_{V_C}(X_1,...,X_n)\leq \frac{\lVert C\rVert_{op}\lVert C^{-1}\rVert_{op}}{2}\Sigma^*(X|V_C)^2 log\left(1+\frac{\Phi^*(X|V_C)}{\lVert C^{-1}\rVert_{op}\Sigma^*(X|V_C)^2}\right),
\end{equation}
\begin{proof}
The proof is obtained by the same arguments as in \cite{FN} Theorem 2.6.
Indeed, we begin with the following de-Bruijn´s identity, we assume that $\Phi^*(X|V_C)<\infty$ (if not there is nothing to prove). We denote also $X_t$ the interpolation given in lemma \ref{lma3} where $X$ is the vector of initial conditions :
\begin{eqnarray}
    -\chi^{*}_{V_C}(X_1,...,X_n)&=&\int_0^{\infty}\Phi^*(X_{2t}|V_C)dt\nonumber\\
    &\leq& \frac{1}{2}\Phi^*(X|V_C)\int_0^{u}e^{-2t\lVert C\rVert_{op}^{-1}}dt +\lVert C^{-1}\rVert_{op}{\Sigma^*(X|V_C)}^2\int_u^{\infty}\frac{e^{-4t\lVert C\rVert_{op}^{-1}}}{{1-e^{-2t\lVert C\rVert_{op}^{-1}}}}dt\nonumber\\
    &\leq& \frac{\lVert C\rVert_{op}}{2}\Phi^*(X|V_C)(1-e^{-2u\lVert C\rVert_{op}^{-1}})\nonumber\\
    &+&\frac{\lVert C\rVert_{op}\lVert C^{-1}\rVert_{op}}{2}{\Sigma^*(X|V_C)}^2\left(-{e^{-2u\lVert C\rVert_{op}^{-1}}}-log\left(1-e^{-2u\lVert C\rVert_{op}^{-1}}\right)\right)\nonumber\\
\end{eqnarray}
Then by optimizing in $r=1-e^{-2u\lVert C\rVert_{op}^{-1}}$, we get :
\begin{equation}
    1-e^{-2u\lVert C\rVert_{op}^{-1}}=\frac{\lVert C^{-1}\rVert_{op}\Sigma^*(X|V_C)^2}{\Phi^*(X|V_C)+\lVert C^{-1}\rVert_{op}\Sigma^*(X|V_C)^2}
\end{equation}
By replacing this value in the previous inequality, we get the desired conclusion.
\end{proof}
\end{theorem} 
 Note that this inequality improves the {\it LSI} inequality. This last one gives exactly the same {\it HSI} inequality found in \cite{FN}, because the product of the operator norm vanishes since the covariance matrix is a positive homothety and moreover the two entropies relative to the potential $V_{\rho^{-1}I_n}$ aggres. We also remark that we have the appearance of a quantity usually called condition number $k(A)=\lVert A\rVert_{op}\lVert A^{-1}\rVert_{op}$ in numerical analysis. This should be compared to the results exposed in \cite{LNP}, where a modified Stein discrepancy is defined in another way and not directly linked to the usual Stein discrepancy for approximations.
\bigbreak
Now, we finish the section by stating a transport inequality between the non commutative Wasserstein distance and the free Stein discrepancy.
\begin{lemma}
For any self-adjoint vector $X=(x_1,\ldots,x_n)$ in $(\mathcal{A},\tau)$, we have :
\begin{equation}
    d_W(X,S)\leq \lVert C\rVert_{op}\lVert C^{-1}\rVert_{op}^\frac{1}{2} \Sigma(X|V_C),
\end{equation}
\end{lemma}
\begin{proof}
By the previous Theorem \ref{Th2} where $(X_t)_{t\geq0}$ is the interpolation given in lemma \ref{lma3}, we obtain that $\frac{d^{+}}{dt} W(X_t,S)\leq \Phi(X_t|V_C)^{\frac{1}{2}}$, then by using the previous inequality between free Fisher information and Stein discrepancy, it follows by a simple integration:

Where we have used that :
\begin{equation}
    \int_0^{\infty}\frac{e^{-2t\lVert C\rVert_{op}^{-1}}}{\sqrt{1-e^{-2t\lVert C\rVert_{op}^{-1}}}}=\lVert C\rVert_{op}
\end{equation}
\end{proof}

\begin{remark}
One can actually prove a Talagrand transport inequality for a convex potential $V$ bounded from below in the sense of (\ref{conv}) by $c(1\otimes 1)\otimes I_n$ for $\chi^*_V$ by the same arguments, it is done by Dabrowski (in his PhD Thesis \cite{Dabth}) and it leads to the following inequality where we denote the (only) free Gibbs state with convex potential $V$ with law $\tau_V$:
\begin{equation}
    d_W(X,\tau_V)\leq\sqrt{-\frac{2}{c}\chi_V^*(X)},
\end{equation}
A shorter way to reach the conclusion is to consider $g(\epsilon)=d_W(X(t+\epsilon),\tau_V)-(-\frac{2}{c}\chi^*_V(X))^{\frac{1}{2}}$,
by using LSI, one can show that $\frac{d}{d\epsilon}g(\epsilon)\leq0$, and then achieved the proof.
\end{remark}
\begin{remark}
Note that the two approach will provide exactly the same bounds for the for the multidimensional semicircular approximations on the Wigner space. In fact, we only want here to provide an upper bound for the quadratic Wasserstein distance in terms of the fourth free cumulants.
\end{remark}
\section{Wigner-Ito chaos}
In this section, we will describe the main tools to prove a multivariate quantitative limit theorem for a tuple whose belongs to some homogeneous Wigner chaos.

\begin{flushleft}
In the sequel, we will denote $S$, a family of joint centered semicircular variables in some tracial $W^{*}$-probability space $(\mathcal{A}, \tau)$ equipped with a faithful normal tracial state. For any integer $n\geq 0$, we denote by $\mathcal{P}_n$ the Wigner chaos of order $n$, that is to say the Hilbert space in $L^2(\mathcal{A},\tau)$ generated by the set $\left\{1_\mathcal{A}\right\}$ $\bigcup \left\{S_1\ldots S_k: 1 \leq k \leq n,S_1,...,S_k \in S\right\}$. 
\end{flushleft}
\begin{flushleft}
We also define the homogeneous Wigner chaos of order $n$ by:
\begin{equation}
    \mathcal{H}_n:=\mathcal{P}_n \cap \mathcal{P}_{n-1}^{\perp},
\end{equation}
and we obviously have:
\begin{equation}
    \mathcal{P}_n=\bigoplus_{k=0}^{n}\mathcal{H}_k,
\end{equation}
\end{flushleft}
\begin{remark}
Note here that the homogeneous chaos of order "$0$" is the complex linear span of $1_{\mathcal{A}}$ and the chaos of order "$1$" is the Hilbert Space generated by $S$, which is assumed to be an infinite separable Hilbert space. Now by the classification of infinite dimensional separable real Hilbert space, we know that the restrictions to real elements (self-adjoints), and we will suppose that $S$ is in this way, implies that it is isometrically isomorphic to $L^2_{\mathbb{R}}(\mathbb{R}_{+})$. We then deduce that $S=\left\{S(h) : h\in L^2_{\mathbb{R}}(\mathbb{R}_{+})\right\}$, and that the map:
$h \mapsto S(h)$ is an isomorphism from $L^2_{\mathbb{R}}(\mathbb{R}_+)$ to $S$.
\end{remark}
\begin{flushleft}
It turns out that the construction of Wigner chaoses could be done efficiently by using an isomorphism between the free Fock space and the Wigner space: such a construction is usually done via the Wick map or Wick product. We refer to the article \cite{BS} for a complete exposure.
\end{flushleft}
In fact, if we denote $\mathcal{H}$, the complexified of $L^2_\mathbb{R}(\mathbb{R}_+)$ to $S$, by setting $\mathcal{H}^{\otimes 0}=\mathbb{C}$, and the projection $\pi_n$ from $L^2(\mathcal{A},\tau)$ to $\mathcal{H}_n$:
\begin{equation}
    I_n : h_1\otimes...\otimes h_n \mapsto \pi_n(S(h_1)...S(h_n)),
\end{equation}
which can be extended to a linear isometry between $\mathcal{H}^{\otimes n}$ to $\mathcal{H}_n$.
\begin{flushleft}

An important lemma is the following, called the {\it Haagerup} inequality, or more precisely its semicircular version (proved by Haagerup for length functions over words in length $n$, $f: F_{\infty}\rightarrow C^*_{red}(F_{\infty})$), which implies that our multiple integrals constructed to be in  $L^2(\mathcal{A},\tau)$ are in fact $\mathcal{A}$ and more precisely they belong to the $C^*$-algebra generated by $\left\{S_t\right\}_{t\geq 0}$ (theorem 5.3.4 of \cite{BS}). This inequality shows that one has a control of the operator norm of a multiple Wigner integral by its $L^2$ norm:
\end{flushleft}
\begin{theorem}(Haagerup, Biane Speicher \cite{BS}).
Let $n\geq 0$, then for all $F \in \mathcal{H}_n$, we have :
\begin{equation}
    \lVert F \rVert_{\mathcal{A}} \leq (n+1)\lVert F \rVert_{L^2(\mathcal{A},\tau)},
\end{equation}
\end{theorem}
The following proposition ensures that the multiple Wiener-Itô integral behaves well with respect to the product. Indeed elements in some homogeneous Wigner chaos are bounded operators by the previous lemma, and so, we are allowed to multiply them and obtain an element in $\mathcal{A}$. In fact, this formula provides a linearization property for the product of two multiple Wigner integrals. 
\begin{flushleft}
Before stating the result, we begin with the definition of (nested) contractions.
\end{flushleft}
\begin{definition}
Let $f \in L^2(\mathbb{R}^n_+)$ and $g \in L^2(\mathbb{R}^m_+)$, for every $0\leq p\leq n\wedge m$, we define the contraction of order $p$ of $f$ and $g$ as the element of $L^2(\mathbb{R}^{n+m-2p}_+)$ by:
\begin{equation}
    f\stackrel{p}{\frown} g(t_1,...,t_{n+m-2p})=\int_{\mathbb{R}_+^{p}}f(t_1,...,t_{n-p},s_p,...,s_1)g(s_1,...,s_p,t_{n-p+1}...,t_{n+m-2p})d_{s_1}...d_{s_p}
\end{equation}
\end{definition}
\begin{prop}(Biane, Speicher prop 5.3.3 in \cite{BS}) \label{pprodui}
For all $f \in L^2(\mathbb{R}^n_+)$ and $g \in L^2(\mathbb{R}^m_+)$, we have:
\begin{equation}
    I_n(f)I_m(g)=\sum_{p=0}^{n\wedge m}I_{n+m-2p}(f\stackrel{p}{\frown} g),
\end{equation}
\end{prop}
\begin{flushleft}
In particular for all $n,m\geq 0$ : \begin{equation}
\tau(I_n(f)^*I_m(g))=\delta_{n,m}\langle g,f\rangle_{L^2(\mathbb{R}_+^n)}
\end{equation}
\end{flushleft}
\begin{remark}
Given a function $f \in L^2(\mathbb{R}^n_+)$, the adjoint of this function is defined by:
\begin{equation}
    f^*(t_1,...,t_n)=\overline{f(t_n,...,t_1)}\nonumber
\end{equation}
which ensure that $I_n(f)^*=I_n(f^*)$. We then easily deduce that $I_n(f)$ is self-adjoint if and if only $f=f^*$. Such functions are usually called mirror-symmetric (see \cite{KNPS}).
\end{remark}
The contractions are really important in the context of Wigner chaos and Wiener chaos since they appear in the computation of moments of a chaotic random variable.

\begin{lemma}(Theorem 1.6 of \cite{KNPS}) :\label{lma5}
Let $f \in L^2(\mathbb{R}^n_+)$, then:
\begin{equation}
    \tau(|I_n(f)|^4)=2\lVert f\rVert^2_{L^2(\mathbb{R}^{n})}+\sum_{p=1}^{n-1}\lVert f\stackrel{p}{\frown} f^* \rVert_{L^2(\mathbb{R}^{n+m-2p}_+)}^2
\end{equation}

\end{lemma}
\bigbreak
A generalized notion of contractions exist also trough the context of pairings, 
we denote $\mathcal{P}_2(n_1\otimes\ldots\otimes n_r)$ the set of pairings $\pi \in \mathcal{P}_2(n_1+...+n_r)$ such that no blocs of $\pi$ contains more than one element from each interval set:
\begin{equation}
    \left\{1\ldots n_r\right\},\left\{n_1+1\ldots n_1+n_2\right\},\ldots,\left\{n_1+\ldots +n_{r-1},\ldots, n_1+\ldots+n_r\right\}\nonumber
\end{equation}
\bigbreak
An integral relative to such pairings is constructed by setting for all
$f_1 \in L^2(\mathbb{R}^{n_1}_+),\ldots f_r ,\in L^2(\mathbb{R}^{n_r}_+)$:
\begin{eqnarray}
\int_{\pi}f_1\otimes \ldots \otimes f_r=\int_{\mathbb{R}^{\frac{n_1+\ldots n_r}{2}}}(f_1\otimes\ldots\otimes f_r)(t_1,...,t_{n_1+...+n_r})\prod_{\left\{i,j\right\}\in \pi}d_{t_i}\nonumber
\end{eqnarray}
and $t_i$ is identified to $t_j$ when $i\overset{\pi}{\sim} j$ (belongs to the same block).
\bigbreak
We finish this section with following proposition which is found in \cite{KNPS} :
\begin{prop}(Lemma 2.1 of \cite{KNPS})
For all $f_1 \in L^2(\mathbb{R}^{n_1}_+),\ldots f_r \in L^2(\mathbb{R}^{n_r}_+)$, we have the following inequality :
\begin{equation}
   \left\lvert\int_{\pi}f_1\otimes \ldots \otimes f_r\right\lvert \leq \lVert f_1\rVert_{L^2(\mathbb{R}^{n_1}_+)}\ldots \lVert f_r\rVert_{L^2(\mathbb{R}^{n_r}_+)}  
\end{equation}
\end{prop}
\section{The free Malliavin calculus}

The classical Malliavin operators have a free counterpart in the context of free probability due to the work of Biane and Speicher in \cite {BS}. This construction could be done efficiently on the Free Fock space and then transferred onto the algebra of field operator by the identification $X\mapsto X\Omega$ where $\Omega$ denotes the "vaccum" vector and the $*$-unital algebra generated by the field operators. Since, we are in presence of a closable operator, we will denote as usual the completion of this unital algebra with respect to the $L^2$-norm which is sufficient for our purpose. For sake of clarity, we will also assume standard identifications of spaces as usual in the Malliavin calculus.
\begin{definition}
The free Malliavin derivative is the unique unbounded closable operator (valued into the square integrable biprocesses $\mathcal{B}_2$):
\begin{eqnarray}
    \nabla &: &L^2(\mathcal{A},\tau) \rightarrow L^2(\mathbb{R}_+,L^2(\mathcal{A},\tau)\bar{\otimes} L^2(\mathcal{A},\tau))\nonumber\\
    &&A\mapsto \nabla A=(\nabla_t A)_{t\geq 0}
\end{eqnarray}
such that for all $h\in L^2_{\mathbb{R}}(\mathbb{R}_+)$, $\nabla(S(h))=h.1_{\mathcal{A}}\otimes 1_{\mathcal{A}}$, and that, for all $A,B \in S_{alg}$ (where $S_{alg}$ is the unital $*$-algebra generated by $\left\{S(h),h\in L^2_{\mathbb{R}}(\mathbb{R}_+)\right\}$, we have the derivation property $\nabla(AB)=A.\nabla B+ \nabla A.B$ where the left and right actions are given by the multiplication on the left leg and opposite multiplication on the right leg.
\end{definition}
\begin{definition}
We denote $dom(\nabla)$, the domain of $\nabla$, defined as the completion of the unital $*$-algebra generated by $\left\{S(h),h\in L^2_{\mathbb{R}}(\mathbb{R}_+)\right\}$, with respect to the following norm,:
\begin{equation}
    \lVert Y\rVert_{1,2}^2=\lVert Y\rVert_2^2+\lVert \nabla Y\rVert_{\mathcal{B}_2}^2.
\end{equation}
\end{definition}

\begin{prop}(Proposition 5.3.10 of \cite{BS})
The domain of $\nabla$ contains $\mathcal{P}_n$, and the restriction of $\nabla$ to this space is a bounded linear operator.
\end{prop}
\begin{flushleft}
We can also explicit the action of $\nabla$ on $\mathcal{P}_n$. First, we note that for any $n,m\geq 0$, the map $I_n\otimes I_m$ from ${L^2(\mathbb{R}^{n}_+)\otimes L^2(\mathbb{R}^{m}_+)}$ to $\mathcal{H}_n\otimes \mathcal{H}_m$.
By the isomorphism between ${L^2(\mathbb{R}^{n+m}_+)}$ and ${L^2(\mathbb{R}^{n}_+)\otimes L^2(\mathbb{R}^{m}_+)}$, we can see the linear extension of the map:
\begin{eqnarray}
    I_n\otimes I_m &:& L^2(\mathbb{R}^{n+m}_+) \rightarrow \mathcal{H}_n\otimes \mathcal{H}_n\nonumber\\
    &&h_1\otimes...\otimes h_{n+m} \mapsto \pi_n(S(h_1)...S(h_n))\otimes\pi_m(S(h_{n+1})...S(h_m))
\end{eqnarray}
\end{flushleft}
For all $A\otimes B \in \mathcal{P}_n \otimes \mathcal{P}_n$ and $B\otimes C \in \mathcal{P}_n \otimes \mathcal{P}_n$,we denote :
$(A\otimes B)^{*}=A^*\otimes B^*$ and $(A\otimes B)\sharp (C\otimes D)=AC\otimes DB$, and we extend them by linearity for the first map which turns out to provide a map from $(\mathcal{P}_n \otimes \mathcal{P}_n)$ to $\mathcal{P}_n \otimes \mathcal{P}_n$, and by bilinearity and continuity for the second one which gives a map from $(\mathcal{P}_n \otimes \mathcal{P}_n)^2$ to $P_{2n} \otimes P_{2n}$.
\begin{flushleft}

Since the state is tracial, it implies moreover (the proof is checked on elementary tensors of the form $A_1\otimes A_2$ and $B_1\otimes B_2$ and then extend by linearity and density) :
for all $A,B \in (\mathcal{P}_n\otimes \mathcal{P}_n)^2$ :
\begin{equation}
    \langle A,B\rangle_{L^2(\mathcal{A},\tau)\otimes L^2(\mathcal{A},\tau)}=\tau\otimes\tau(B\sharp A^{*}),\nonumber
\end{equation}
We define for all $f\in L^2(\mathbb{R}^{n}_+)$, the function $f_t^k \in L^2(\mathbb{R}^{n-1}_+)$ by the equality for almost all $t\geq 0$ :
\begin{equation}
    f(t_1,...,t,t_{k+1},...,t_n)=f_t^k(t_1,..,t_{k-1},t_{k+1},...,t_n)\nonumber
\end{equation}
\end{flushleft}
\begin{prop}\label{pp3}(Proposition 5.3.9 of \cite{BS})
The Malliavin derivative maps $\mathcal{P}_n$ into $L^2(\mathbb{R},\mathcal{P}_n\otimes \mathcal{P}_n)$, indeed for $f \in L^2(\mathbb{R}^{n}_+)$, then for almost all $t\geq 0$ :

\begin{equation}
    \nabla_t(I_n(f))=\sum_{k=1}^n I_{k-1}\otimes I_{n-k}(f_t^k),
\end{equation}
\end{prop}

\begin{flushleft}
We now state the chain rule of the free Malliavin derivative:
\end{flushleft}

\begin{prop}\label{pp5}
For all $F\in \mathcal{P}_d$, and for all $P\in \mathbb{C}[X]$
\begin{equation}
    \nabla_t(P(F))=(\partial P(F))\sharp \nabla_t(F),
\end{equation}
and for the multivariate case, we have for all $(F_1,...,F_n)$ with each $ F_i\in \mathcal{P}_d$ and $P \in \mathds{P}$ :
\begin{equation}
    \nabla_t (P(F))=\sum_{k=1}^n (\partial_k P(F))\sharp \nabla_t F_k,
\end{equation}
\end{prop}
\begin{theorem}(Biane, Speicher, prop 5.3.12 in\cite{BS})
Let $F\in dom(\nabla)$, then we have the following Clark-Ocone-Bismut formula:
\begin{equation}
    F=\tau(F)+\delta(\Gamma\circ \nabla F)
\end{equation}
where $\Gamma$ denotes the orthogonal projection from $\mathcal{B}_{2}$ onto the square integrable adapted biprocesses $\mathcal{B}_{2}^a$ and $\delta$ the Skorohod integral (the adjoint of the Malliavin derivative). 
\end{theorem}
Note also that there exists a free Poincaré inequality on the Wigner space (which have never been mentioned or used) and which can be seen as the infinite dimensional analog of the one's proved by Voiculescu in an unpublished note (see Dabrowski \cite{Dab10}, lemma 2.2).
\begin{theorem}
Let $F\in dom(\nabla)$, then we have the free Poincaré inequality:
\begin{equation}
    \lVert F-\tau(F).1\rVert_2^2\leq \lVert \nabla F\rVert_{\mathcal{B}_2}^2
\end{equation}
\end{theorem}
The proof is straightforward via the chaotic decomposition or the free Clark-Ocone-Bismut formula.

\section{Discrepancy for Wigner chaos}
In this section, we will construct a free Stein kernel with the help of two linear maps. The main idea of this other construction of a free Stein kernel with respect to the semicircular potentials is based on the idea that we have to avoid the use of the inverse of the Ornstein-Uhlenbeck operator (number operator). In fact, by using duality, one can find a free Stein kernel, that involves this last one, but as mentioned by Bourguin and Campese (and in the one dimensional case), one cannot control the associated free Stein discrepancy by a quantity involving the \textit{fourth free cumulant} for all self-adjoints multiple Wigner integrals, but for a smaller class of these last one whose kernel are "\textit{fully-symmetric}".
\bigbreak
\begin{flushleft}
For all $n\geq 0$, we will denote $\tau\otimes id$ : $\mathcal{P}_n\otimes \mathcal{P}_n \rightarrow \mathcal{P}_n$ and $id\otimes \tau$ : $\mathcal{P}_n\otimes \mathcal{P}_n \rightarrow \mathcal{P}_n$ and we define them by : $\tau\otimes id(A\otimes B)=\tau(A).B$ and $id\otimes \tau(A\otimes B)=\tau(B).A$, for all $A,B \in \mathcal{P}_n$.
\end{flushleft}
The following lemma proved by Cébron (lemma 3.9 of \cite{C}) is the crucial idea to avoid the use of the free Ornstein-Uhlenbeck operator and construct a new free Stein kernel.
\begin{lemma}
For all $A,B \in \mathcal{P}_n$, such as $\tau(A)=0$ or $\tau(B)=0$ :
\begin{equation}
\tau(AB)=\tau\left(\int_{\mathbb{R}_+} id\otimes\tau(\nabla_t A).(\tau\otimes id(\nabla_t B))dt\right),
\end{equation}
\end{lemma}
In particular, we can obtain a new corollary which is the following :
\begin{lemma}
Consider a n-tuple $F=(F_1,...,F_n) $ of centered self-adjoint element in some $\mathcal{P}_d$ (with $d$ fixed).
For each $P \in \mathds{P}$ and for each $i=1,...,n$  we have :
\begin{equation}
    \tau(P(F)F_i)=\sum_{k=1}^n \tau\otimes\tau\left(\partial_k(P(F))\sharp\int_0^\infty \nabla_t F_k\sharp(\tau\otimes id(\nabla_t F_i)\otimes 1_{\mathcal{A}})dt  \right)
\end{equation}
\begin{proof}
The proof follows exactly the arguments of Cébron in addition to the Chain-rule of Malliavin derivative for evaluation of non-commutative polynomials in several variables, i.e for all $P \in \mathds{P}$, $\nabla_t P(F)=\sum_{k=1}^n \partial_k(P(F))\sharp \nabla_t F_k$. 
\end{proof}
\end{lemma}
\bigbreak
This lemma will be extremely useful to compute a Stein kernel in our setting.
\begin{theorem}
Let $F=(F_1,...,F_n) $ be a n-tuple of self-adjoint element in $\mathcal{P}_d$, such as $\tau(F)=(0,\ldots,0)$ (this can also be weakened by the assumption $\tau([DV_C](F)=\tau(C^{-1}F)=(0,\ldots,0)$), then 
\begin{equation}
    A=\left(\sum_{j=1}^n C_{i,j}^{-1}\int_{\mathbb{R}_+} (id\otimes\tau(\nabla_t F_j)).(\nabla_t F_k)^* dt\right)_{i,k=1}^{n},
\end{equation}
 is a free Stein kernel of $F$ with respect to the potential $V_C$ and it belongs to $M_n(\mathcal{P}_{2d}\otimes \mathcal{P}_{2d})$.
\end{theorem}
\begin{proof}
{\it 
By the previous lemma, it suffice to compute for all $P \in \mathds{P}$ $\tau(P(F)(C^{-1}F)_i)$ for each $i=1,...,n$, by linearity (of the linear operator $X\mapsto C^{-1}X$, we easily see that it is equal to :
\begin{equation}
    \sum_{k=1}^n\sum_{j=1}^n \tau\otimes\tau\left(\partial_k(P(F))\sharp\int_{\mathbb{R}_+} C_{i,j}^{-1}\nabla_t F_k\sharp(\tau\otimes id(\nabla_t F_j)\otimes 1_{\mathcal{A}})dt  \right)\nonumber
\end{equation}
then by the definition of the inner product on $M_n(L^2(\mathcal{A}\otimes \mathcal{A},(\tau\otimes\tau)\circ Tr)$, we obtain the desired conclusion.
It could also be achieved via the use of the Skorohod operator. Indeed, it is simply obtained by using that for any $G$ in some $\mathcal{P}_d$: $\delta((id\otimes\tau(\nabla G))\otimes 1_{\mathcal{A}})=G$, and then again the use of linearity of the Skorohod integral.}
\end{proof}
The free Stein kernel that we have constructed $A$ is related to the following object: the free Malliavin-Stein matrix which is the free counterpart of the well known Malliavin-Stein matrix defined in the classical case for a $n$-tuple $(F_1,...,F_n)$ in $L^2(\Omega,\mathcal{F},\mathds{P})$ where the filtration $\mathcal{F}$ is generated by a isonormal Gaussian process over a real separable Hilbert space $\mathcal{H}$ by :
$\Gamma(F)=\bigg(\langle DF_j,-DL^{-1}F_i\rangle_{\mathcal{H}}\bigg)_{i,j=1}^n$ and where $L^{-1}$ is the pseudo-inverse of the Ornstein-Uhlenbeck operator.
\begin{definition}
We define the free Malliavin-Stein matrix valued in $M_n(\mathcal{A}\otimes \mathcal{A})$ of a $n$-tuple $(F_1,...,F_n)$ in $\mathcal{P}_d$ with $d$ an bounded integer (in particular, they belong to the domain of $Dom(\nabla)$) as :
\begin{equation}
    \Gamma(F)=\bigg(\int_{\mathbb{R}_+} (id\otimes\tau(\nabla_t F_i)).(\nabla_t F_j)^* dt\bigg)_{i,j=1}^{n}\in M_n(\mathcal{P}_{2d}\otimes\mathcal{P}_{2d})
\end{equation}
Note that we will view it in the following as an element of $M_n(\mathcal{A}\otimes\mathcal{A}^{op})$

\end{definition}
From this definition, we can see that our free Stein kernel $A$ could be expressed as :
\begin{equation*}
    A=\left((1_{\mathcal{A}}\otimes1_{\mathcal{A}})\otimes C^{-1}\right).\Gamma(F)
\end{equation*}
where $F$ is a n-tuple of elements in the non-homogeneous Wiener chaos of any bounded order and "$.$" is the usual matrix product for element in $M_n(\mathcal{A}\otimes\mathcal{A}^{op})$.
\begin{flushleft}
We can, in fact, obtain a much more powerful result, as we can compute explicitly a free Stein kernel on the Wigner space for every potential (for sake of clarity, we will assume that the potential is self-adjoint).
\end{flushleft}
\begin{theorem}
Let $F=(F_1,...,F_n) $  a n-tuple of self-adjoints element in some $\mathcal{P}_d$ ($d$ fixed positive integer), $V$ a oself-adjoint polynomial (i.e $V=V^*$) or even a formal power series, such as  $\tau([DV](F))=(0,\ldots,0)$, alors:
\begin{equation}
    \Gamma_V(F)=\bigg(\int_{\mathbb{R}_+} (id\otimes\tau(\nabla_t [D_iV](F))).(\nabla_t(F))^* dt\bigg)_{i,j=1}^{n}\in M_n(\mathcal{P}_{deg(V)d}\otimes\mathcal{P}_{deg(V)d})
, \end{equation}

 is a free Stein kernel for $F$ relatively to the potential $V$.
\end{theorem}
\begin{proof}
Again, the result is easily achieved, knowing that the free Malliavin gradient and the Skorohod integral are linked trough the following relation, for any $G\in \mathcal{P}_d$ :
\begin{equation}
    \delta((id\otimes \tau)(\nabla G)\otimes 1_{\mathcal{A}})=G
, \end{equation}
in particular, by setting $(G_i=[D_iV](F))_{i=1}^n$, we obtain by duality for $P=(P_1,\ldots,P_n)\in \mathds{P}^n$ :
\begin{eqnarray}
    \tau(P_i(F)[D_iV](F))&=&\tau(P_i(F)\delta((\tau\otimes Id)(\nabla ([D_iV](F))\otimes 1_{\mathcal{A}}))\nonumber \\ &=&\sum_{j=1}^n\tau\otimes\tau\left([\partial_jP_i](F)\sharp\int_0^{\infty}\nabla_t(F_j)\sharp(\tau\otimes id(\nabla_t[D_iV](F))\otimes 1_{\mathcal{A}})dt\right)\nonumber
, \end{eqnarray}
it is then easy to conclude since the inner product on $M_n(L^2(\mathcal{A}\otimes\mathcal{A},\tau\otimes\tau\circ Tr)$ is given $\tau\otimes\tau\circ Tr(A^*\sharp B)$, and knowing for $G\in dom(\nabla)$ self-adjoint $(\tau\otimes id(\nabla_t(G)))^{*}=id\otimes \tau(\nabla_t(G))$ (see, lemma 4.6 in \cite{Mai}, or Cébron, corollary 3.10 in \cite{C}, from where our conclusion follows exactly in the same way).
\end{proof}
\begin{remark}
We would like to insist on the previous definition of Malliavin-Stein matrix, since it is well known in the classical case that the almost sure invertibility of the Malliavin matrix ensure the absolute continuity with respect to the Lebesgue measure of a random vector belonging to the Malliavin derivative domain. In free probability, a recent and deep result of Mai in \cite{Mai} shows that the distribution of non-trivial Wigner chaos (and more generally finite sum) cannot have atoms. These results were provided by analyzing the problem in an algebraic setting: the absence of atoms is equivalent to the absence of zero-divisors which should "{\it survive}" under certain operations build on directional gradients which are the link between free Malliavin calculus and the theory of noncommutative derivatives (especially the ones which satisfies a coassociatity relation), and thus giving it the contraction by iterating the procedure (see \cite{Mai} for complete exposure). We see in particular that the Malliavin derivative and the operator $\tau\otimes Id$ (which is crucial to lowers the degree of a polynomial evaluation) play an important role to deduce the result. It might be of independent interest to study this object, to maybe deduce regularity properties of distributions.
\bigbreak
Another potential application, which is well known in the classical case, is the possibility to compute recursively the cumulants of a vector of random variables which are Malliavin differentiable via the Malliavin-Stein matrix. It would be very interesting to study this free counterpart.
\end{remark}
\bigbreak

The goal is now to compute a sharp bound for the discrepancy for a tuple with elements living in homogeneous Wigner chaos, one can see that we have the following bound :
\begin{eqnarray}
    \bigg\lVert A-(1_{\mathcal{A}}\otimes1_{\mathcal{A}})\otimes I_n \bigg\rVert_{HS}&=&\bigg\lVert\left((1_{\mathcal{A}}\otimes1_{\mathcal{A}})\otimes C^{-1}\right).\Gamma(F)-(1_{\mathcal{A}}\otimes1_{\mathcal{A}})\otimes I_n \bigg\rVert_{HS}\nonumber\\
    &=&\bigg\lVert(1_{\mathcal{A}}\otimes1_{\mathcal{A}})\otimes C^{-1}.\bigg(\Gamma(F)-(1_{\mathcal{A}}\otimes1_{\mathcal{A}})\otimes C\bigg)\bigg\rVert_{HS}\nonumber\\
    &\leq& \lVert C^{-1}\rVert_{op}\bigg\lVert\Gamma(F)-(1_{\mathcal{A}}\otimes1_{\mathcal{A}})\otimes C\bigg\rVert_{HS}\label{61}
\end{eqnarray}
\bigbreak

\begin{theorem}\label{Theorem 8}

Consider a n-tuple $F=(F_1,...,F_n)=(I_{q_1}(f_1),...,I_{q_n}(f_n))$ of self-adjoint element in $(\mathcal{A},\tau)$, with $f_i \in L^2(\mathbb{R}_+^{q_i})$ \textit{mirror-symmetric}. 
\newline
Let $C$ be a symmetric definite positive matrice $\in M_d(\mathbb{R}) $, which is such that $C_{i,j}=\tau(F_iF_j)$ (note that the definite positive assumption is essential here), and denote $d=\underset{i=1,\ldots,n}{\max}q_i$.
\bigbreak 
We Set : 
\begin{equation}
    M(F)=\psi(\tau(F_1^4)-2\tau(F_1^2)^2,\tau(F_1^2),...,\tau(F_n^4)-2\tau(F_n^2)^2,\tau(F_n^2))
\end{equation}
with $\psi : (\mathbb{R}\times \mathbb{R}_+)^d \rightarrow \mathbb{R}$:
\begin{eqnarray}
    \psi(x_1,y_1,...,x_n,y_n)=
 \sum_{j,k=1}^n\mathds{1}_{q_k= q_j} q_k^{\frac{3}{4}}min\left(\lvert x_k \rvert^{\frac{1}{4}}y_j^{\frac{1}{2}},\lvert x_j \rvert^{\frac{1}{4}}y_k^{\frac{1}{2}}\right)\nonumber\\
+ \sum_{j,k=1}^n \mathds{1}_{q_k\neq q_j}(q_k\vee q_j)^{\frac{3}{4}}min\left(\lvert x_k\rvert^{\frac{1}{4}}y_j^{\frac{1}{2}},\lvert x_j\rvert^{\frac{1}{4}}x_k^{\frac{1}{2}}\right)
\end{eqnarray}

Then, we have that the non-commutative quadratic Wasserstein distance between $F$ and $S$ is bounded as follows :
\begin{equation}
    d_W(F,S)\leq  \lVert C\rVert_{op}^{\frac{1}{2}} \lVert C^{-1}\rVert_{op}M(F)
\end{equation}
\end{theorem}

\begin{flushleft}
Before giving the proof, we begin with a lemma which is the free counterpart of the classical commutative estimates proved by Nourdin and Peccati in the classical case (see section 6.2 of \cite{NP-book}):
\end{flushleft}

\begin{lemma}\label{lma8}
Let $F=I_p(f)\in (\mathcal{A},\tau)$ and $G=I_q(g)\in (\mathcal{A},\tau)$ with $f \in L^2(\mathbb{R}_{+}^p) $ and $g \in L^2(\mathbb{R}_{+}^q)$ assumed to be mirror-symmetric functions and let $a \in \mathbb{R}$.
\newline
\bigbreak
If $p=q$,we have :
\begin{eqnarray}
\left\lVert\int_{\mathbb{R}_+}(id\otimes \tau)(\nabla_t I_p(f)).(\nabla_t I_q(g))^*dt-a.1_\mathcal{A}\otimes1_\mathcal{A}\right\rVert_{L^2(\mathcal{A},\tau)\otimes L^2(\mathcal{A},\tau)}^2
&\leq&
(a-\langle f,g\rangle_{L^2(\mathbb{R}_{+}^p)})^2
\nonumber\\&+&\left(\sum_{m=1}^{p-1}\sum_{l=0}^{m-1} min\left(A_{f,g}^{p,l},A_{g,f}^{p+1,m}\right)\right)\nonumber\\&+&\left(\sum_{l=0}^{p-2} min(A_{f,g}^{p,l},A_{g,f}^{p+1,p})\right)\nonumber
\end{eqnarray}
\bigbreak
If $p < q$, we have :
\begin{eqnarray}
\left\lVert\int_{\mathbb{R}_+}((id\otimes \tau)(\nabla_t I_p(f)).(\nabla_t I_q(g))^*-a.1_\mathcal{A}\otimes1_\mathcal{A}\right\rVert_{L^2(\mathcal{A},\tau)\otimes L^2(\mathcal{A},\tau)}^2
&\leq&
a^2+\lVert f\rVert_{L^2(\mathbb{R}_{+}^p)}^2\lVert g\stackrel{q-p}{\frown}g\rVert_{L^2(\mathbb{R}_{+}^{2p})}
\nonumber\\&+&\left(\sum_{m=1, m\neq p}^{q}\sum_{l=0}^{p\wedge m-1}min\left(A_{f,g}^{p,l},A_{g,f}^{q+1,m}\right)\right)
\nonumber\\&+&\left(\sum_{l=0}^{p-2} min\left(A_{f,g}^{p,l},A_{g,f}^{q+1,p}\right)\right)\nonumber
\end{eqnarray}
Where we have defined for every $p,q \geq 1$, for $f \in L^2(\mathbb{R}_{+}^p) $ and $g \in L^2(\mathbb{R}_{+}^q)$ and for each $l\leq p\wedge q$, the quantity:
\begin{eqnarray}
A_{f,g}^{p,l}=\left\lVert f
\stackrel{p-l-1}{\frown} f\right\rVert_{L^2(\mathbb{R}_{+}^{2l+2})}\lVert g \rVert_{L^2(\mathbb{R}_{+}^{q})}^2
\end{eqnarray}
\end{lemma}

\newpage
Before giving the proof of this lemma, let us introduce some notations. Indeed, for $F=I_n(f)$ with $f=f^*$, we define $\tilde f_t^k \in L^2(\mathbb{R}_{+}^{n-1})$ (for almost all $t\geq 0$) by :
\begin{equation}
    \overline{f(t_1,...,t_{k-1},t,t_{k+1},...,t_n)}=\tilde f_t^k(t_{k-1},t_{k-2},...,t_1,t_n,...,t_{k+1})\nonumber
\end{equation}
and we remark easily that the following equality holds true for almost all $t\geq 0$ and for all $f\in L^2(\mathbb{R}_{+}^{n})$ (firstly on elementary functions, then by linearity and a density argument, it extends to the whole space),
\begin{equation}
    \left(I_{k-1}\otimes I_{n-k}(f_t^k)\right)^*=I_{k-1}\otimes I_{n-k}(\tilde f_t^k)\nonumber
\end{equation}

\begin{proof}
{\it 
\begin{eqnarray}
   (id\otimes \tau)(\nabla_t I_p(f)).(\nabla_t I_q(g))^*&=&\left(id\otimes \tau\left(\sum_{n=1}^p I_{n-1}\otimes I_{p-n}(f_t^n)\right)\right).\left(\sum_{m=1}^q I_{m-1}\otimes I_{q-m}(g_t^m)\right)^*\nonumber\\
    &=& I_{p-1}(f_t^p).\left(\sum_{m=1}^q I_{m-1}\otimes I_{q-m}(\Tilde{g}_t^m)\right)\nonumber\\
    &=&\left(\sum_{m=1}^q\sum_{l=0}^{p \wedge m-1} I_{p-1+m-1-2l}\otimes I_{q-m}(f_t^p\stackrel{l}{\frown} \Tilde{g}_t^m)\right)\label{3m-1}
\end{eqnarray}
Where the last equality is verified for elementary tensors, that is for functions $f=f_1\otimes...\otimes f_p$ and $g=g_1\otimes...\otimes g_q$, with each $f_i,g_j\in L^2_{\mathbb{R}}(\mathbb{R}_+)$ and then extended by linearity and density.

\bigbreak
We can then distinguish the two cases of our lemma, in the first one, the two chaos have the same order, we compute this quantity in the following way to get:
\begin{eqnarray}
\left(\sum_{m=1}^p\sum_{l=0}^{m-1} I_{p-1+m-1-2l}\otimes I_{p-m}(f_t^n\stackrel{l}{\frown} \Tilde{g}_t^m)\right)&=&\left(\sum_{m=1}^{p-1}\sum_{l=0}^{m-1} I_{p-1+m-1-2l}\otimes I_{p-m}(f_t^p
\stackrel{l}{\frown} \Tilde{g}_t^m)\right)\nonumber\\
&+&\left(\sum_{l=0}^{p-2} I_{2p-2l-2}\otimes 1_{\mathcal{A}}(f_t^p\stackrel{l}{\frown} \Tilde{g}_t^p)\right)\nonumber\\
&+&f_t^p\stackrel{p-1}{\frown} \Tilde{g}_t^p.1_{\mathcal{A}}\otimes 1_{\mathcal{A}}
\end{eqnarray}
\bigbreak
Where we separated the sum for $m=p$ and $l=p-1$, because it leads to the constant
\newline
$f_t^p\stackrel{p-1}{\frown} \tilde g_t^p.1_{\mathcal{A}}\otimes1_{\mathcal{A}} $ and since :
\begin{equation}
    f_t^p\stackrel{p-1}{\frown} \tilde g_t^p=\int_{\mathbb{R}^{p-1}_+}f(t_1,...,t_{p-1},t)g(t,t_{p-1},...,t_1)dt_1...dt_{p-1}\nonumber
\end{equation}
Now, by the assumption of mirror symmetries of $f,g$, then by integrating over "$t$", we have: 
\begin{equation}
    \langle f,g\rangle_{L^2(\mathbb{R}_{+}^{p})}.1_{A}\otimes 1_{A}=\tau(FG).1_{A}\otimes 1_{A}\nonumber
\end{equation}
\bigbreak
Now we use the Wigner bi-isometry to get:
\begin{eqnarray}
\left\lVert\int_{\mathbb{R}_+}(id\otimes \tau)(\nabla_t I_p(f)).(\nabla_t I_q(g))^*dt-a.1_\mathcal{A}\otimes1_\mathcal{A}\right\rVert_{L^2(A,\tau)\otimes L^2(A,\tau)}^2\nonumber\\
=\Bigg\lVert \sum_{m=1}^{p-1}\sum_{l=0}^{m-1} \int_{\mathbb{R}_+}I_{p-1+m-1-2l}\otimes I_{p-m}(f_t^p
\stackrel{l}{\frown} \Tilde{g}_t^m)dt
+\sum_{l=0}^{p-2} \int_{\mathbb{R}_+}I_{2p-2l-2}\otimes 1_\mathcal{A}(f_t^p\stackrel{l}{\frown} \Tilde{g}_t^p)dt\nonumber\\ +(\tau(I_p(f)I_p(g))-a).1_\mathcal{A}\otimes 1_\mathcal{A} \Bigg\rVert_{L^2(A,\tau)\otimes L^2(A,\tau)}^2\nonumber\\
=(\tau(I_p(f)I_p(g))-a)^2+\sum_{m=1}^{p-1}\sum_{l=0}^{m-1}\left\lVert\int_{\mathbb{R}_+}(f_t^p
\stackrel{l}{\frown} \Tilde{g}_t^m)dt\right\rVert^2_{L^2(\mathbb{R}_{+}^{2p-2l-2})}+\sum_{l=0}^{p-2} \left\rVert\int_{\mathbb{R}_+}(f_t^p
\stackrel{l}{\frown} \Tilde{g}_t^p)dt\right\rVert^2_{L^2(\mathbb{R}_{+}^{2p-2l-2})}\nonumber
\end{eqnarray}
\bigbreak
Now we compute at $m,l$ fixed 
\begin{eqnarray}
    f_t^p\stackrel{l}{\frown} \Tilde{g}_t^m(t_1,...,t_{p-l-1},r_1,...,r_{p-l-1})\nonumber\\
    =\int_{\mathbb{R}^l_+}f_t^p(t_1,...,t_{p-l-1},s_l,...,s_1)\tilde g_t^m(s_1,...,s_l,r_1,...,r_{p-l-1})ds_1...ds_l\nonumber\\
    =\int_{\mathbb{R}^l_+}f(t_1,...,t_{p-l-1},s_l,...,s_1,t)\overline{g(r_{m-l-1},...,r_1,s_l,...,s_1,t,r_{p-l-1},...,r_{m-l})} ds_1...ds_l\nonumber
\end{eqnarray}

\bigbreak
We can then infer that :
\begin{eqnarray}
\left\lVert\int_{\mathbb{R}_+}(f_t^p
\stackrel{l}{\frown} \Tilde{g}_t^m)dt\right\rVert^2_{L^2(\mathbb{R}_{+}^{2p-2l-2})}\nonumber\\
=\int_{L^2(\mathbb{R}_{+}^{2p})}f(t_1,...,t_{p-l-1},s_l,...,s_1,t)\overline{g(r_{m-l-1},...,r_1,s_l,...,s_1,t,r_{p-l-1},...,r_{m-l})}\nonumber\\
\overline{f(t_1,...,t_{p-l-1},s_l^{'},...,s_1^{'},t^{'})}g(r_{m-l-1},...,r_1,s_l^{'},...,s_1^{'},t,r_{p-l-1},...,r_{m-l}) ds_1...ds_ldtds_1^{'}...ds_l^{'}\nonumber
\\dt^{'}dt_1...dt_{p-l-1}dr_1...dr_{p-l-1}\nonumber
\end{eqnarray}

Remind that the functions $f,g$ are mirror-symmetric, and thus we obtain:
\begin{eqnarray}
\left\lVert\int_{\mathbb{R}_+}(f_t^p
\stackrel{l}{\frown} \Tilde{g}_t^m)dt\right\rVert^2_{L^2(\mathbb{R}_{+}^{2p-2l-2})}\nonumber\\
=\int_{L^2(\mathbb{R}_{+}^{2p})}f(t_1,...,t_{p-l-1},s_l,...,s_1,t){g(r_{m-l},...,r_{p-l-1},t,s_1,...,s_l,r_1,,...,r_{m-l-1}})\nonumber\\{f(t^{'},...,s_1^{'},...,s_l^{'},t_{p-l-1},...,t_1)}
g(r_{m-l-1},...,r_1,s_l^{'},...,s_1^{'},t^{'},r_{p-l-1},...,r_{m-l}) ds_1...ds_ld_tds_1^{'}...ds_l^{'}\nonumber
\\dt^{'}dt_1...dt_{p-l-1}dr_1...dt_{r-l-1}\nonumber
\end{eqnarray}
\newpage
This expression seems complicated to be handled, to deduce the result, the idea is to integrate over $dt_1,...,dt_{p-l-1}$ or $dr_{m-l},...,dr_{p-l-1}$ to make appear a contraction of $f$ and itself or $g$ and itself (modulo a product of norms of $f,g$).
\bigbreak
Indeed, when we integrate over $dt_1,...,dt_{p-l-1}$, it is equal to the quantity :
\begin{equation}
    \int_{\pi}(f\stackrel{p-l-1}{\frown} f)\otimes g\otimes g\nonumber
\end{equation}
for some pairing $\pi \in \mathcal{P}_2((2l+2)\otimes p\otimes p)$.
\bigbreak

Now we use the Proposition \ref{pp3}, to get :
\begin{eqnarray}
\left\lVert\int_{\mathbb{R}_+}(f_t^p
\stackrel{l}{\frown} \Tilde{g}_t^m)dt\right\rVert^2_{L^2(\mathbb{R}_{+}^{2p-2l-2})}
&\leq&\left\lvert\int_{\pi}(f\stackrel{p-l-1}{\frown} f)\otimes g\otimes g\right\rvert\nonumber\\
&\leq &\lVert f\stackrel{p-l-1}{\frown} f \rVert_{L^2(\mathbb{R}_{+}^{2l+2})}\lVert g\rVert_{L^2(\mathbb{R}_{+}^{p})}\lVert g\rVert_{L^2(\mathbb{R}_{+}^{p})}\nonumber
\end{eqnarray}
\bigbreak

And when we integrate over $dr_{m-l},...,dr_{p-l-1}$, we have:
\begin{equation}
    \int_{\pi}(g\stackrel{p-m}{\frown} g)\otimes f\otimes f\nonumber
\end{equation}
for some pairing $\pi \in \mathcal{P}_2(2m\otimes p\otimes p)$
and we deduce that :
\begin{eqnarray}
\left\lVert\int_{\mathbb{R}_+}(f_t^p
\stackrel{l}{\frown} \Tilde{g}_t^m)dt\right\rVert^2_{L^2(\mathbb{R}_{+}^{2p-2l-2})}
&\leq&\left\lvert\int_{\pi}(g\stackrel{p-m}{\frown} g)\otimes f\otimes f\right\rvert\nonumber\\
&\leq &\lVert g\stackrel{p-m}{\frown} g \rVert_{L^2(\mathbb{R}_{+}^{2l+2})}\lVert f\rVert_{L^2(\mathbb{R}_{+}^{p})}\lVert f\rVert_{L^2(\mathbb{R}_{+}^{p})}
\end{eqnarray}
\bigbreak

For the other case, when the order of chaos are distinct (we can always suppose that $p < q$), we have separated the sum when $m=p$ and $l=p-1$, to get :
\begin{eqnarray}
\left(\sum_{m=1 }^q\sum_{l=0}^{p\wedge m-1} I_{p-1+m-1-2l}\otimes I_{q-m}(f_t^n\stackrel{l}{\frown} \Tilde{g}_t^m)\right)
&=&\left(\sum_{m=1,m\neq p}^q\sum_{l=0}^{p \wedge m-1} I_{p-1+m-1-2l}\otimes I_{q-m}(f_t^p\stackrel{l}{\frown}\Tilde{g}_t^m)\right)\nonumber\\
&+&\left(\sum_{l=0}^{p-2} I_{2p-2l-2}\otimes I_{q-p}(f_t^p\stackrel{l}{\frown} \Tilde{g}_t^m)\right)\nonumber\\
&+& 1_{\mathcal{A}}\otimes I_{q-p}(f_t^p\stackrel{p-1}{\frown}\Tilde{g}_t^p)
\end{eqnarray}

From that, we infer again by using the Wigner bi-isometry that :
\begin{eqnarray}
\left\lVert\int_{\mathbb{R}_+}(id\otimes \tau)(\nabla_t I_p(f)).(\nabla_t I_q(g))^*dt-a.1_{\mathcal{A}}\otimes1_{\mathcal{A}}\right\rVert_{L^2(A,\tau)\otimes L^2(A,\tau)}^2\nonumber\\
=\Bigg\lVert \sum_{m=1,m\neq p}^{q}\sum_{l=0}^{p \wedge m-1} \int_{\mathbb{R}_+}I_{p-1+m-1-2l}\otimes I_{q-m}(f_t^p
\stackrel{l}{\frown} \Tilde{g}_t^m)dt
+\sum_{l=0}^{p-2} \int_{\mathbb{R}_+}I_{2p-2l-2}\otimes I_{q-p}(f_t^p\stackrel{l}{\frown} \Tilde{g}_t^p)dt\nonumber\\ +\int_{\mathbb{R}_+}1_{\mathcal{A}}\otimes I_{q-p}(f_t^p\stackrel{p-1}{\frown}\Tilde{g}_t^p)dt -a.1_{\mathcal{A}}\otimes 1_{\mathcal{A}} \Bigg\rVert_{L^2(A,\tau)\otimes L^2(A,\tau)}^2\nonumber\\
=a^2+\left\lVert \int_{\mathbb{R}_+}(f_t^p\stackrel{p-1}{\frown}\Tilde{g}_t^p)dt\right\rVert_{L^2(\mathbb{R}_{+}^{p-q})}^2 +\sum_{m=1,m\neq p}^{q}\sum_{l=0}^{p\wedge m-1}\left\lVert\int_{\mathbb{R}_+}(f_t^p
\stackrel{l}{\frown} \Tilde{g}_t^m)dt\right\rVert^2_{L^2(\mathbb{R}_{+}^{p+q-2l-2})}\nonumber\\
+\sum_{l=0}^{p-2} \left\rVert\int_{\mathbb{R}_+}(f_t^p
\stackrel{l}{\frown}\Tilde{g}_t^p)dt\right\rVert^2_{L^2(\mathbb{R}_{+}^{p+q-2l-2})}
\end{eqnarray}

We are left to analyse the following terms when $l,m$ are fixed :
\begin{eqnarray}
f_t^p \stackrel{l}{\frown} \Tilde{g}_t^m(t_1,...,t_{p-l-1},r_1,...,r_{q-l-1})\nonumber\\
    =\int_{\mathbb{R}^l_+}f_t^p(t_1,...,t_{p-l-1},s_l,...,s_1)\tilde g_t^m(s_1,...,s_l,r_1,...,r_{p-l-1})ds_1...ds_l\nonumber\\
    =\int_{\mathbb{R}^l_+}f(t_1,...,t_{p-l-1},s_l,...,s_1,t)\overline{g(r_{m-l-1},...,r_1,s_l,...,s_1,t,r_{q-l-1},...,r_{m-l})} ds_1...ds_l\nonumber
\end{eqnarray}
and then compute the following quantity :
\begin{eqnarray}
\left\lVert\int_{\mathbb{R}_+}(f_t^p
\stackrel{l}{\frown} \Tilde{g}_t^m)dt\right\rVert^2_{L^2(\mathbb{R}_{+}^{p+q-2l-2})}\nonumber\\
=\int_{L^2(\mathbb{R}_{+}^{p+q})}f(t_1,...,t_{p-l-1},s_l,...,s_1,t)\overline{g(r_{m-l-1},...,r_1,s_l,...,s_1,t,r_{q-l-1},...,r_{m-l})}\nonumber\\
\overline{f(t_1,...,t_{p-l-1},s_l^{'},...,s_1^{'},t^{'})}g(r_{m-l-1},...,r_1,s_l^{'},...,s_1^{'},t,r_{q-l-1},...,r_{m-l}) ds_1...ds_ldtds_1^{'}...ds_l^{'}\nonumber
\\dt^{'}dt_1...dt_{p-l-1}^{'}dr_1...dr_{q-l-1}
\end{eqnarray}

When $l\neq p-1$, the same bound provided before holds. 
\bigbreak
And for $m=p$ and $l=p-1$,
we are left to the quantity  :
\begin{eqnarray}
\left\lVert\int_{\mathbb{R}_+}(f_t^p
\stackrel{p-1}{\frown} \Tilde{g}_t^p)dt\right\rVert^2_{L^2(\mathbb{R}_{+}^{q-p})}\nonumber\\
=\int_{L^2(\mathbb{R}_{+}^{p+q})}f(s_{p-1},...,s_1,t){g(s_{p-1},...,s_1,t,r_{q-p},...,r_{1})}\nonumber\\
{f(t^{'},s_1^{'},...,s_{p-1}^{'})}g(r_{1},...,r_{q-p},t^{'},s_{p-1}^{'},...,s_1^{'}) ds_1...ds_{p-1}dtds_1^{'}...ds_{p-1}^{'}\nonumber
\\dt^{'}dr_1...dr_{q-p}
\end{eqnarray}

which is equal by integrating over $dr_{1},...,dr_{q-p}$, to:
\begin{equation}
    \int_{\pi}(g\stackrel{q-p}{\frown} g)\otimes f\otimes f\nonumber
\end{equation}
for some pairing $\pi \in \mathcal{P}_2(2p\otimes p\otimes p)$
\bigbreak
We then deduce that :
\begin{eqnarray}
\left\lVert\int_{\mathbb{R}_+}(f_t^p
\stackrel{p-1}{\frown} \Tilde{g}_t^p)dt\right\rVert^2_{L^2(\mathbb{R}_{+}^{p+q-2l-2})}
&\leq&\left\lvert\int_{\pi}(g\stackrel{q-p}{\frown} g)\otimes f\otimes f\right\rvert\nonumber\\
&\leq &\lVert g\stackrel{q-p}{\frown} g \rVert_{L^2(\mathbb{R}_{+}^{2p})}\lVert f\rVert_{L^2(\mathbb{R}_{+}^{p})}\lVert f\rVert_{L^2(\mathbb{R}_{+}^{p})}\nonumber
\end{eqnarray}}
\end{proof}
\qed

With these estimations, we are now in a position to give the proof of theorem \ref{Theorem 8}.
\begin{proof}(Theorem \ref{Theorem 8})\hypertarget{proof Theorem 8}
\bigbreak
{\it We know that \begin{equation}
    \lVert \Gamma(F)-(1_{\mathcal{A}}\otimes1_{\mathcal{A}})\otimes C \rVert_{HS}^2=\sum_{i,j=1}^n\lVert \Gamma(F)_{i,j}-C_{i,j}.(1_{\mathcal{A}}\otimes1_{\mathcal{A}})\rVert_{L^2(\mathcal{A}\otimes\mathcal{A},\tau\otimes\tau)}^2\nonumber\\
\end{equation}
From the lemma 8, to compute carefully the discrepancy, we will distinguish two cases: when the order of the two chaos is equal or different. This idea is similar from the original one of Nourdin and Peccati (section 6.2 of \cite{NP-book}. Indeed, the starting point of our investigations is that, if $F_i,F_j$ are multie Wigner integrals, then $\int_0^\infty  (id\otimes\tau(\nabla_t F_j)).(\nabla_t F_i)^* dt$ should be close to $C_{i,j}.1_{\mathcal{A}}\otimes 1_{\mathcal{A}}$ in $L^2(\mathcal{A}\otimes\mathcal{A},\tau\otimes\tau)$ in terms of the two and fourth free cumulants.
\bigbreak
We then analyse the term inside the sum by distinguishing when the order of chaos are equals or different:
\begin{eqnarray}
\Gamma(F)_{i,j}-C_{i,j}.(1_{\mathcal{A}}\otimes1_{\mathcal{A}})\nonumber
    &=&\left(\int_{\mathbb{R}_+}(id\otimes \tau)(\nabla_t I_{q_i}(f_i)).(\nabla_t I_{q_j}(f_j))^*dt\right)-C_{i,j}.(1_{\mathcal{A}}\otimes1_{\mathcal{A}})\nonumber\\
    &=&\mathds{1}_{q_i= q_j}\left(\int_{\mathbb{R}_+}(id\otimes \tau)(\nabla_t I_{q_i}(f_j)).(\nabla_t I_{q_j}(f_j))^*dt-C_{i,j}(1_{\mathcal{A}}\otimes1_{\mathcal{A}})\right)\nonumber\\
    &+&\mathds{1}_{q_i\neq q_j}\left(\int_{\mathbb{R}_+}(id\otimes \tau)(\nabla_t I_{q_i}(f_i)).(\nabla_t I_{q_j}(f_j))^*dt-C_{i,j}(1_{\mathcal{A}}\otimes1_{\mathcal{A}})\right)\nonumber\\
    &=&\mathds{1}_{q_i= q_j}\left(\int_{\mathbb{R}_+}(id\otimes \tau)(\nabla_t I_{q_i}(f_i)).(\nabla_t I_{q_j}(f_j))^*dt-C_{i,j}(1_{\mathcal{A}}\otimes1_{\mathcal{A}})\right)\nonumber\\
    &+&\mathds{1}_{q_i\neq q_j}\left(\int_{\mathbb{R}_+}(id\otimes \tau)(\nabla_t I_{q_i}(f_i)).(\nabla_t I_{q_j}(f_j))^*dt\right)\nonumber\\
\end{eqnarray}
Where we have used that in the second sum the correlation between $I_{q_i}$and $I_{q_j}$ vanishes since two Wigner-Itô multiple integrals of different orders are orthogonal (with respect to $L^2(\tau)$), when the orders are different, i.e necessarily: $C_{i,j}=0$ when $q_i\neq q_j$. 
\bigbreak
We could have been more specific by distinguishing the second terms when $q_i<q_j$ and $q_i>q_j$, however in this second case is straightforward to adapt the lemma \ref{lma8}, to obtain a same type of bound.
\bigbreak
Now, we recall by the Lemma \ref{lma5}, that for an homogeneous self-adjoint non commutative random variable $X=I_n(f)$, we have :
\begin{equation}\label{te}
    \sum_{i=1}^{n-1} \lVert f\stackrel{i}{\frown} f \rVert^2_{L^2(\mathbb{R}_{+}^{2n-2i})}\leq \tau(F^4)-2\tau(F^2)^2
\end{equation}

\bigbreak
And by using a simple Cauchy-Scharwz inequality, we get :
\begin{equation}
    \sum_{i=1}^{n-1} \lVert f\stackrel{i}{\frown} f \rVert_{L^2(\mathbb{R}_{+}^{2n-2i})}\leq \sqrt{n}\left(\tau(F^4)-2\tau(F^2)^2\right)^{\frac{1}{2}}\nonumber
\end{equation}
\bigbreak

\begin{remark}
Note that in the followings computations, we will only set $I_{q_i}$ instead of $I_{q_i}(f_{q_i})$, for readers convenience.
\end{remark}
\bigbreak
Then, we can  use the previous bound obtained in the lemma \ref{lma8} to obtain : 
\begin{eqnarray}
    \lVert \Gamma(F)_{i,j}-C_{i,j}.(1_{\mathcal{A}}\otimes1_{\mathcal{A}})\rVert_{L^2(\mathcal{A}\otimes\mathcal{A},\tau\otimes\tau)}^2\nonumber\\ \leq \Bigg[\mathds{1}_{q_i
    = q_j}\Bigg(\sum_{m=1}^{q_i-1}\sqrt{q_i} min\left(\bigg(\tau(I_{q_i}^4)-2\tau(I_{q_i}^2)^2\bigg)^{\frac{1}{2}}\tau(I_{q_j}^2),(\tau(I_{q_j}^4)-2\tau(I_{q_j}^2)^2) ^{\frac{1}{2}}\tau(I_{q_i}^2)\right)\nonumber\\
    +\sqrt{q_{i}}min\left(\bigg(\tau(I_{q_i}^4)-2\tau(I_{q_i}^2)\bigg)^{\frac{1}{2}}\tau(I_{q_j}^2),\bigg(\tau(I_{q_j}^4)-2\tau(I_{q_j}^2)^2\bigg)^{\frac{1}{2}}\tau(I_{q_j}^2)\right)\Bigg)\nonumber\\
    +\mathds{1}_{q_i\neq q_j}min\left(\bigg(\tau(I_{q_i}^4)-2\tau(I_{q_i}^2)^2\bigg)^{\frac{1}{2}}\tau(I_{q_i}^2),\bigg(\tau(I_{q_j}^4)-2\tau(I_{q_j}^2)^2\bigg)^{\frac{1}{2}}\tau(I_{q_i}^2)\right)\nonumber\\
    +\mathds{1}_{q_i\neq q_j}\sum_{m=1,m\neq q_i\wedge q_j}^{q_i\vee q_j-1}\sqrt{q_i\vee q_j}
    min\left(\bigg(\tau(I_{q_i}^4)-2\tau(I_{q_i}^2)^2\bigg)^\frac{1}{2}\tau(I_{q_j}^2),\bigg(\tau(I_{q_j}^4)-2\tau(I_{q_j}^2)^2\bigg)^\frac{1}{2}\tau(I_{q_i}^2)\right)\nonumber\\
    +\mathds{1}_{q_i\neq q_j}\sum_{l=0}^{q_i\wedge q_j-2}min\left(\bigg(\tau(I_{q_i}^4)-2\tau(I_{q_i}^2)^2\bigg)^{\frac{1}{2}}\tau(I_{q_j}^2),\bigg(\tau(I_{q_j}^4)-2\tau(I_{q_j}^2)^2\bigg)^{\frac{1}{2}}\tau(I_{q_i}^2)\right)\Bigg]\nonumber\\
    \leq \mathds{1}_{q_i= q_j} q_i^{\frac{3}{2}}min\bigg(\bigg(\tau(I_{q_i}^4)-2\tau(I_{q_i}^2)^2\bigg)^{\frac{1}{2}}\tau(I_{q_j}^2),\bigg(\tau(I_{q_j}^4)-2\tau(I_{q_j}^2)^2\bigg)^{\frac{1}{2}}\tau(I_{q_i}^2)\bigg)\bigg)\nonumber\\
    + \mathds{1}_{q_i\neq q_j}(q_i\vee q_j)^{\frac{3}{2}}min\left(\bigg(\tau(I_{q_i}^4)-2\tau(I_{q_i}^2)^2\bigg)^{\frac{1}{2}}\tau(I_{q_j}^2),\left(\tau(I_{q_j}^4)-2\tau(I_{q_j}^2)^2\right)^{\frac{1}{2}}\tau(I_{q_i}^2)\right)\nonumber\Bigg]
\end{eqnarray}
\bigbreak
It is important to notice that only one of the terms is non-zero since we have the indicators.

\bigbreak
Now to have a bound for the discrepancy, it remains to sum over $i,j$, use the inequality \ref{61} and take the square root of the quantity.
\bigbreak

And by using the inequality : $\sqrt{x_1+...+x_n}\leq \sqrt{x_1}+...+\sqrt{x_n}$, we have :

\begin{eqnarray}
\Sigma^*(X|V_C)&\leq& \lVert C^{-1}\rVert_{op} \Bigg[ \sum_{i,j=1}^n\mathds{1}_{q_i= q_j} q_j^{\frac{3}{4}}min\left(\bigg(\tau(I_{q_i}^4)-2\tau(I_{q_i}^2)^2\bigg)^{\frac{1}{4}}\tau(I_{q_j}^2)^{\frac{1}{2}},\bigg(\tau(I_{q_j}^4)-2\tau(I_{q_j}^2)^2\bigg)^{\frac{1}{4}}\tau(I_{q_i}^2)^{\frac{1}{2}}\right)\nonumber\\
&+& \sum_{i,j=1}^n \mathds{1}_{q_i\neq q_j}(q_i\vee q_j)^{\frac{3}{4}}min\left(\bigg(\tau(I_{q_i}^4)-2\tau(I_{q_i}^2)^2\bigg)^{\frac{1}{4}}\tau(I_{q_j}^2)^{\frac{1}{2}},\bigg(\tau(I_{q_j}^4)-2\tau(I_{q_j}^2)^2\bigg)^{\frac{1}{4}}\tau(I_{q_i}^2)^{\frac{1}{2}}\right)\Bigg]\nonumber
\end{eqnarray}}
\end{proof}
\qed
\begin{remark}
We can see that the right leg of (\ref{te}) is positive and it is a corollary of \cite{KNPS} which shows that the inequality is an equality only and if only the random variable is semicircular. In particular, any variable living in some homogeneous Wigner chaos of any order strictly greater than one cannot have a semicircular distribution. We may notice that for now, it is still an open question to know that two noncommutative random variables living in different homogeneous Wigner chaos can have or not the same distribution.
\end{remark}
\begin{flushleft}
Now, we return to the equivalence between componentwise convergence and joint convergence for sequence of self-adjoint vectors of multiple Wigner-Itô integrals. Our previous result will allow us to prove in an easier way the following theorem, whose was first proved by Nourdin, Peccati and Speicher (and which is the free analog of the Peccati, Tudor theorem \cite{PT}), which have been proved by analyzing carefully the contractions, which does appear in pairing integrals, as well as the a sophisticated analysis of the lattice of non-crossing partitions.
\end{flushleft}
\begin{theorem}(Theorem 1.3 in \cite{NPS})
Let $d\geq 2$ and $q_1,...,q_d$ be some fixed integers, and consider a positive definite symmetric matrix $C=\left\{C_{i,j}\right\}_{i,j=1}^d$. Let
$(S_1,...,S_d)$ be a semicircular family with covariance C.
For each $i=1,\ldots,d$, we consider a sequence $(f^{(i)}_k)_{k \in \mathbb{N}}$ of mirror-symmetric function in $L^2(\mathbb{R}_+^{q_i})$, such that for all $i,j=1,\ldots,d$: 
\begin{equation}
    \lim_{k\rightarrow\infty}\tau(I_{q_i}(f^{(i)}_k)I_{q_j}(f^{(j)}_k)=C_{i,j}
\end{equation}
then as $k\rightarrow\infty$ the following conditions are equivalent :
\begin{enumerate}
\item The vector $F_k=(I_{q_1}(f^{(1)}_k),...,I_{q_d}(f^{(d)}_k))$ converges in distribution to $(S_1,...,S_d)$. 
\item For each $i=1,\ldots,d$, the random variable $I_{q_i}(f^{(i)}_k)$ converges in distribution to $S_i$. \label{(b)}
\end{enumerate}
\end{theorem}
\begin{proof}
The first implication is trivial, and the reverse follows directly by our theorem \ref{Theorem 8}, since it implies that for every $i=1,\ldots,d$, $\lim_{k\rightarrow\infty}\tau(I_{q_i}(f^{(i)}_k)^4)=2C_{i,i}^2$ and, therefore  we obtain that $M(F_k)\rightarrow 0$.
\end{proof}
\qed

\section{Rate of Convergence in the multivariate free functional Breuer-Major CLT}

The main result of this section provides a quantitative bound for the (quadratic) non-commutative Wasserstein distance in the multivariate Breuer-Major free central limit theorem for free fractional Brownian motion, which was studied in section 5.1 in \cite{NT} (see also section 4 of \cite{BC} for the Berry-Essen bounds in the univariate case). The last result is stated for the weaker distance $d_{C_2}$, which turns out to be less or equal than the non-commutative Wasserstein distance (see Cébron \cite{C}). We will actually prove the theorem by analyzing carefully the contractions, which is fortunately done by previous authors: see \cite{NPR} section 4.1. Contrary to previous results, one indeed has two options which gives both the same bound. For the first one, we won't express the non-commutative fractional Brownian motion as a Wigner integral with respect to the free Brownian motion, we'd rather see it as an analogue of centered isonormal process: a centered semicircular process which exists for a large class of covariance functions. The construction is done via the free Fock Space (see \cite{BS} for further details and further properties). This approach seems interesting for the reason that there is a few references which developp free Malliavin calculus with respect to a general semicircular process, and that usual results in Malliavin calculus, and especially the probabilistic approximation part can easily be extended for more general setting than the the usual Brownian motion case, as many results which don't invoke stochastic integration can be easily adapted in this more general context. For reader's convenience, one also state the other approach base on the representation of the non-commutative fractional Brownian-motion as a Wigner integral with respect to a free Brownian motion.

\bigbreak
We remind some results about the non-commutative fractional Brownian motion with index $H\in (0,1]$ which is defined as the centered semicircular process with covariance given by :
\begin{equation}
    \tau(S_t^HS_s^H)=\frac{1}{2}(t^{2H}+s^{2H}-|t-s|^{2H})
\end{equation}

The orthogonal polynomials associated with the semicircular distribution are the Chebyshev polynomials (of the second kind) $\left(U_n\right)_{n\geq 0}$ which are defined for $x\in [-2,2]$ by the following recursive relation :
$U_0(x)=1,U_1(x)=x$ and for $n\geq 2$ :
\begin{equation}
   U_{n+1}(x)=xU_n(x)-U_{n-1}(x)
\end{equation}
\bigbreak
We define the discrete increment sequence of $(S_t)_{t\geq 0}$ by $\left\{X_k=S_{k+1}^H-S_{k}^H, k\geq 0\right\}$.
We also define the covariance function of the stationary sequence $(X_k)_{k\geq 0}$ is given by :
\begin{equation}
    \rho_H(x)=\frac{1}{2}(\lvert x+1\rvert^{2H}+\lvert x-1\rvert^{2H}-2\lvert x\rvert^{2H})
\end{equation}
We define also the sequence $\left\{V_n, n\geq 1\right\}$ by :
\begin{equation}
    V_n(t)=\frac{1}{\sigma\sqrt{n}}\sum_{k=0}^{\lfloor nt \rfloor-1}U_q(X_k)
\end{equation}

Where $\sigma=\sqrt{\sum_{r\in \mathbb{Z}}\rho^2(r)}$.
\bigbreak
Note also that the well-known construction of the Hilbert space associated with a centered isonormal Gaussian process in the commutative case, have a noncommutative counterpart. Indeed, we can construct an Hilbert space $\mathcal{H}$ associated with the non-commutative fractional Brownian motion, that is $\mathcal{H}$ is the completion of the space of elementary functions with respect to the following inner product where $X=\left\{X(h),h\in \mathcal{H}\right\}$ is a centered semicircular process such that:

\begin{equation}
    \langle \mathds{1}_{[0,t]},\mathds{1}_{[0,s]}\rangle_{\mathcal{H}}=\tau(X(\mathds{1}_{[0,t]})X(\mathds{1}_{[0,s]}))
\end{equation}

It is important to notice that we can again construct a non-commutative Malliavin calculus with respect to that process. All the details could be found in \cite{DS}, and it implies in particular that all the setting that we used before is true in a more general setting and based on a semicircular process (like the usual Malliavin calculus with respect to a isonormal process). The free Malliavin derivative, with respect to this process, will be denoted as $\nabla^X$
\bigbreak
This implies, that we can in particular write $S_t^H=X(\mathds{1}_{[0,t]})$ rather than writing the Fractional Brownian motion as a Wigner integral with respect to a free Brownian motion.
We also supposed constructed, in the same way, the Wigner-Ito chaos with respect to this process.
\begin{theorem}
Let $H<1-\frac{1}{2q}$, then for any fixed $d\geq 1$ and $0 = t_0 < t_1 < \ldots < t_d$ , there exists a constant c, (depending only on d, H and
$(t_0,t_1,...,t_d )$, and not on n) such that, for every $n \geq 1$:
\begin{equation}
    d_W\left(\frac{V_n(t_i)-V_n(t_{i-1})}{\sqrt{t_{i+1}-t_i}},S\right) \leq c \times \begin{cases} n ^{-\frac{1}{4}} \mbox{ if } H \in (0, \frac{1}{2}] \\
n^{\frac{H-1}{2}}, \mbox{ if } H \in \bigg(\frac{1}{2},\frac{2q-3}{2q-2}\bigg] \\
n ^{\frac{2qH-2q+1}{4}} \mbox{ if } H \in \bigg( \frac{2q-3}{2q-2}, \frac{2q-1}{2q}\bigg)
\end{cases}
\end{equation}

\end{theorem}
\begin{proof}
Let $d\geq 1$ and $0 = t_0 < t_1 < \ldots < t_d$ and a constant $c$ which can change from one line to another, it is easily seen that we can write (complete analogy with the Gaussian case).
If one choose to use directly the free Malliavin calculus with respect to the semicircular process $X$ (the semicircular process associated with the Hilbert space $\mathcal{H}$), one can write :
\begin{equation}
F_i=\frac{V_n(t_i)-V_n(t_{i-1})}{\sqrt{t_{i+1}-t_i}}=I_q^{X}(f_{i}^{(n)})
\end{equation}
where
\begin{equation}
    f_{i}^{(n)}=\frac{1}{\sigma \sqrt{n}\sqrt{t_{i+1}-t_i}}\sum_{k=\lfloor {nt_{i-1}}\rfloor}^{\lfloor{nt_{i}}\rfloor-1}\mathds{1}^{\otimes q}_{[k,k+1]}
\end{equation}

Now, it suffices to remark that in this setting, the free Stein kernel with respect to the normalized semicircular potential has also a nice form. The case of Wigner integrals with respect to the free Brownian motion being a particular case of this more general setting ($\mathcal{H}_{\mathbb{R}}=L^2_{\mathbb{R}}(\mathbb{R_+}$)):
\begin{equation}
    A=\bigg(\bigg\langle \tau\otimes id(\nabla^X(F_i)).\nabla^X(F_j)\bigg\rangle_{\mathcal{H}}\bigg)_{i,j=1}^n 
\end{equation}
where for for a process $U=h.u$ with $u \in (\mathcal{A},\tau)$, $h\in \mathcal{H}$ and a biprocess: $V=v\otimes g\otimes w$ with $v,w \in (\mathcal{A},\tau)$, and $g\in \mathcal{H}$, we have used the following linear extension of the pairing:
\begin{equation}
    \langle h.u,g.v\otimes w\rangle_{\mathcal{H}} =\langle h,g\rangle_{\mathcal{H}}. uv\otimes w
\end{equation}

\begin{flushleft}
Otherwise, one writes as usual, the {\it ncfBm} as a Wigner integral with respect to a free Brownian motion.
\newline
Indeed, it is well know that one can write for $S$ a free Brownian motion :
\begin{equation}
    S^H_t=\int_0^t  K^H(t,u)dS_u
\end{equation}
where $K^H$ is a covariance function whose expression can be found in \cite{LD}.
\end{flushleft}
\qed

\begin{flushleft}
All Wigner-Ito integrals are of the same order $q$, it is then necessary to evaluate the difference between the fourth moment and two, since we have $\tau(I_q(f_i^{n})^2)=1$, for all $i=1,\ldots,n$, we only have to estimate the difference between fourth moment and two, which turns out to be related to the contractions by Lemma 5,and this is fortunately done for the fractional Brownian motion (here it is exactly the same by fully-symmetry) in \cite{NPR}.
\end{flushleft}

\bigbreak
We have 
\begin{equation}
     \lVert f_{i}^{(n)}\otimes f_{i}^{(n)}\rVert_{\mathcal{H}^{\otimes 2(q-r)}}\leq c \begin{cases} n ^{-\frac{1}{2}} \mbox{ if } H \in (0, \frac{1}{2}]\\
n^{H-1}, \mbox{ if } H \in \bigg(\frac{1}{2},\frac{2q-3}{2q-2}\bigg] \\
n ^{\frac{2qH-2q+1}{2}} \mbox{ if } H \in \left( \frac{2q-3}{2q-2}, \frac{2q-1}{2q}\right)
\end{cases}
\end{equation}
\end{proof}
\section{Applications to q-Brownian chaos}
In this section, we will study the case of {\it $q$-Gaussian algebras} with $q\in(-1,1)$ discovered by Bos\.zjeko, Kümmerer and Speicher \cite{BKS}, and which are a continuous interpolation between fermionic, free and standard Brownian motion. Indeed, taking $q=0$, we find the usual free Fock space and letting $q$ close to $1$, we retrieve the usual Brownian motion.
Several properties of these algebras were studied, especially the factoriality, non-$\Gamma$ property or strong solidity...
\bigbreak
It is important to notice that we can actually construct a $q$-stochastic integration with respect to this process. 
Indeed, Donati-Martin in the fundamental work \cite{Don03} focused on {\it q-stochastic analysis} proved several important results in the infinite dimensional setting, that is when : $\mathcal{H}_{\mathbb{R}}=L^2_{\mathbb{R}}(\mathbb{R}_+)$, such as Ito integration for biprocesses, chaos decomposition and representation theorem. Recently, a non-commutative rough integration with respect to that process were constructed by Deya and Schott \cite{DS} which shows that contrary to the free case where the freeness allows deducing the $L^{\infty}$-version of the Burkhölder-David-Gundy inequality, for stochastic integral of biprocesses, we are unable for now to obtain such results for the $q$-Gaussian stochastic integrals of biprocesses.
\bigbreak
In another context, when the Fock space is built onto the finite dimensional Hilbert space, several authors such as Dabrowski, Shlyakhtenko, Nelson, Zeng in \cite{Dab14}, \cite{S03} or \cite{NZe}  were able to construct derivations which allow them to obtain conjugate variables of $q$-semicircular families with respect to the standard semicircular potential "$V_{I_n}$" provided that an operator is Hilbert-Schmidt. One can even obtain a more powerful result concerning the free transport. Indeed, Guionnet and Shlyakhtenko in the breaktrough paper \cite{GS2} focused on free monotone transport, obtained the isomorphism between the von Neumann algebras generated by {\it q-semicircular system} with $N$ (finite) generators and the free group factor with $N$ generators. This last result being true for $q$ small enough (the bound depending on $q,N$ and $q\rightarrow 0$ as $N\rightarrow \infty$). Recently, an important result was provided by Caspers \cite{Cas}, which has shown in the infinite dimensional setting, that is when $\mathcal{H}_{\mathbb{R}}$ is an infinite dimensional real separable Hilbert space, $\Gamma_0({\mathcal{H}_{\mathbb{R}}})$ and $\Gamma_q({\mathcal{H}_{\mathbb{R}}})$, for $-1<q<1$, $q\neq 0$, are non isomorphic.
\bigbreak
Nelson and Zeng in \cite{NZi} were able to generalize free monotone transport results in the infinite dimensional setting, but no longer for {\it q-Gaussian algebras}, rather for a deformation called {\it mixed q-Gaussian algebras} which depend on a infinite array with coefficients $q_{i,j} \in (-1,1)$. They proved that the {\it mixed q-von Neumann algebras} with infinite generators are isomorphic to the von Neumann algebra of the countably free group factor with infinite generators : $L(\mathbb{F}_{\infty})$, if the entries of the array are uniformly small with a rapid decay.
\bigbreak
We note that, contrary to the free Fock space and symmetric Fock space, constructing a {\it q-Malliavin calculus} is rather difficult. In fact, in a previous version of this work, we had the idea to construct it, however we encounter the real difficulty to build it, which is the appearance of a crucial operator denoted $\Xi_q$ which is never Hilbert-Schmidt in an infinite dimensional setting, that is when $\mathcal{H}=L^2_{\mathbb{R}}(\mathbb{R}_+)$ unless $q=0$ and, so it reduces to the Free Fock space. It would be rather interesting to be able to construct such Malliavin calculus for the {\it q-Fock space} to deduce further properties of the distributions of {\it q-Brownian chaos}, especially to prove analogues of powerful Mai´s result (\cite{Mai}), which states that the spectral measure of multiple Wigner integral does not have atoms. It is also of interest to show that the support of the distribution of any element in the homogeneous $q$-Brownian chaos is connected (in particular, it is equivalent to show that the $C^*$-algebra generated by the $q$-Brownian motions $C^*\left\{S_t^q,t\geq 0\right\}$ is projectionless, which is for now, far from our reach. In fact, Dabrowski in \cite{Dab14} proved the result in the finite-dimensional case for $q$ very small with $\lvert q\rvert<g(N)$). We leave it here for further investigations.
\bigbreak
Let $\mathcal{H}_{\mathbb{R}}$ be a real Hilbert Space and $\mathcal{H}_{\mathbb{C}}$ its complexification (in the sequel we will only focus on the finite dimensional case).
\bigbreak
We define $\mathcal{F}_q(\mathcal{H}_{\mathbb{C}})$ the $q$-Fock space as the completion of 
\begin{equation}
\mathcal{F}_{alg}:= \mathbb{C}\Omega\oplus\bigoplus_{i=1}^{k}\mathcal{H}_{\mathbb{C}}^{\otimes{n}},
\end{equation}
where $\Omega$ is a vacuum vector (dense and separating).
\bigbreak
with respect to the inner product :
\begin{equation}
\langle g_{1} \otimes ...\otimes g_{n},h_{1} \otimes ... \otimes h_{m} \rangle_{q}=\delta_{n,m}\sum_{\sigma \in S_n}q^{inv(\sigma)}\langle g_1,h_{\sigma(1)}\rangle...\langle g_n,h_{\sigma(n)}\rangle,
\end{equation}
where {\it inv} denotes the number of inversions of a permutation.
\bigbreak
We will take $\mathcal{H}_\mathbb{R}=\mathbb{R}^n$, and we will denote the following operators with are respectively the {\it left q-creation} and {\it left q-annihilation} operators associated with $h\in \mathcal{H}$ by :
\begin{equation}
    l(h)(g_1\otimes \ldots\otimes g_n)=h\otimes g_1\otimes \ldots\otimes g_n,
\end{equation}
and \begin{equation}
l^*(h)\Omega=0
\end{equation}
\begin{equation}
    l^*(h)(g_1\otimes \ldots\otimes g_n)=\sum_{k=1}^n q^{k-1}\langle h,g_1\rangle g_2\otimes\ldots\otimes \hat{g}_k\ldots\otimes  g_n,
\end{equation}
where  $\hat{.}$ denote the omission.
\newline 
And the right creation operator:
\begin{equation}
    r(h)(g_1\otimes \ldots\otimes g_n)= g_1\otimes \ldots\otimes g_n\otimes h,
\end{equation}
\bigbreak
One can define a state on $\Gamma_q(\mathcal{H}_{\mathbb{R}}):=vN\left\{l(h)+l^*(h),h\in\mathcal{H}_{\mathbb{R}}\right\}$, by setting :
\begin{equation}
    \tau(X):=\langle X\Omega,\Omega\rangle,
\end{equation}

\bigbreak
Then $x(h):=l(h)+l^*(h)$ is called a {\it q-semicircular operator}, the interpolation between a fermionic, $(0,1)$ semicircular and the standard gaussian.
\bigbreak
There also exist the $q$-counterpart of semicircular families. We will denote a {\it q-semicircular family}, as a family of elements in $\Gamma_q(\mathcal{H})$ such as : 
\begin{definition}
Let $n\geq 2$ be an integer, and let $C = (C_{i, j})_{i,j=1}^n$
be a positive definite symmetric matrix. A n-dimensional vector $(S_1, ..., S_n)$ of random variables in $(\mathcal{A},\tau)$ is said to be a q-semicircular family with covariance $C$, if $\forall n \in \mathbb{N}$, $\forall(i_1, ..., i_n) \in [n]=\left\{1,\ldots,n\right\}$ :

\begin{equation}
    \varphi(S_{i_1}S_{i_2}...S_{i_n})=\sum_{\pi \in \mathcal{P}_2[n]}q^{cr(\pi)}\prod_{\left\{a,b\right\}\in \pi}C_{i_a,i_b},
\end{equation}
Where $\mathcal{P}_2[n]$ is the set of all the pairings of $\left\{1,\ldots,n\right\}$ and $cr(\pi)$ denotes the number of crossings of $\pi$.
\end{definition}
As seen previously, these (centered) families are only determined by the set of covariance: $\left\{\tau(X_iX_j)/ i,j\in [n]\right\}$.
\bigbreak
Let's denote :
$
\Xi_q \in  \mathcal{B}(\mathcal{F}_q(\mathcal{H}))$ (the second quantization operator of $T=qId$):
\begin{equation}
\Xi_q=\sum_{N\geq 0} q^{N}P_{N}
\end{equation}where $P_{N}$ is the orthogonal projection on the tensors of rank $N$.
\bigbreak
We can remark that $\Xi_q$ is in fact an Hilbert-Schmidt operator when $q^2n<1$.
\bigbreak
In the following, we will fix a covariance matrix $C$ supposed to be symmetric definite positive and we let $(e_i)_{i=1}^n$ a set of $\mathcal{H}_{\mathbb{C}}$ such as $\langle e_i,e_j\rangle_{\mathcal{H}}:=C_{i,j}$ and we will denote $\left\{x(e_i)\right\}_{i=1}^n$ a set of {\it q-semicircular operator}.
\bigbreak
We see that \begin{equation}
    \tau(X_iX_j)=\langle X_iX_j\Omega,\Omega\rangle=\langle X_j\Omega,X_i\Omega\rangle=\langle e_i,e_j\rangle=C_{i,j},
\end{equation}
And it readily checked that this family is a {\it q-semicircular} family with covariance $C$.
\bigbreak
\bigbreak
Using the identification mentioned previously, when $q^2n<1$:
\begin{eqnarray}
    L^2(W^*(X)\otimes {W^*(X)}^{op})&\rightarrow& HS(\mathcal{F}_q)\nonumber\\
    a\otimes b^{op}&\mapsto &\langle .,b^*\Omega\rangle_{\mathcal{F}_q}a\Omega,\nonumber
\end{eqnarray}
This means that one can identify $\Xi_q$ as an element of $L^2(W^*(X)\otimes {W^*(X)}^{op})$.
\bigbreak
We can construct a free Stein kernel with respect to the potential $V_C$. It turns out that this construction was already done by Shlyakhtenko in \cite{S03} with respect to the potential $V_{I_n}$.
\begin{lemma}
Under the condition $q^2n<1$,
\begin{equation}
A=\Xi_q\otimes I_n,
\end{equation}
is a Free Stein kernel for $X$ with respect to the potential $V_C$.
\end{lemma}
Which is not surprising, since letting $q=0$, we would find the Schwinger-Dyson equation of a semicircular family of covariance $C$ which is:
$\langle C^{-1}S,S\rangle_2=\langle (1\otimes 1)\otimes I_n,[\mathcal{J}P](S)\rangle_{HS}$

\begin{proof}
We recall by lemma 3.1 in \cite{S03},
for $\lvert q\rvert<1$, and $g,h\in\mathcal{H}_{\mathbb{C}}$
\begin{eqnarray}
    [l(h),r(g)]=0\nonumber\\
     \end{eqnarray}
     and
    \begin{eqnarray}
    [l(h)^{*},r(g)]=\langle g,h\rangle \Xi_q
\end{eqnarray}
Now we deduce that 
\begin{eqnarray}
\left[\left(C^{-1}X\right)_i,r(e_j)\right]&=&\left[\sum_{k=1}^nC_{i,k}^{-1}X(e_k),r(e_j)\right]\nonumber\\
&=&\sum_{k=1}^nC_{i,k}^{-1}\left[X(e_k),r(e_j)\right]\nonumber\\
&=&\sum_{k=1}^nC_{i,k}^{-1}\langle e_k,e_j\rangle\Xi_q\nonumber\\
&=&\sum_{k=1}^nC_{i,k}^{-1}C_{k,j}
\Xi_q\nonumber\\
&=&\delta_{i,j}\Xi_q
\end{eqnarray}
Then we apply the proposition 2.6 in \cite{S03} to get the result.
\bigbreak
Thus, we are left to evaluate the free Stein discrepancy $\Sigma^*(X|V_C)$ by :
\begin{equation}
    \Sigma^*(X|V_C)=\lVert \Xi_q-(1\otimes 1^{op})\otimes I_n\rVert \leq \sqrt{n} \lVert \Xi_q-1\otimes 1^{op}\rVert_{L^2(\tau\otimes \tau^{op})}=\frac{\lvert q\rvert n}{\sqrt{1-q^2n}}
\end{equation}
Since $\lVert \Xi_q-P_0\rVert_{HS}^2=\frac{q^2n^2}{1-q^2n}$
\end{proof}
\qed
\bigbreak
Fathi and Nelson obtained this type of estimates in section 3 of \cite{FN}, but only for the potential $V_{I_n}$.
\begin{flushleft}
We can go even further in the construction of free Stein Kernel, that is we can in fact construct a free Stein Kernel with respect to the standard semicircular potential for every tuple of non commutative polynomial in $X_1,\ldots,X_n$ a standard $q$-semicircular family.
\newline
Before introducing the main theorem, let us introduce some preliminaries result (especially the construction of another derivation) due to Dabrowski \cite{Dab14}. This derivation is really important as it will provide the number operator as the corresponding generator of the Dirichlet form.
\end{flushleft}
\bigbreak
Firstly, by embedding $\mathcal{H}_{\mathbb{R}}$ as $\mathcal{H}_{\mathbb{R}}\oplus 0$ in $\mathcal{H}_{\mathbb{R}}\oplus\mathcal{H}_{\mathbb{R}}$, $L^2(\Gamma_q(\mathcal{H}_{\mathbb{R}}\oplus\mathcal{H}_{\mathbb{R}})$) is a normal Hilbert $\Gamma_q(\mathcal{H}_{\mathbb{R}})$ bimodule.
\begin{definition}
We define the derivation $\delta_q: L^2(\mathcal{H}_{\mathbb{R}})\rightarrow L^2(\Gamma_q(\mathcal{H}_{\mathbb{R}}\oplus\mathcal{H}_{\mathbb{R}}))$, such that
\newline
$\delta_q(x(h))=0\oplus x(h)$, which satisfies the derivation property.
\end{definition}
By using the identification given by $X\mapsto X\Omega$ between $L^2(\Gamma_q(\mathcal{H}_{\mathbb{R}}))$ and $\mathcal{F}_q(\mathcal{H}_{\mathbb{C}})$,
one can compute explicitly the action of $\delta_q$ on $\mathcal{F}_q(\mathcal{H}_{\mathbb{C}})$
as :
\begin{equation}
    \delta_q(f_1\otimes\ldots\otimes f_n)=\sum_{k=1}^n(f_1\oplus 0)\otimes\ldots (f_{k-1}\oplus 0)\otimes (0\oplus f_k)\otimes(f_{k+1}\oplus 0)\otimes\ldots(f_n\oplus 0),
\end{equation}
Now the essential idea (which is always fulfilled in the free case and it is at the basis on free Malliavin calculus since this derivation is well defined even in the infinite dimensional case), is that under some restriction (depending on $q,n$ and never satisfied in a infinite dimensional setting for $q\neq 0$), is (up to invertibility) that one can see this derivation valued into a sub-bimodule of the coarse correspondence as formulated and proved in the following proposition of Dabrowski.
\begin{prop}(Dabrowski prop.30 in \cite{Dab14})
Let suppose that $q^2n<1$ and that $\Xi_q$ is invertible $\Gamma_{q}(\mathcal{H}_{\mathbb{R}})\bar{\otimes }\Gamma_q(\mathcal{H}_{\mathbb{R}})^{op}$, (the precise conditions can be found in the corollary 29 of \cite{Dab14}, e.g when $q\sqrt{n}<0.13$), then $\delta_q$ is a closable derivation with $\delta_q^*\delta_q=\Delta$ where $\Delta$ is the number operator and moreover $\delta_q$ is seen as valued into a sub-bimodule of the coarse correspondence $L^2(\Gamma_q(\mathcal{H}_{\mathbb{R}}))\bar{\otimes}L^2(\Gamma_q(\mathcal{H}_{\mathbb{R}})^{op})$.
\end{prop}
We are now in position to construct a free Stein kernel with respect to the semicircular potential and not only for q-semicircular systems, that is we can consider non commutative polynomials in the $q$-semicirculars operators.
\begin{flushleft}
We also denote $\Delta^{-1}$ the pseudo-inverse of the number operator (as the number operator is self-adjoint with spectrum, $sp(\Delta)=\mathbb{N}$) which acts on centered random variable in $L^2_0(\Gamma_q(\mathcal{H}_{\mathcal{R}}))$, as $\Delta\Delta^{-1}F=F-\tau(F)$ (and the two operators commute).
\end{flushleft}
We can state now our main construction and we assume the previous identifications
\begin{theorem}
Let $F_1,\ldots,F_p$ self-adjoint elements in $\bigoplus_{k=0}^d Ker(\Delta-kId)$ with "$d$" a positive bounded integer, then :
\begin{equation}
A=\left(\delta_q(\Delta^{-1}F_i)\sharp (\delta_q(F_j))^*\right)_{i,j=1}^p,
\end{equation}
is a free Stein kernel with respect to the semicircular potential.
\end{theorem}
\begin{proof}
As the proof is also a motivation for the more general abstract setting of the last part, we postpone the proof to the main result of section \ref{10} (see theorem \ref{th13}).
\end{proof}\qed
\begin{flushleft}
At the cost of much heavier computations, wee can also obtain similar estimates for {\it mixed q-Gaussian algebras} and {\it q-deformed Araki-Woods algebras}. We left the details to the reader.
\end{flushleft}
\section{Chaos of a quantum Markov operator}\label{10}
\begin{flushleft}
In this section, we will give a possible interpretation of structures which can exhibit a "{\it free} fourth moment phenomenon": The non commutative fourth moment diffusive structures. We will construct a free Stein kernel relative to the standard semicircular potential via the non commutative {\it carré du champ} and the pseudo-inverse of the $L^2$ generator, which will provide an free analog of the construction of Stein kernels on Markov chaoses by Ledoux, Nourdin and Peccati (see section 5.1 in \cite{LNP}). As the main goal of previous section was to construct a new free Stein kernel on the Wigner space which is such that the free Stein discrepancy is controlled by the {\it fourth free cumulants}, it might require much stronger assumptions on the derivation $\delta$ such as coassociativity of the directional derivatives to deal with the full classes of self-adjoint {\it chaotic} random variables.
\newline
In fact, one cannot hope that every quantum Markov semigroup (even with an appropriate notion of chaotic decomposition) will lead to a fourth moment phenomenon (towards the semicircular distribution). Indeed, one have to assume (and this is the main assumption) that the derivation is valued into (modulo the operator valued setting of free Malliavin calculus) into a direct sum of coarse bimodules, this property as we will in the sequel is essential, as the Schwinger-Dyson equation which characterize the semicircular distribution involves free difference quotient valued into this coarse correspondence). It is also well know that there exists uncountably many non isomorphic WOT separable $II_1$-factors, Mc-Duff \cite{MC} (the free groups factors and their deformations are thus a discrete part in the continuum set of $II_1$ factors) and no separable univerval one (containing a copy of all others, Ozawa \cite{NO}), in particular the semicircular distribution is the fundamental distribution under "freeness". More precisely, from various papers (see the first main contribution of Cipriani and Sauvageot \cite{CipS}) that the Laplacian, which the $L^2$ generator of the Dirichlet form associated ($1-1$ correspondence) to a quantum Markov semigroup, supposed to be completely Markovian in a tracial non commutative probability spaces, can be written as the square of a derivation valued in some Hilbert $\mathcal{M}$-bimodule $\mathcal{H}$, that is a Hilbert space $\mathcal{H}$ equipped with the two commuting actions, $\pi: \mathcal{M}\rightarrow \mathcal{B}(\mathcal{H})$ and $\pi^{op}:\mathcal{M}^{op}\rightarrow \mathcal{B}(\mathcal{H})$ denoted left and right respectively, and we will denote $x\varepsilon y$ the vector $\pi(x)\pi^{op}(y)\varepsilon$.

\end{flushleft}
\subsection{The abstract setting of diffusive non commutative fourth moment structures}
Let's suppose that $\mathcal{M}$ is a finite von Neumann algebra, equipped with a faithful normal tracial state $\tau$. Let also assume that in the following $\mathcal{H}$ is the complexification of some real separable Hilbert space  $\mathcal{K}$. 
\begin{flushleft}
Let $\delta: D(\delta)\rightarrow L^2(\mathcal{M})\otimes \mathcal{H}\otimes L^2(\mathcal{M}^{op})$ be a real closable derivation, the bimodule structure being given by the usual multiplication on the left leg and right multiplication on the right leg, which satisfies the "{\it real}" property: $\langle\delta(x),y\delta(z)\rangle=\langle \delta(z^*)y^*,\delta(x^*)\rangle$. We will also assume that $D(\delta)$ is weakly dense $*$-subalgebra of $\mathcal{M}$.
\newline
We can assume without for sake of simplicity, respectively that $\mathcal{K}=L^2_{\mathbb{R}}(\mathbb{R_+})$ or $\ell^2(\mathbb{N})$ in the infinite dimensional case, and in the finite dimensional case $\mathcal{K}=\mathbb{R}^n$, and see respectively $\delta$ as follows:
\begin{eqnarray}
\delta: D(\delta)\rightarrow L^2(\mathbb{R}_+,L^2(\mathcal{M})\bar{\otimes} L^2(\mathcal{M}^{op}))\nonumber\\
x\mapsto \delta(x)=(\delta_t (x))_{t\geq 0},
\end{eqnarray}
which is seen as a biprocess $\left\{\delta_t(x), t\geq 0\right\}\in \mathcal{B}_2$, (valued for almost all $t\geq 0$ into the coarse correspondence).
\newline
Or, in the second case, when $\mathcal{K}=\ell^2(\mathbb{N})$, 
\begin{eqnarray}
\delta:D(\delta)\rightarrow (L^2(\mathcal{M})\bar{\otimes}L^2(\mathcal{M}^{op}))^{\oplus\infty},
\end{eqnarray}

and in the finite dimensional case:

\begin{eqnarray}
\delta:D(\delta)\rightarrow (L^2(\mathcal{M})\bar{\otimes}L^2(\mathcal{M}^{op}))^{\oplus N},
\end{eqnarray}
\end{flushleft}

\begin{remark}
Note that we adopt the standards notations of Kemp and al. \cite{KNPS}, this means that the involution is given by $(A\otimes B)^*=A^*\otimes B^*$, contrary to the usual conventions in \cite{BS}. The real assumptions is an important fact in the infinite dimensional setting as only the small subspace of \it{fully-symmetrics} multiple Wigner integrals verify this condition and are these ones one can control the discrepancy constructed via the usual integration by parts involving the inverse $\Delta^{-1}$ in terms of fourth free cumulants. In the finite dimensional case, the derivative "\it{looks like}" a free difference quotient which is always a real derivation.
\end{remark}
\begin{flushleft}
We let now $\bar{\delta}$ the $L^2$ extension of the operator $\delta$ and we will omit to denote it when the context is clear. \end{flushleft}
We denote now $\Delta=\delta^*\bar{\delta}$, the following operator which is the associated generator of a completely Dirichlet form: $\mathcal{E}(x)=\langle \delta(x),\delta(x)\rangle$, which means that $\Delta\otimes I_n$ is also the generator of a Dirichlet form on $M_n(\mathcal{M})$.
\begin{flushleft}
We denote then $\phi_t:=e^{-t\Delta}$ the corresponding semigroup of contractions on $L^2(\mathcal{M})$ generated by $-\Delta$, which is tracially symmetric ($\tau$ symmetric) and trace preserving.
\end{flushleft}
\begin{remark}
In the following, we denote (especially to deal with $\mathcal{K}=L^2_{\mathbb{R}}(\mathbb{R}_+))$ for $u=g.a\otimes b$, $a,b,c,d\in L^2(\mathcal{M})$ and $g,h\in \mathcal{H}$, the linear extension of the pairing:
\begin{equation}
    \langle u,h\rangle_{\mathcal{H}}=a\otimes b.\langle g,h\rangle_{\mathcal{H}}
\end{equation}
\end{remark}
\begin{flushleft}
Consider know the associated operator (in the infinite dimensional setting) "\it{carré du champ}", for $x,y \in dom(\delta)$ :
\begin{equation}
    \Gamma(x,y)=\langle \delta(x),\delta(y)\rangle_{\mathcal{H}},
\end{equation}
and in the finite dimensional case, where $\delta=(\delta_1,\ldots,\delta_n)$:
\begin{equation}
    \Gamma(x,y)=\sum_{k=1}^n\delta_i(x)\sharp(\delta_i(y))^*,
\end{equation}
Which is such that for $x,y$ self-adjoints, one has:
\begin{equation}
    \tau\otimes\tau(\Gamma(x,y))=\tau(x\Delta y)=\langle \delta(x),\delta(y)\rangle,
\end{equation}
if we denote also the linear extension of the pairing $\langle  A\otimes h\otimes B,C\otimes \otimes g\otimes D\rangle_{\mathcal{H}}$  for $A,B,C,D\in L^2(\mathcal{M})$, and $g,h\mathcal{H}$. We can also express:
\begin{equation}
    \tau\otimes\tau(\langle \delta(f),\delta(g)\rangle_{\mathcal{H}})=\langle\delta(f),\delta(g)\rangle,
\end{equation}

\begin{flushleft}
Consider for $F=(F_1,\ldots,F_n)$ self-adjoint, the associated vector $\Delta^{-1}(F)=(\Delta^{-1}F_1,\ldots,\Delta^{-1}F_n$) where $\Delta^{-1}$ is the pseudo-inverse of $\Delta$ (since it is self-adjoint and positive) which acts centered non commutative random variables through the following relation between these two operators, that is for $F\in L^2(\mathcal{M})$:
\begin{equation}
    \Delta\Delta^{-1}(F)=F-\tau(F)
\end{equation}
\end{flushleft}
As we supposed that $\delta$ is a derivation, we are in a diffusive context, $\Gamma$ satisfies a chain rule, for $X=x_1,\ldots,x_n\in dom(\delta)$ and $Y=y_1,\ldots,y_n\in dom(\delta)$ and $P,Q\in \mathbb{P}$, such that $P(x_1,\ldots,x_n), Q(y_1,\ldots,y_n)\in dom(\delta)$, and:
\begin{equation}
    \Gamma(P(x_1,\ldots,x_n),Q(y_1,\ldots,y_n))=\sum_{i,j=1}^n\partial_i{P(X)}\sharp\Gamma(x_i,y_j)\sharp(\partial_j{Q}(Y))^*
\end{equation}
\begin{definition}
A non-commutative fourth moment diffusive structures will be a triple $(\mathcal{M},\delta,\Delta)$ with $\mathcal{M}$ a finite von Neumann algebra, $\delta$ a derivation with the properties mentioned before, and $\Delta$ the associated generator with the additional property:
\begin{enumerate}
\item $\Delta$ has a pure point spectrum (and positive by assumptions since it is selfadjoint), that is there exists a increasing sequence $(\lambda_k)_{k\geq 0}$ with $\lambda_1<\lambda_2<...$, (for sake of simplicity we can assume that $sp(\Delta)=\N$).
\newline
This implies that $\Delta$ diagonalize $L^2(\mathcal{M})$:
\begin{equation}
        L^2(\mathcal{M})=\bigoplus_{k=0}^{+\infty}Ker(\Delta-\lambda_kId)
\end{equation}
\end{enumerate}
\end{definition}
\begin{theorem}
Let's $F=(F_1,\ldots,F_n)\in dom(\delta)$ be a n-tuple of centered self-adjoint non-commutative random variables. Then, $\left(\Gamma(\Delta^{-1}F_i,F_j)\right)_{i,j=1}^n$ is a free Stein kernel with respect to the standard semicircular potential :
\end{theorem}
\begin{proof}\label{th13}
Let's take $P=(P_1,\ldots,P_n)\in \mathbb{P}^n$ and compute:
\begin{eqnarray}
\langle\Gamma(\Delta^{-1}F,F),[JP](F)\rangle_{HS}&=&\sum_{i,j=1}^n\langle \Gamma(\Delta^{-1}F_i,F_j),[\partial_iP_j](F)\rangle_{L^2(\mathcal{M})\bar{\otimes} L^2(\mathcal{M}^{op})}\nonumber\\
&=&\sum_{i=1}^n\langle \Gamma(\Delta^{-1}F_i,P_i(F_1,\ldots,F_n)),(1\otimes 1)\rangle_{L^2(\mathcal{M})\bar{\otimes} L^2(\mathcal{M}^{op})}\nonumber\\
&=&\langle \Delta\Delta^{-1}F,P(F_1,\ldots,F_n)\rangle_2
\nonumber\\
&=&\langle F,P(F)\rangle_2\nonumber
\end{eqnarray}
\end{proof}
\begin{proposition}
Suppose now, that $F=(F_1,\ldots,F_n)\in  D(\delta)$ such that each $F_i$ is an eigenvalue of $\Delta$, that is it exists $\lambda_i>0$, such that $\Delta F_i=\lambda_iF_i$ and such that $\Gamma(F_i,F_j)$ belongs to $L^2(\mathcal{M}\otimes \mathcal{M}^{op})$. We then deduce that :
\begin{equation}
    \left(\Gamma(\frac{1}{\lambda_i}F_i,F_j)\right)_{i,j=1}^n
\end{equation} is a free Stein kernel with respect to the standard semicircular potential.
\end{proposition}
As, on the Wigner space, we can also compute explicitly a free Stein kernel with respect to every free Gibbs state (the proof is a straightforward modification if the precious theorem).
\begin{theorem}
Let's $F_1,\ldots,F_n \in \bigoplus_{k=0}^d Ker(\Delta-\lambda_kId)$ self-adjoint with "$d$" a bounded integer, and $V$ a polynomial potential or a formal power series. Assume that $\tau([DV](F))=(0,\ldots,0)$, then:
\begin{equation}
    \left(\Gamma(\Delta^{-1}([D_iV](F)),F_j)\right)_{i,j=1}^n
\end{equation}
is a free Stein kernel with respect to the potential $V$.
\end{theorem}
It implies that one can also recover the existence of free Stein kernels with respect to every potential on these structures and provide analogues of Fathi, Cébron and Mai results \cite{FCM} about existence.
\begin{flushleft}
This is in fact the starting point of the famous and recent investigations of Ledoux in the striking contribution \cite{ML} which proved the "fourth moment theorem" for some eigenfunctions of a Markov operator (the notion of chaotic random variables in this context will be a weaker one compared to the original definition in his paper, and is very similar to the one used by Azmoodeh, Campese and Poly in the paper \cite{ACP}). That is, under suitable assumptions over the generator $\Delta$, and over the eigenfunctions which are supposed to be "chaotic", one can obtain a quantitative "fourth moment theorem" for theses eigenfunctions. More interestingly, the product formula on the Wiener space, which is a crucial tool in  original proof of the quantitative "fourth moment theorem" by Nourdin and Peccati (see e.g \cite{NP-book} section 5.5.2) can be subtly avoided (and is no longer necessary as soon as the square of this eigenfunction (of order $p$) can be expanded into a sum over the eigenspaces of $ker(\Delta-kId)$ for $0\leq k\leq 2p$):
\end{flushleft}

\end{flushleft}
\begin{remark}
As we assume that the generator $\Delta$ is the square of a derivations valued in these particular coarse bimodule (build in particular for the infinite dimensional case), such that for almost all $t\geq0$ $\delta_t :D(\delta)\rightarrow L^2(\mathcal{M})\bar{\otimes} L^2(\mathcal{M}^{op})$ or equivalently that the directional derivatives:
\begin{equation}
\delta^h(x)=\langle \delta(x),h\rangle_{\mathcal{H}}:=\int _{\mathbb{R}_+}\delta_t(x)\overline{h(t)}dt
\end{equation}
are valued into this coarse correspondence. It implies directly (and doesn't have to be supposed) that the chain rule is always ensured.
\end{remark}
\begin{definition}
We will call $X\in L^2(\mathcal{M})$ (self-adjoint) an eigenfunction of $\Delta$ a "chaos" eigenfunction of order "$p$" if :
\begin{equation}
    X^2\in \bigoplus_{k=0}^{2p}ker(\Delta-\lambda_kId)
\end{equation}
\end{definition}
\begin{flushleft}
We leave here the following conjecture, which is the main interest for building this abstract context and which can be seen as a free counterpart of Ledoux's result. \cite{ML}.
\end{flushleft}
\begin{conjecture}\label{c1}
Suppose that $X$ is self-adjoint, chaotic (of order $q\geq 1$), "fully-symmetric" which means in this context that for any $t\geq 0$, $\delta_t(x)^{*}=\delta_t(x)$, and normalized $\tau(X^2)=1$.
Then, there exists $C_q$ depending only on "$q$", such that we have the following bound which implies an abstract "fourth moment theorem", since $\Sigma^*(F|V_{I_n})\leq C_q\sqrt{\tau(F^4)-2}$:
\begin{equation}
    \lVert \Gamma(x)-q.1\otimes1\rVert_{L^2(\mathcal{M})\bar{\otimes}L^2(\mathcal{M}^{op})}\leq C_q\sqrt{\tau(x^4)-2}
\end{equation}
\end{conjecture}
\begin{remark}
The convergence result might require maybe stronger assumptions:
\begin{enumerate}
    \item for all integers $q\geq 1$, for almost all $t\geq 0$,
    \begin{equation}
        \delta_t : Ker(\Delta-qId) \rightarrow \left(\bigoplus_{k=0}^q Ker(\Delta-kId)\right)\otimes  \left(\bigoplus_{k=0}^q Ker(\Delta-jId)\right)
    \end{equation}
    \item Or even the more restrictive condition:\begin{equation}
        \delta_t : Ker(\Delta-qId) \rightarrow \bigoplus_{j+k=q} Ker(\Delta-jId)\otimes  Ker(\Delta-kId)
    \end{equation}
\end{enumerate}
\end{remark}

\section{Open problems}
We would like to finish this paper with some open questions which are the following ones:
\begin{enumerate}
\item Is it possible to obtain the {\it HSI} inequality for all the class of self-adjoint convex potentials, in particular does the convexity assumption : $\mathcal{J}DV\geq c(1\otimes 1)\otimes I_n$ implies a good control of the Free Fisher information along the flow of the free Langevin diffusion with sel-adjoint convex potential as drift by a quantity involving the free 
Stein discrepancy (one knows for now that the second convexity assumption implies a good control to deduce {\it LSI}, by the exponential decay of the norm of the semigroup generated by  $Q=\left(\partial_jD_iV\right)_{i,j=1}^n$ and it's not clear for now if we don't have to suppose an additional unknown condition on the potential. Precisely, we used the equality in distribution of a free Langevin diffusion (Ornstein-Uhlenbeck case with linear drift) with cyclic derivatives as drift, to get to the result. Unfortunately, this won't be possible for the general case and seems rather challenging as the proof for non-homothetic drift is not straightforward, and so we would have to use Dabrowski formulas for conjugate variables and 
link them with a free Stein kernel.
\item What can be said in the case of the non-invertibility of the covariance matrix. In our approach, the positive definite aspect plays an important role, and we think it might be possible to obtain a bound for a weaker distance (such distances have not been yet defined). Indeed, we are still unable to define properly the distribution $\mu_{x_1,\ldots,x_n}$ as a non-commutative probability measure, and we have for now no idea of what it means.
\item It is well know from Ledoux, Azmoodeh, Campese, Poly or Bourguin, that the fourth moment phenomenon (even for more general distribution which belongs to the Pearson class) occurs in the abstract setting of Chaos of a Markov operator, especially estimates of the following type
$Var(\Gamma(F))\leq C(\E(F^4)-3\E(F^2)^2)$ lead to the fourth moment theorem (toward the standard Gaussian) for chaotic random variable: that is an eigenvalue of order $q$ of the opposite of the Markov generators $-L$ (which is suppose to simplify to have a pure point spectrum $sp(-L)=\N$), and which satisfies some conditions (for example the square of F can be expressed as finite sum over the eigenspaces of $Ker(L+k Id)$ for $1\leq k\leq 2q$. Unfortunately, we are unable for now to deduce such a result in our abstract setting in a free probabilistic context, even the heuristic given by Ledoux in the classical context, that is, suppose that $\E(F^2)=1$ and satisfies the previous assumptions, then:
\begin{eqnarray}
    q\E(F^4)&=&\E(F(-L)F)\nonumber\\
&=&3\E(F^2\Gamma(F))
\end{eqnarray}
and recall, since $F$ is an eigenfunction of $-L$ (of order $q$), this implies that $\E(\Gamma(F))=q$ (this is the farthest that we can go in our abstract setting):
\begin{equation}
    q\left(\frac{1}{3}\E(F^4)-1\right)=\E(F^2(\Gamma(F)-q))
\end{equation}
and this means that if $\tau(F^4)\sim 3$, then necessarily $\Gamma(F)\sim q$.
\newline
The heuristic in the free case seems much more intriguing since the same type of computations (and under the same hypothesis, that is $\Delta(F)=qF$) leads to:
\begin{eqnarray}
q\tau(F^4)&=&\tau\otimes \tau\left((F^2\otimes 1+F\otimes F+1\otimes F^2)\Gamma(F)\right)\nonumber\\
&\sim& \tau\otimes \tau \left((F^2\otimes 1 +1\otimes F^2).\Gamma(F)\right)
\end{eqnarray}
And thus, assuming that the central term if close to $0$ (since $\tau(F)=0$), it gives that:
\begin{eqnarray}
q\left(\tau(F^4)-2\right)\sim \tau\otimes \tau \left((F^2\otimes 1 +1\otimes F^2).(\Gamma(F)-q.1\otimes 1)\right)
\end{eqnarray}
which should be compared with the relation given in lemma \ref{lma5}, which is the starting point of the investigations about "{\it fourth moment theorem}" on the Wigner space.
\newline
Moreover, on the Wigner space, the constant $C_q$ of conjecture \ref{c1} obtained by Bourguin and Campese \cite{BC} on the Wigner space is complicated (contrary to the usual Gaussian case) and it is obtained through a very sophisticated combinatorial analysis (which involves deep estimates obtained via the product formula for Wigner integrals).
\item In light of these result, we also leave here the following conjecture: suppose that we are given a self-adjoint convex potential $V$ (polynomial to simplify) and the unique free Gibbs state associated, and a vector of self-adjoints multiple Wigner integrals $F$. Suppose that $\Gamma_V(F)$ is invertible (for example in the Banach algebra considered in \cite{GS}), then the convergences of $X$ toward the free Gibbs state $\tau_V$ is equivalent to the convergence of some moments. This idea is motivated by the conjecture of Voiculesu which states that every analytic perturbation of the semicircircular potential generates again the free groups factors: $L(\mathbb{F}_n)$, which have been proved by Guionnet et Shlyakhtenko in \cite{GS} for small (and convex preserving) perturbations and then extended by Dabrowski, Guionnet and Shlyakthenko in \cite{YGS}, and we think that this phenomenon should be traduced in terms of convergence of multiple Wigner integrals.
\item Ledoux, Nourdin and Peccati have been able to obtain the "{\it HSI}" inequality for families of invariant measures of second-order differential operator by the Triple Markov approach and Gamma calculus, especially by criterion over $\Gamma_2$: the Bakry-Emery criterion (which provides {\it LSI} inequality in the commutative case), and $\Gamma_3$ (which combined with the previous criterion over $\Gamma_2$ leads to the {\it HSI} inequality). Is it possible in our setting to deduce {\it LSI}, {\it HSI} and {\it WSH} inequality via an appropriate non commutative gradient operator ? To our best knowledge, a proof of Log-Sobolev via a non commutative Bakry-Emery condition where the generator is the one of a free Langevin diffusion with uniformly convex drift, is not known. If it could be proved, it would certainly implies that all these inequalities are true in full generality.
\end{enumerate}
\section*{Acknowledgements}
The author would like to thank deeply Pr. Guillaume Cébron and Pr. Solesne Bourguin for their numerous advice and discussions on this topic, his former PhD supervisor Pr. Ciprian Tudor and Dr. Obayda Assaad for their invaluable comments to improve the paper. 
\section{Funding}
This work is partially supported by Labex CEMPI (ANR-11- LABX-0007-01). The last parts of the paper were written at the Department of Statistics of The Chinese University of Hong-Kong, Hong-Kong SAR, China where the author is currently.

\end{document}